\documentclass[12pt]{article}
\usepackage[margin =1in]{geometry}
\usepackage{amssymb,amsthm}
\usepackage{amsmath}
\usepackage{enumerate}
\usepackage{hyperref}
\usepackage{booktabs}
\usepackage{cite}
\usepackage{color}
\usepackage{xcolor}
\definecolor{blue}{RGB}{0,0,0} % uncomment to switch on blue.
\numberwithin{equation}{section}
\newcommand{\cV}{\mathcal{V}}

\newcommand{\cU}{\mathcal{U}}

\newcommand{\norm}[1] {\left \| #1 \right \|}
\newcommand{\inclu}[0] {\ar@{^{(}->}}

\newcommand{\gph}{{\rm gph}\,}

\newcommand{\dist}{{\rm dist}}

\newcommand{\cA}{\mathcal{A}}

\newcommand{\EE}{\mathbb{E}}

\newcommand{\lip}{\mathrm{lip}}

\newcommand{\RR}{\mathbb{R}}

\newcommand{\minimize}{\mathrm{minimize}}

\newcommand{\cX}{\mathcal{X}}

\newcommand{\cM}{\mathcal{M}}

\newcommand{\range}{\mathrm{range}}

\newcommand{\bdry}{\mathrm{bdry}\,}

\newcommand{\conv}{\mathrm{conv}}

\newcommand{\abs}[1]{\left| #1 \right|}

\newcommand{\argmin}{\operatornamewithlimits{argmin}}

\newcommand{\NN}{\mathbb{N}}

\newcommand{\dotp}[1]{\left\langle #1\right\rangle}

\newtheorem{thm}{Theorem}[section]

\newtheorem{definition}[thm]{Definition}
\newtheorem{proposition}[thm]{Proposition}
\newtheorem{lem}[thm]{Lemma}
\newtheorem{cor}[thm]{Corollary}

\newtheorem{assumption}{Assumption}

\newtheorem{remark}{Remark}

\newcommand{\paren}[1]{ \left( #1 \right) }
\theoremstyle{remark}

\usepackage{mathtools}

\numberwithin{equation}{section}
\usepackage{algorithm}
\usepackage{enumitem}
\usepackage{algpseudocode}
\allowdisplaybreaks
\newenvironment{claim}[1]{\par\noindent\underline{Claim:}\space#1}{}
\newenvironment{claimproof}[1]{\par\noindent\underline{Proof:}\space#1}{\hfill $\blacksquare$}

\newcommand{\scc}{\gamma}
\newcommand{\closedball}{\overline B}
\newcommand{\ver}{C_{(a)}}

\newcommand{\lipg}{\beta}
\newcommand{\lipf}{L}
\newcommand{\sharpc}{\mu}

\newcommand{\shape}{C_{\cM}}

\newcommand{\eps}{\epsilon}

\newcommand{\innerg}{t}

\newcommand{\opt}{f^\ast}

\newcommand{\algname}{\mathtt{NTDescent}} 
\newcommand{\T}{\top}

\newcommand{\tgoldstein}{\mathtt{TDescent}}
\newcommand{\ngoldstein}{\mathtt{NDescent}}
\newcommand{\linesearch}{\mathtt{linesearch}}
\newcommand{\polyak}{\texttt{PolyakSGM}}

\newcommand{\tangent}{T_{\cM}}
\newcommand{\normal}{N_{\cM}}

\newcommand{\minnorm}{\mathtt{MinNorm}}

\newcommand{\Lin}{\mathrm{lin}}
\newcommand{\ri}{\mathrm{ri}\,}

\newcommand{\gridsize}{G}
\newcommand{\deltaA}{\delta_{\mathrm{A}}}
\newcommand{\deltaGKL}{\delta_{\mathrm{GI}}}
\newcommand{\deltagrid}{\delta_{\mathrm{Grid}}}

\newcommand{\deltaND}{\delta_{\mathtt{ND}}}
\newcommand{\deltaLS}{\delta_{\mathtt{LS}}}
\newcommand{\deltaNCVX}{\delta_{\mathtt{NTD}}}

\newcommand{\diam}{\mathrm{diam}}
\newcommand{\cfive}{\dtwo}
\newcommand{\csix}{C_4}
\newcommand{\cseven}{C_5}
\newcommand{\cnine}{C_6}
\newcommand{\slb}{s_{\mathrm{lb}}}

\newcommand{\sscale}{c_0}
\newcommand{\widthforbig}{.37}
\newcommand{\cfour}{C_3}
\newcommand{\dtwo}{D_2}
\newcommand{\done}{D_1}
\newcommand{\ceight}{\done^{-1}} % previously C_6
\newcommand{\cone}{C_1}
\newcommand{\ctwo}{C_2}

\usepackage{appendix}

\usepackage{caption}
\usepackage{subcaption}

\begin{document}
	
\title{A {\color{blue} local} nearly linearly convergent first-order method for nonsmooth functions with quadratic growth}

	\author{Damek Davis\thanks{School of Operations Research and Information Engineering, Cornell University. Ithaca, NY 14850, USA;
\texttt{people.orie.cornell.edu/dsd95/}. Research of Davis supported by an Alfred P. Sloan research fellowship and NSF DMS award 2047637.}  \qquad Liwei Jiang\thanks{School of Operations Research and Information Engineering, Cornell University. Ithaca, NY 14850, USA;
	\texttt{orie.cornell.edu/research/grad-students/liwei-jiang}}}	
	\date{}
	\maketitle
%\tableofcontents

\begin{abstract}
Classical results show that gradient descent converges linearly to minimizers of smooth strongly convex functions. A natural question is whether there exists a locally nearly linearly convergent method for nonsmooth functions with quadratic growth. This work designs such a method for a wide class of nonsmooth and nonconvex locally Lipschitz functions, including max-of-smooth, Shapiro's decomposable class, and generic semialgebraic functions. The algorithm is parameter-free and derives from Goldstein's conceptual subgradient method.
\end{abstract}

%%%%%% Intro %%%%%%%

% !TEX root = template.tex

\section{Introduction}

Slow {sublinear} convergence of first-order methods in nonsmooth optimization is often illustrated with the following simple strongly convex function:
\begin{align}\label{eq:lbexample}
f(x) = \max_{1 \leq i \leq m} x_i + \frac{1}{2}\|x\|^2 \qquad \text{for some $m \leq d$ and all $x \in \RR^d$}.
\end{align}
For example, consider the subgradient method applied to $f$, which generates iterates $x_k.$
Since $f$ is strongly convex, classical results dictate that $f(x_k) - \inf f = O(k^{-1})$. 
On the other hand, under proper initialization and an adversarial first-order oracle, there is a matching lower bound for the first $m$ iterations: $f(x_k) - \inf f \geq (2k)^{-1} $ for all $k \leq m$; see~\cite{intro_lect,bubeck2015convex}.
Beyond the subgradient method, the lower bound also holds for any algorithm  whose $k$th iterate lies within the linear span of the initial iterate and past $k-1$ computed subgradients.
Thus, one must make more than $m$ first-order \emph{oracle calls} to $f$, i.e., function and subgradient evaluations, before possibly seeing improved convergence behavior.

While such methods make little progress when $k \leq m$, 
this behavior may or may not continue for $k \gg m$. 
On one extreme, the subgradient method continues to converge slowly even when equipped with the popular Polyak stepsize (\texttt{PolyakSGM})~\cite{Polyak69}; \textcolor{blue}{see dashed lines in Figure~\ref{fig: lowerbound}}.
On the opposite extreme, more sophisticated algorithms such as the center of gravity method or the ellipsoid method converge linearly, but their complexity scales with the dimension of the problem, a necessary consequence of the linear rate of convergence; see the discussion in~\cite[Chapter 2]{bubeck2015convex}.

  \begin{figure}[H]
 	\centering
 	\begin{subfigure}[b]{0.45\textwidth}
 		\centering
 		\includegraphics[width=\textwidth]{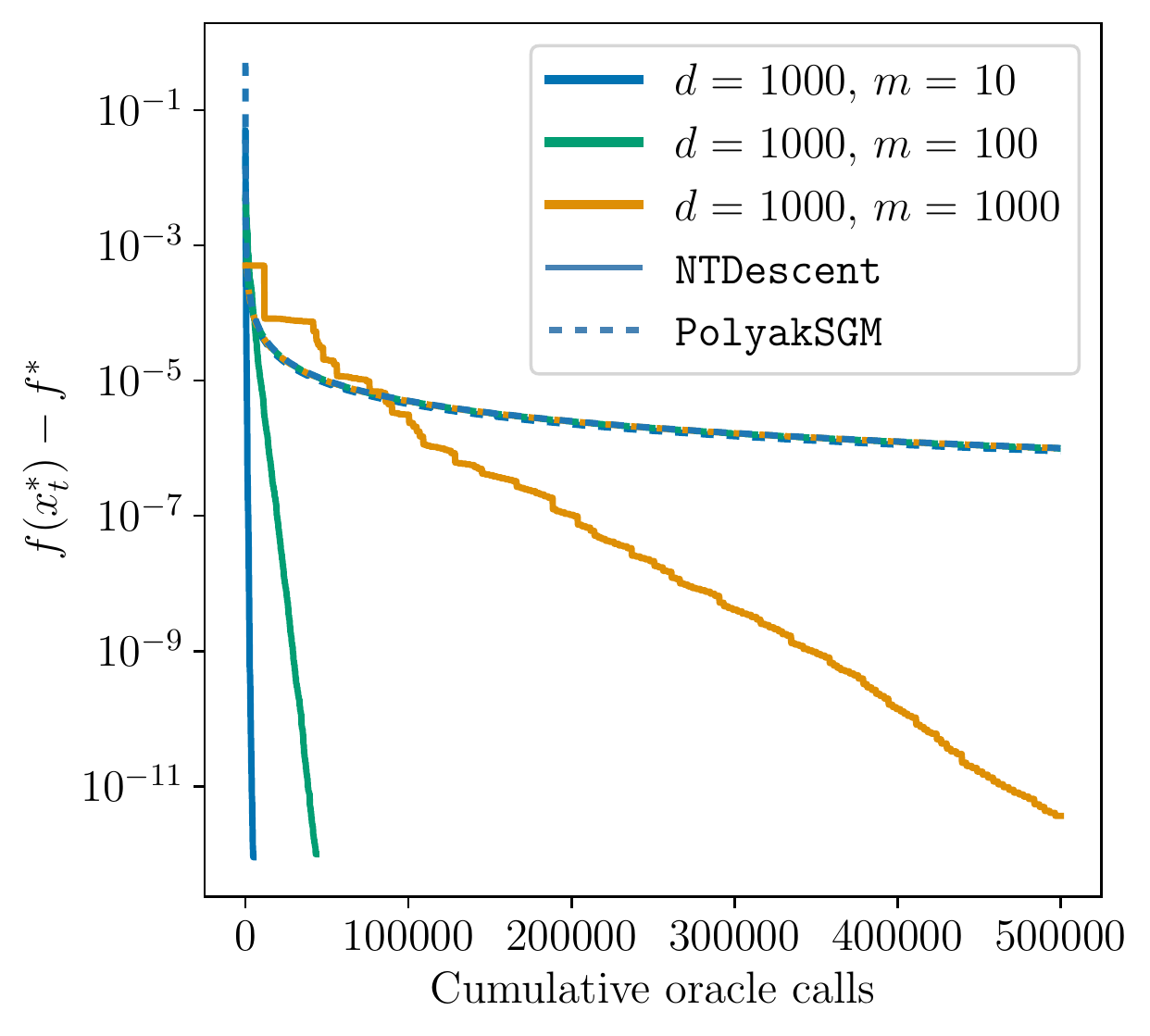}
 		\caption{}
 		\label{fig:lba}
 	\end{subfigure}
 	\hfill
 	\begin{subfigure}[b]{0.45\textwidth}
 		\centering%\vspace{-10pt}
 		\includegraphics[width=1\textwidth]{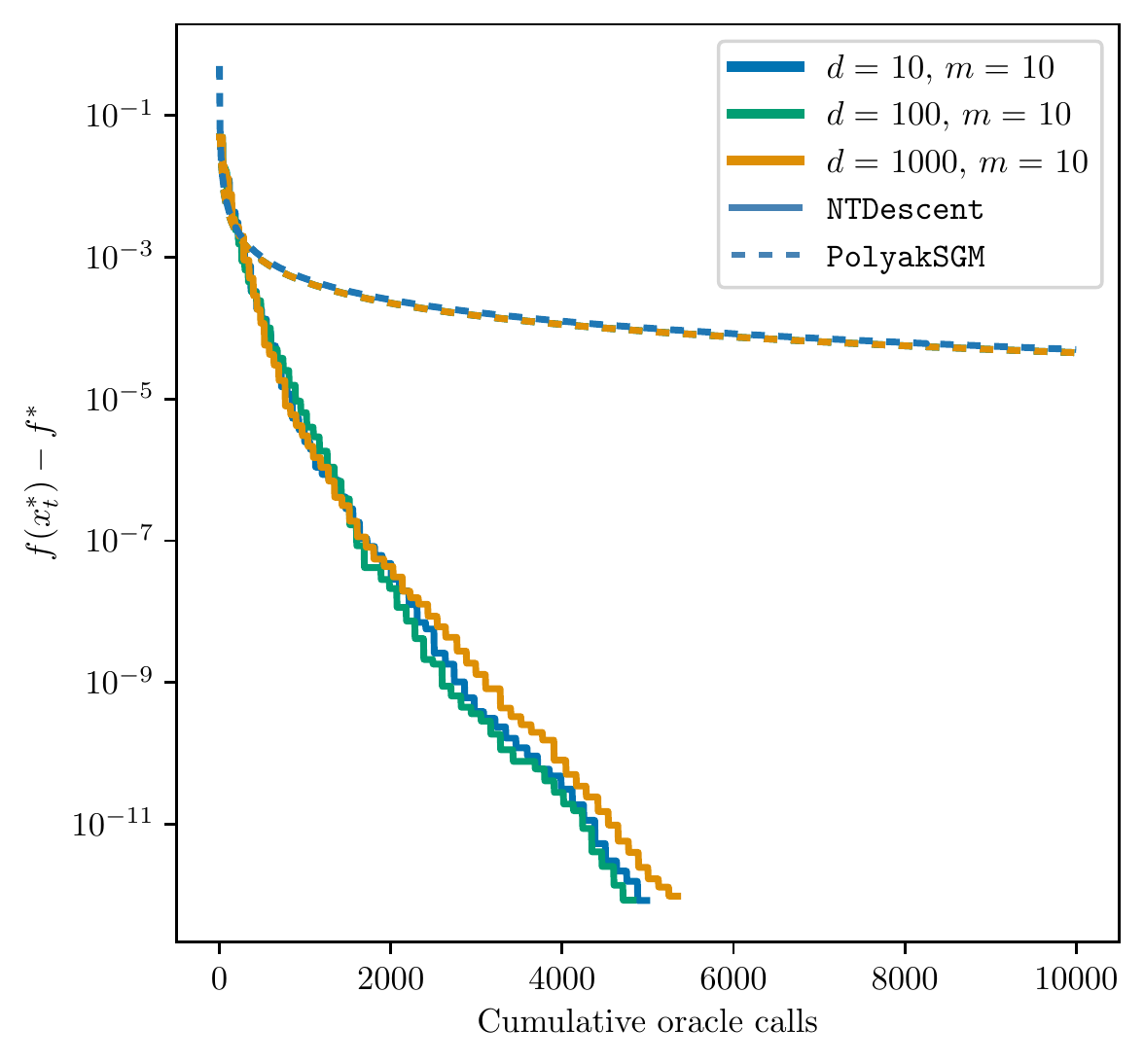}
 		\caption{}
 		\label{fig:lbb}
 	\end{subfigure}
 	\caption{Comparison of $\algname$ with $\polyak$ on~\eqref{eq:lbexample}. Left: we fix $d$ and vary $m$; Right: we fix $m$ and vary $d$. For both algorithms, the value $f(x_t^*)$ denotes the best function value seen after $t$ oracle evaluations.}\label{fig: lowerbound} 
 \end{figure}

A natural question is whether there exists a first-order method whose behavior lies in between these two extremes, at least for nonsmooth functions $f$ satisfying regularity conditions at local minimizers. 
Regularity conditions often take the form of growth -- linear or quadratic -- away from minimizers. 
Well-known results show that subgradient methods converge linearly on nonsmooth functions with linear (also called \emph{sharp}) growth~\cite{Polyak69}. 
On the other hand, in smooth convex optimization, quadratic growth entails linear convergence of gradient methods.
However, to the best of our knowledge, no parallel result for nonsmooth functions with quadratic growth exists. 
Thus, in this work, we ask
\begin{quote}
is there a locally nearly linearly convergent method for nonsmooth functions with quadratic growth whose rate of convergence and region of rapid local convergence solely depends on $f$?
\end{quote}
Let us explain the qualifiers ``nearly" and ``solely depends on $f$."
First, the qualifier ``nearly" signifies that the method locally achieves a function gap of size $\varepsilon$ using at most, say, $O(C_f \log^3(1/\varepsilon))$ first-order oracle evaluations of $f$, where $C_f$ depends on $f$.
Second, the qualifier ``solely depends on the function," signifies  that $C_f$ and the size of the region of local convergence do not depend on the dimension of the problem, but instead depend only on the function $f$ through intrinsic quantities, such as Lipschitz and quadratic growth constants.

In this work, we positively answer the above question for a class of nonsmooth optimization problems with quadratic growth.
The method we develop is called \emph{Normal Tangent Descent} ($\algname$). 
We formally describe $\algname$ in Section~\ref{sec:algdescription}.
For now, we illustrate the performance of $\algname$ on $f$ from~\eqref{eq:lbexample} in Figure~\ref{fig: lowerbound}.
In both plots, we see $\algname$ improves on the performance of $\polyak$, measured in terms of {oracle calls}. 
This is a fair basis for comparison since both $\polyak$ and $\algname$ perform a similar amount of computation per oracle call.
Figure~\ref{fig:lbb} also shows that the performance of $\algname$ is dimension independent.
We highlight that this performance was achieved without any tuning of parameters for $\algname$. 
Indeed, our main theoretical guarantees for $\algname$ (Theorem~\ref{thm:maintheoremsemi}) do not require the user to set any parameters.

The problem class on which $\algname$ succeeds consists of locally Lipschitz nonsmooth functions with quadratic growth and  a certain \emph{smooth substructure} at local minimizers. 
Importantly, we do not assume the problems under consideration are convex, though convexity entails improved guarantees. 
Two example classes with such smooth substructure include (i) ``generic" semialgebraic functions and (ii) properly $C^p$ decomposable loss functions satisfying strict complementarity and quadratic growth conditions~\cite{shapiroreducible}. 
A semialgebraic function is one whose graph is the finite union of intersections of polynomial inequalities. 
Semialgebraic functions (more generally \emph{tame}~\cite{tame_opt} functions) 
model most problems of interest in applications.
If $f$ is semialgebraic, for a full Lebesgue measure set of $w \in \RR^d$, we will use show that the tilted function $f_w \colon x \mapsto f(x) + w^\T x$ has quadratic growth and the desired smooth substructure at each local minimizer, explaining the qualifier ``generic." 
We mention that this fact essentially follows from combining results of~\cite{drusvyatskiy2016generic,davis2021subgradient}. 
On the other hand, a properly $C^p$ decomposable function is one that decomposes near local minimizers as a composition of a positively homogeneous convex function with a smooth mapping that maps the minimizer to the origin.
Decomposable functions appear often in practice, e.g., in eigenvalue and data fitting problems. 
An important subclass of decomposable functions consists of so-called ``max-of-smooth" functions, which are the maximum of finitely many smooth functions that satisfy certain regularity conditions at minimizers, e.g., $f$ in~\eqref{eq:lbexample}.

The precise smooth substructure used in this work was recently identified in~\cite{davis2021subgradient}, where it was shown to be available in decomposable and generic semialgebraic problems.
Since it is available in many problems of interest, throughout this introduction we call this the combination of quadratic growth and smooth substructure \emph{typical structure} and call functions possessing this combined structure \emph{typical.}
We present the formal structure in Section~\ref{sec:mainassumption}. 
At the heart of this structure is a distinguished smooth manifold $\cM$ -- called the \emph{active manifold} --  containing a local minimizer of interest. 
We formally define the active manifold concept in Definition~\ref{defn:ident_man}, 
but at a high level, the two crucial characteristics are that 
(i) along the manifold, the function $f$ is smooth
and (ii) normal to the manifold, the function grows sharply. 
For example, Figure~\ref{fig:vufunction} depicts the nonsmooth function $f(u, v) = u^2 + |v|$ for which the $u$-axis plays the role of $\cM$.
In Section~\ref{sec:simpleexample} we will examine this function and explain how we use its typical structure in $\algname$.
This example also has the smooth substructure developed in several seminal works in the optimization literature, including those found in work on identifiable surfaces \cite{wright1993identifiable}, partly smooth functions \cite{lewis2002active}, $\mathcal{VU}$-structures \cite{lemarecha2000,vuoriginal}, and minimal identifiable sets \cite{drusvyatskiy2014optimality}. 
However, crucial to the analysis of $\algname$ are two further properties introduced in~\cite{davis2021subgradient}, called \emph{strong $(a)$-regularity} and \emph{$(b_{\leq})$-regularity}. 
Strong $(a)$-regularity roughly states that the function is smooth in tangent directions to the manifold up to an error term which is linear in the distance to the manifold.
On the other hand, $(b_{\leq})$-regularity is a one-sided uniform semismoothness~\cite{Mifflin77} property that holds automatically when $f$ is (weakly) convex. 
Both properties hold for the two variable example in Figure~\ref{fig:vufunction} and for the function in~\eqref{eq:lbexample}, where the active manifold is the subspace in which the first $m$ variables take on the same value: $\cM = \{ x \in \RR^d \colon x_1 = x_2 = \ldots x_m\}$.

\begin{figure}[H]
	\centering
\centering\includegraphics[width=.35\textwidth]{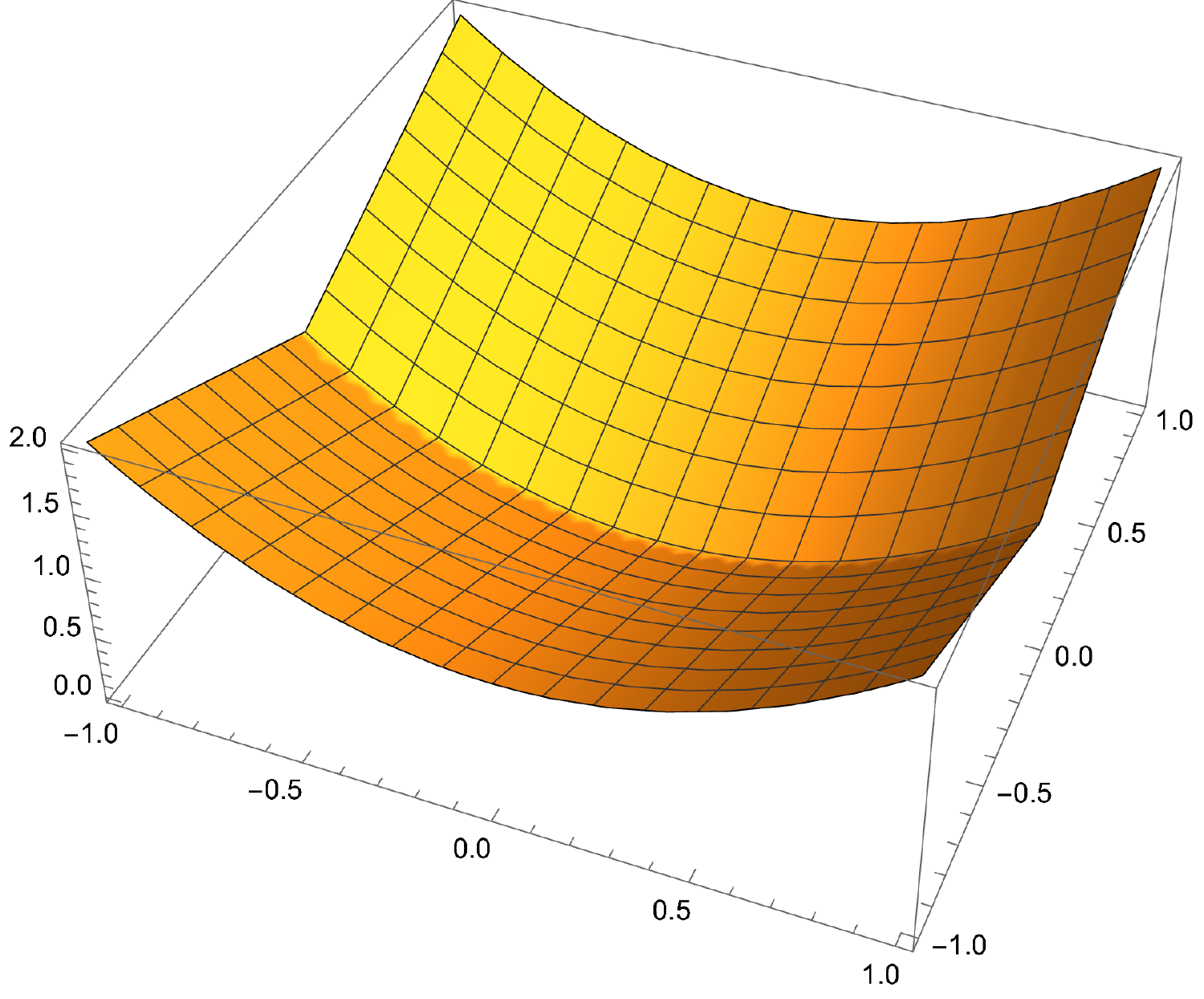}
\caption{The function $f(u,v) = u^2 + |v|$ has typical structure.}\label{fig:vufunction}
\end{figure}

Before turning to the description of $\algname$, we point out that similar smooth substructure has been used in the analysis first-order methods in nonsmooth optimization, most famously for functions with $\mathcal{VU}$-structure \cite{lemarecha2000,vuoriginal} and more recently for max-of-smooth functions.\footnote{Though they also benefit from smooth substructure, \emph{proximal-methods} do not fall within the oracle model of first-order methods considered in this work. Thus, we omit them from our discussion.}
For $\cV\cU$ functions, so-called ``bundle-methods,"~\cite{wolfe1975method,lemarechal1975extension} which possess an inner-outer loop structure, have been shown to converge superlinearly with respect to the number of outer-loop steps~\cite{vuoriginal}; see also the survey~\cite{oliveira2014bundle}.
These methods have excellent empirical performance, but a complete account of their inner-loop complexity remains elusive.
On the other hand, in a recent work, 
Han and Lewis proposed a first-order method -- Survey Descent -- that converges linearly on certain strongly convex max-of-smooth objectives, stepping beyond the classical smooth setting~\cite{han2021survey}. 
The method shows favorable performance beyond the max-of-smooth class, e.g., on certain eigenvalue optimization problems, but no theoretical justification for this success is available.
We discuss Survey Descent in more detail in Section~\ref{sec:maxofsmoothsection}.
We now motivate $\algname$.

\subsection{Motivation: Goldstein's conceptual subgradient method}

To motivate $\algname$ and the role of smooth substructure, let us set 
the stage: consider the nonsmooth optimization problem:
\begin{align*}
\minimize_{x \in \RR^d} \; f(x), 
\end{align*}
where $f \colon \RR^d \rightarrow \RR$ is a locally Lipschitz function, which is not necessarily convex. 
The algorithm developed in this work assumes \emph{first-order oracle access} to $f$~\cite{intro_lect,bubeck2015convex,complexity}.
In particular, at every $x\in \RR^d$ we must be able to evaluate $f(x)$ and retrieve an element of the \emph{Clarke subdifferential} $\partial f(x)$.
Informally, the Clarke subdifferential is comprised of convex combinations of limits of gradients taken at nearby points; a formal definition appears in Section~\ref{sec:notation}.
The Clarke subdifferential reduces to the familiar objects in classical settings.
For example, when $f$ is $C^1$, the Clarke subdifferential reduces to the singleton mapping $\{\nabla f \}$. In addition, when $f$ is convex, the Clarke subdifferential reduces to the subdifferential in the sense of convex analysis.

The starting point of this work is the classical conceptual subgradient method of Goldstein~\cite{goldstein1977optimization}.
The core object in this method is the {Goldstein subdifferential}:
\begin{align}\label{eq:goldsteinsubgradientdef}
\partial_{\sigma} f(x) := \conv \left(\bigcup_{y \in \overline B_{\sigma}( x)} \partial f(y)\right) \qquad \text{ for all $x \in \RR^d$ and $\sigma > 0$.}
\end{align}
This subdifferential is simply the convex hull of all Clarke subgradients of $f$ taken at points inside  the ball of radius $\sigma.$
Its importance arises from the following descent property proved in~\cite{goldstein1977optimization}: fix $\sigma > 0$ and $x \in \RR^d$ and let $w$ denote the minimal norm element of $\partial_{\sigma} f(x)$. 
Then
\begin{align}\label{eq:descent}
f\left(x - \sigma \frac{w}{\|w\|}\right) \leq f(x) - \sigma \|w\| \qquad\text{ if $w \neq 0$.}
\end{align}
This property motivates Goldstein's conceptual subgradient method, which simply iterates:
\begin{align}\label{eq:goldsteinsgm}
x_{k+1} = x_k - \sigma \frac{w_k}{\|w_k\|} \qquad \text{ where } \qquad w_k = \argmin_{w \in \partial_{\sigma} f(x_k)} \|w\|.
\end{align}
This algorithm is remarkable since it is provably a descent method for any Lipschitz function and even converges at a sublinear rate.
Indeed, a quick appeal to~\eqref{eq:descent} yields
$$
\min_{k = 0, \ldots, K-1}\|w_k\| \leq \varepsilon  \qquad\text{holds when} \qquad K \geq \frac{f(x_0) - \min f}{\sigma \varepsilon}.
$$
While this exact variant of the Goldstein method is not necessarily implementable, recent work has devised approximate versions of the method that have similar sublinear convergence properties~\cite{zhang2020complexity,davis2021gradientsampling}.

The algorithm introduced in this work approximately implements the method~\eqref{eq:goldsteinsgm}.
The goal of this work is to prove that the method is locally nearly linearly convergent on typical nonsmooth functions.
To develop such a method, we must resolve two issues for this problem class.
First, we must develop rapidly convergent algorithms that approximately compute the minimal norm element of the Goldstein subdifferential.
Second, we must devise an appropriate regularity property that ensures the proposed method converges nearly linearly. 
We will discuss both of these properties in turn, beginning with a regularity property {\color{blue} that relates the decrement in~\eqref{eq:descent} to the function gap.}

\subsection{Linear convergence via a gradient inequality}\label{sec:introgkl}

{\color{blue} Observe that if the bound}
$$\sigma\|w_k\| \geq \eta (f(x_k) - \min f)$$ 
holds for some $\eta > 0$ and all $k > 0$, then the Goldstein method~\eqref{eq:goldsteinsgm} converges linearly to a minimizer of $f$.
A potential issue with this inequality is that the vector $w_k$ is zero whenever $\sigma$ is larger than the distance of $x_k$ to the nearest critical point of $f$; thus the algorithm may stall whenever $x_k$ is near enough to a minimizer. 
This suggests a simple relaxation of the property that allows $\sigma$ to depend on $x_k$.

Indeed, we will provide conditions under which the following bound holds near a local minimizer $\bar x$ of $f$: 
there exists a constant $\eta > 0$ and a function $\sigma \colon \RR^d\rightarrow \RR_{+}$ such that for all $x$ near $\bar x$, we have
\begin{align}\label{eq:introgkl}
\sigma(x) \dist(0, \partial_{\sigma(x)} f(x)) \geq \eta (f(x) - f(\bar x)).
\end{align}
{\color{blue} throughout, we will refer to this bound as} {\color{blue} a \emph{gradient inequality}},  
due to its similarity to the Kurdyka-{\L}ojasiewicz (KL) gradient inequality~\cite{doi:10.1137/060670080}.
The KL inequality requires that a suitable nonlinear reparameterization $\psi \colon \RR \rightarrow \RR$ of the function gap is bounded by the minimal norm Clarke subgradient for all $x$ near $\bar x$:
$$
\dist(0,  \partial f(x)) \geq \psi(f(x) - f(\bar x)).
$$
In recent years, the KL inequality has played a key role in establishing convergence and rates of convergence for proximal methods in nonsmooth optimization and in continuous time analogs of the subgradient method; see e.g.,~\cite{attouch2010proximal,attouch2013convergence,bolte2014proximal,xu2013block,doi:10.1137/060670080}. 

{\color{blue} To illustrate, let us specialize to the semialgebraic setting, where the desingularization function $\psi$ is known to take the form $\psi(r) = r^{\theta}$ for $\theta \in [0, 1)$. 
The work~\cite[Theorem 2]{attouch2009convergence} initiated the study of convergence of proximal methods in this setting, showing that the proximal point method asymptotically converges to its limit point, which is critical but not necessarily optimal. 
The method convergence in finitely many steps when $\theta = 0$, locally converges linearly when $\theta \in (0, 1/2]$, and locally converges at the rate $k^{\frac{-(1-\theta)}{2\theta - 1}}$ when $\theta \in (1/2, 1)$. Further works such as~\cite{bolte2014proximal,attouch2010proximal} generalized the techniques to related proximal methods. 
Passing to continuous time, one is interested in the convergence of the trajectory of subdifferential inclusion satisfying $\dot x(t) \in -\partial f(x(t))$ at almost every $t$. Here, the rates of convergence exactly parallel those in the proximal methods as shown in~\cite[Theorem 4.7]{loja}.\footnote{{\color{blue}These rates were shown only for ``lower-$C^2$" semialgebraic losses, but extend to locally Lipschitz semialgebraic functions via the semialgebraic ``chain rule" proved in~\cite{davis2020stochastictame}.}}
In contrast to the proximal and continuous-time settings, we do not know whether the KL inequality alone allows one to design a locally linearly convergent discrete-time subgradient method, except in the setting where $\theta = 0$ (i.e., $f$ is \emph{sharp}) and $f$ is convex~\cite{Polyak69} or \emph{weakly convex}~\cite{davis2018subgradientweaklyconvex}; weakly convex functions form a broad class of nonconvex functions that includes all compositions of Lipschitz convex functions with smooth mappings. 
When $\theta > 0$, to the best of our knowledge, the best rate proved in the literature for any subgradient type method is $k^{\frac{-(1-\theta)}{2\theta}}$~\cite{johnstone2020faster}; this result is only known to hold for convex functions.\footnote{{\color{blue}The results stated in~\cite{johnstone2020faster} pertain to functions with Hölder growth; thus, to prove the results stated in the paragraph, we must use the following known fact: functions satisfying the KL inequality with exponent $\theta$ have Hölder growth with exponent $1/(1-\theta)$, which follows from the proof of~\cite[Theorem 3.7]{drusvyatskiy2021nonsmooth}.}}
}

A well-known property of the KL inequality is its prevalence: it holds at each critical point of an arbitrary lower-semicontinuous semialgebraic function $f$~\cite{doi:10.1137/060670080}.
We will show that the {\color{blue} gradient inequality}~\eqref{eq:introgkl} is also prevalent in the sense that it holds for the aforementioned  problems with typical structure.
In this way, the conceptual method~\eqref{eq:goldsteinsgm} with varying $\sigma_k := \sigma(x_k)$ will locally converge linearly on such problems.
The reader may wonder whether we can or must find the precise value $\sigma(x_k)$.
We will show that for typical problems, an appropriate $\sigma_k$ may be found through a line search procedure.

\subsection{Approximately implementing Goldstein's method}
The gradient inequality ensures that the conceptual Goldstein method converges linearly, provided the stepsize $\sigma$ is chosen adaptively.
To move beyond the conceptual setting, we must develop strategies for approximating the minimal norm element of $\partial_{\sigma}f(x)$ for $\sigma > 0$ and $x \in \RR^d$. 
Let us suppose we have such a method and denote it by $\minnorm(x, \sigma)$. 
Then the method of this work simply iterates:
\begin{align}\label{eq:approximateminnorm}
x_{k+1} = x_k - \sigma_k \frac{w_k}{\|w_k\|} \qquad\text{ and } \qquad  w_k = \minnorm(x_k, \sigma_k)
\end{align}
for an appropriate sequence $\sigma_k >0$.
We will discuss and develop two different implementations of $\minnorm(x, \sigma)$ in this work.
Given $x \in \RR^d$ and $\sigma > 0$, both methods iteratively construct a sequence of Clarke subgradients $ g_0, \ldots,  g_{T-1}$ taken at points in the ball $\overline B_{\sigma} (x)$ and then output a ``small" convex combination $w \in \conv\{ g_0, \ldots,  g_{T-1}\}$, which satisfies the descent condition 
\begin{align}\label{eq:descentintroapprox}
f\left(x - \sigma \frac{w}{\|w\|}\right) \leq f(x) - \frac{\sigma}{8} \|w\|.
\end{align}
The oracle complexity of $\minnorm(x, \sigma)$ is then $T$ function/subgradient evaluations, and we hope to ensure that $T$ is relatively small, for example, a constant or at most 
$$
T = O\left(\log\left(\Delta_{x, \sigma}^{-1}\right)\right) \qquad \text{ where $\Delta_{x,\sigma} := \dist(0, \partial_{\sigma} f(x)).$}
$$
Provided that $T$ is on this order, that $f$ satisfies the gradient inequality~\eqref{eq:introgkl}, and that $\sigma_k$ is chosen appropriately, the iterate $x_k$ will satisfy $f(x_k) - f(\bar x) \leq \varepsilon$ after at most $O(\log^2(1/\varepsilon))$ iterations, a nearly linear rate of convergence. 
This complexity ignores the cost of choosing an appropriate stepsize $\sigma_k$, but we will show that in typical problems we can find appropriate $\sigma_k$ with at most $O(\log(1/\varepsilon))$ function/subgradient evaluations.

We are aware of two $\minnorm$ type methods in the literature, but their complexity is either too large or is useful only in low dimensions problems.
For example, the works~\cite{zhang2020complexity,davis2021gradientsampling} introduced such a method for general locally Lipschitz functions. However, the complexity of the method is $T = O(1/\Delta_{x, \sigma})$ -- too large for our purposes.
On the other hand, the work~\cite{davis2021gradientsampling} also introduced a method tailored to low-dimensional weakly convex functions.
However, the method is based on cutting plane techniques, so its complexity scales linearly with dimension: $T = O(d\log(1/\Delta_{x, \sigma}))$.

{\color{blue}While existing $\minnorm$ methods are slow for general Lipschitz functions, we show that the aforementioned typical structure allows us to develop $\minnorm$ methods that accelerate in a neighborhood of the minimizer. 
Our approach is based on a decomposition of a neighborhood of the minimizer into two regions: one where the method of \cite{zhang2020complexity,davis2021gradientsampling} is applicable, and another region where a novel $\minnorm$ method may be applied.
}

\subsubsection{The normal and tangent regions}\label{sec:normal_tangent_regions}
{\color{blue}
In this work, we use the active manifold $\cM$ to split the space of $(x, \sigma)$ for $x$ nearby the minimizer $\bar x \in \cM$ into two sets where fast $\minnorm$ methods are available. 
We call the first set the \emph{normal region.} This region consists of points whose normal distance $\dist(x, \cM)$ is larger than a multiple of the squared tangential distance $\|P_{\cM}(x) - \bar x\|^2$, together with stepsizes $\sigma$ proportional to a multiple of the normal distance:
$$
\begin{cases}
\frac{a_1}{2} \dist(x, \cM) \leq \sigma \leq a_1\dist(x, \cM) ;\\
a_2^2\|P_{\cM}(x) - \bar x\|^2 \leq \dist(x, \cM),
\end{cases}
$$
for problem dependent constants $a_1, a_2 \in (0, 1)$; see Theorem~\ref{thm:gkl} for more details.
We will show that in this region, we have $\Delta_{x, \sigma} = \Omega(1)$, so the $\minnorm$ method of~\cite{zhang2020complexity,davis2021gradientsampling} terminates with descent in finitely many steps. 

On the other hand, we call the second set the \emph{tangent region}. This set consists of points whose squared tangential distance is larger than a multiple of the normal distance, together with stepsizes $\sigma$ proportional to a multiple of the tangential distance:
$$
\begin{cases}
\frac{a_2}{2}\|P_{\cM}(x) - \bar x\|  \leq \sigma \leq a_2\|P_{\cM}(x) - \bar x\| ;\\
\frac{\dist(x, \cM)}{\sigma} \leq2a_2\|P_{\cM}(x) - \bar x\|,
\end{cases}
$$
where $a_1$ and $a_2$ are as in the normal region. 
For this region, we will propose a new $\minnorm$ method, which terminates rapidly. 
We note that in both cases we provide a range of valid $\sigma$, rather than a single value since we aim to estimate $\sigma$ with a line search.

\subsubsection{A simple example} \label{sec:simpleexample}
Before describing the methods in detail, let us illustrate the regions and the  principles of the methods on the following simple function of two variables $f(u, v) = u^2 + |v|$, which has a unique minimizer at $\bar x= (0, 0)$.} 
Here, the $u$-axis is the active manifold $\cM$. Along the manifold, $f$ is smooth and grows quadratically, while off of the manifold, $f$ grows sharply; see Figure~\ref{fig:vufunction} {\color{blue}for a plot of the function and see Figure~\ref{fig:normal_and_tangent_regions} for the $x$-component of the normal and tangent regions for $f$ (with $a_2 = 1/8$).
The manifold $\cM$ induces a decomposition of $f$ into smooth $f_{\cU}(u,v) = u^2$ and nonsmooth $f_{\cV}(u, v) = |v|$ components. In particular, with variable $x = (u, v)$, we have 
\begin{align}
f(x) = f_\cU(x) + f_{\cV}(x) = \|P_{\cM}(x) - \bar x\|^2 + \dist(x, \cM). \label{eq:normaltangent_decomp_f}
\end{align}

From this decomposition, we see $f_\cV$ is dominant in the normal region, while $f_{\cU}$ is dominant in the tangent region.
Likewise, as we will argue momentarily, the minimal norm Goldstein subgradient $w_{\sigma} \in \partial_{\sigma} f(x)$ satisfies $\|w_{\sigma}\| \geq \|\nabla f_{\cV}(x)\|$ in the normal region, while $\|w_{\sigma}\| = \Omega(\|\nabla f_{\cU}(x)\|)$ in the tangent region.
This has several consequences. First, in the normal region, the $\minnorm$ method of~\cite{zhang2020complexity,davis2021gradientsampling} will terminate in finitely many steps, due to the lower bound $\|w_{\sigma}\| \geq 1$.
On the other hand, in the tangent region, $\|w_{\sigma}\|$ can be much smaller, so we must introduce a new method to generate descent.
Finally, assuming these approximations are accurate, the gradient inequality~\eqref{eq:approximateminnorm} quickly follows:
in the normal region, we have 
$$
\sigma \|w_{\sigma}\| = \Omega(\dist(x, \cM)) =  \Omega(f(x)),
$$
while in the tangent region, we have 
$$
\sigma \|w_{\sigma}\| = \Omega(\|P_{\cM}(x) - \bar x\|\|\nabla f_{\cU}(x)\|) = \Omega( \|P_{\cM}(x) - \bar x\|^2) = \Omega(f(x)).
$$
Though it follows from immediate calculations in this example, in the more general setting, the following consequence of quadratic growth will be crucial in establishing a similar bound: $\|\nabla f_{\cU}(x)\| = \Theta(\|P_{\cM}(x) - \bar x\|)$.
}

\begin{figure}[H]
     \centering
     \begin{subfigure}[b]{0.45\textwidth}
         \centering
         \includegraphics[width=\textwidth]{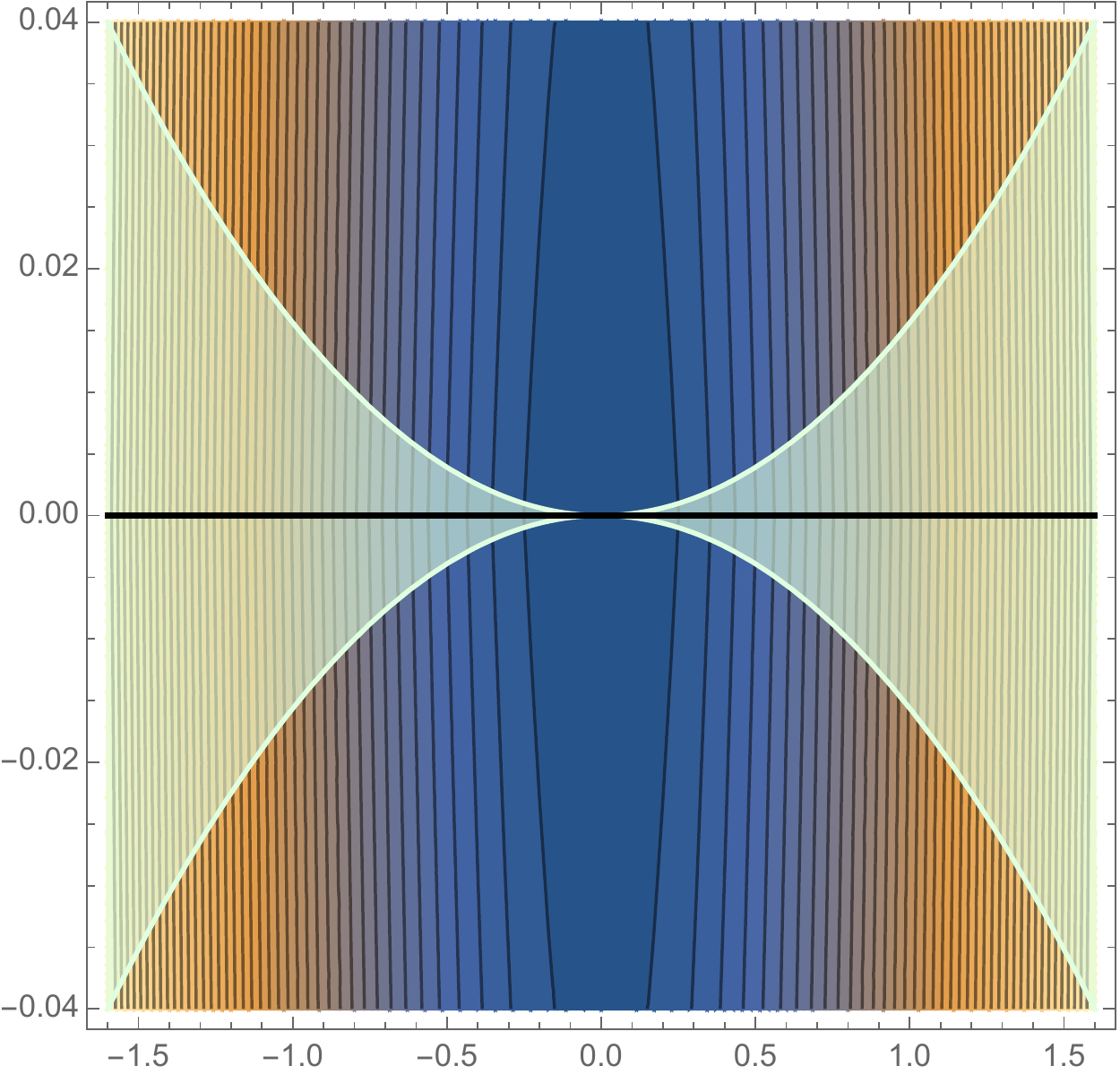}
     \end{subfigure}
     \hfill
     \begin{subfigure}[b]{0.45\textwidth}
         \centering
         \includegraphics[width=\textwidth]{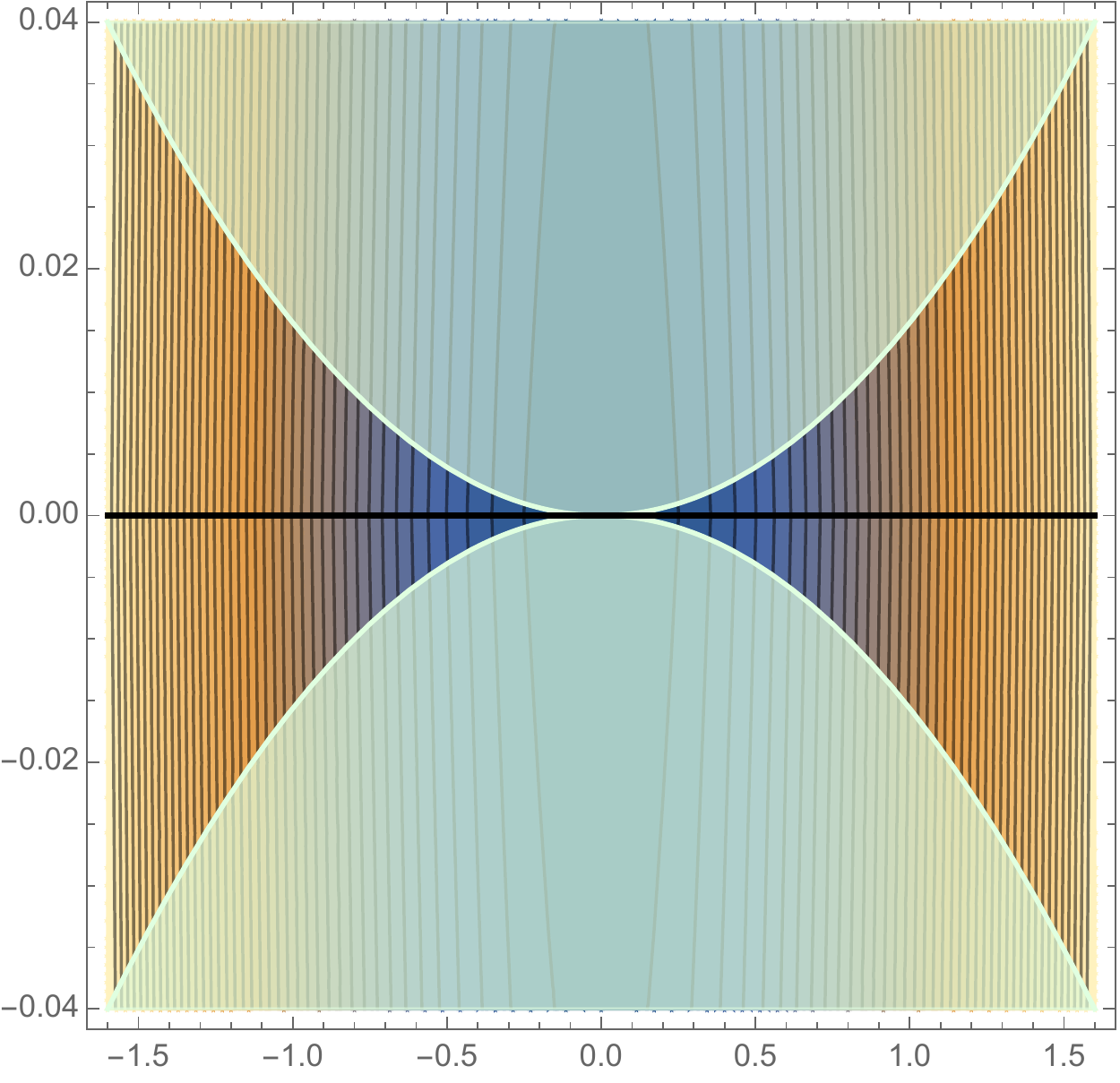}
     \end{subfigure}
        \caption{{\color{blue}Contour plots for $f(u, v) = u^2 + |v|$ together with $\cM$, shown in black. The $x$-component of the tangent (left) and normal (right) regions are overlaid in light green.}}
        \label{fig:normal_and_tangent_regions}
\end{figure}

{\color{blue} 
Now, to lower bound $\|w_{\sigma}\|$ we use the following fact: $\nabla f_{\cU}(u, v)$ is tangent to $\cM$, while $\nabla f_{\cV}(u, v)$ is normal to $\cM$ when $v \neq 0$.
Thus, to lower bound $\|w_{\sigma}\|$ in the normal region, we will simply lower bound the size of the normal component of $w_{\sigma}$. 
Indeed, since $\sigma < \dist(x, \cM)$, all points $x' \in B_{\sigma}(x)$ are on the same side of $\cM$. Therefore, the normal component of $w_{\sigma}$ is an average of \emph{identical} gradients $\nabla f_{\cV}(x')= \nabla f_{\cV}(x)$. 
Likewise, in the tangent region, we lower bound the tangent component of $w_{\sigma}$. Indeed, since $\sigma < \|\bar x - P_{\cM}( x)\|/8$, the projection onto $\cM$ of all points $x' \in B_{\sigma}( x)$ are on the same side of the origin. Thus, the tangent component of $w_{\sigma}$ is an average of nearly identical gradients $\nabla f_{\cU}(x') \approx \nabla f_{\cU}(x)$, yielding the lower bound. We prove a more general form of both of these lower bounds Lemma~\ref{lem:lbnormal} and Lemma~\ref{lem:lowerboundgoldstein}, which follow from similar principles.

Turning to algorithms, we have so far noted that the $\minnorm$ method of~\cite{zhang2020complexity,davis2021gradientsampling} may be used in the normal region. 
In the tangent region, we are unsure how to design a method that can quickly recover $w_\sigma$.
Instead of searching for $w_\sigma$ directly, we take a slightly different perspective in the tangent region: we seek a vector $g \in \partial_{\sigma} f(x)$ with ``small" normal component, meaning: 
$$
\|P_{N}(g)\| = O(\|\nabla f_{\cU}(x)\|^2)
$$
where $N$ is the normal space to $\cM$. 
Intuitively, when $g$ has small normal component, the nonsmooth part $f_{\cV}$ minimally changes along a gradient step. On the other hand, if $g$ is sufficiently correlated with $\nabla f_{\cU}$, the smooth part $f_\cU$ decreases at an appropriate rate; we prove this in a more general setting in Lemma~\ref{lem: decentwithsmallnormal}.   

Why might one expect such a $g$ to be available in the tangent region? 
The reason is that the gradient of the smooth component is itself a Goldstein subgradient.
Indeed, for points near the origin and in the tangent region, the tangential distance is much larger than the normal distance. Thus, the reflection of any point $(u, v)$ across the manifold $\cM$ is contained in $B_{\sigma}(x)$, 
which immediately implies gradient of the smooth component is an element of Goldstein subdifferential:
\begin{align}
\nabla f_{\cU}(u, v) = \frac{1}{2}\nabla f(u, v) + \frac{1}{2} \nabla f(u, -v) \in \partial f_{\sigma}(x). \label{eq:exact_reflection}
\end{align}
While the inclusion~\eqref{eq:exact_reflection} illustrates one way to construct such a $g$, we cannot hope for perfect symmetry in general problems. 

Instead, a central insight of this work is that a similar approximate reflection exists in problems with typical structure. 
To illustrate, consider Figure~\ref{fig:flipgradient}. 
This figure depicts a point $x$ in the tangent region together with the result of a  
normalized gradient step:
$$
x_+ := x - \sigma \frac{\nabla f(x)}{\|\nabla f(x)\|}.
$$
As can be seen from the figure, $x_+$ is an approximate reflection of $x$ across the $u$-axis, 
which ``flips the sign" of the nonsmooth component of $\nabla f$: $\nabla f_{\cV}(x) = -\nabla f_{\cV}(x_+)$.
Thus, in this setting, one may ``cancel out" the nonsmooth component by a simple averaging: 
$$
\nabla f_{\cU}(x) \approx \frac{1}{2}\nabla f(x) + \frac{1}{2}\nabla f(x_+).
$$
While seemingly crude, we will show this strategy generalizes to typical functions.
An important distinction with the general setting is that a single averaging step alone will no longer suffice. 
Nevertheless, we show that by iterating this process, we can geometrically shrink the normal component of the Goldstein gradient, eventually yielding descent.
}

\begin{figure}[H]
     \centering
     \begin{subfigure}[b]{0.45\textwidth}
         \centering
         \includegraphics[width=\textwidth]{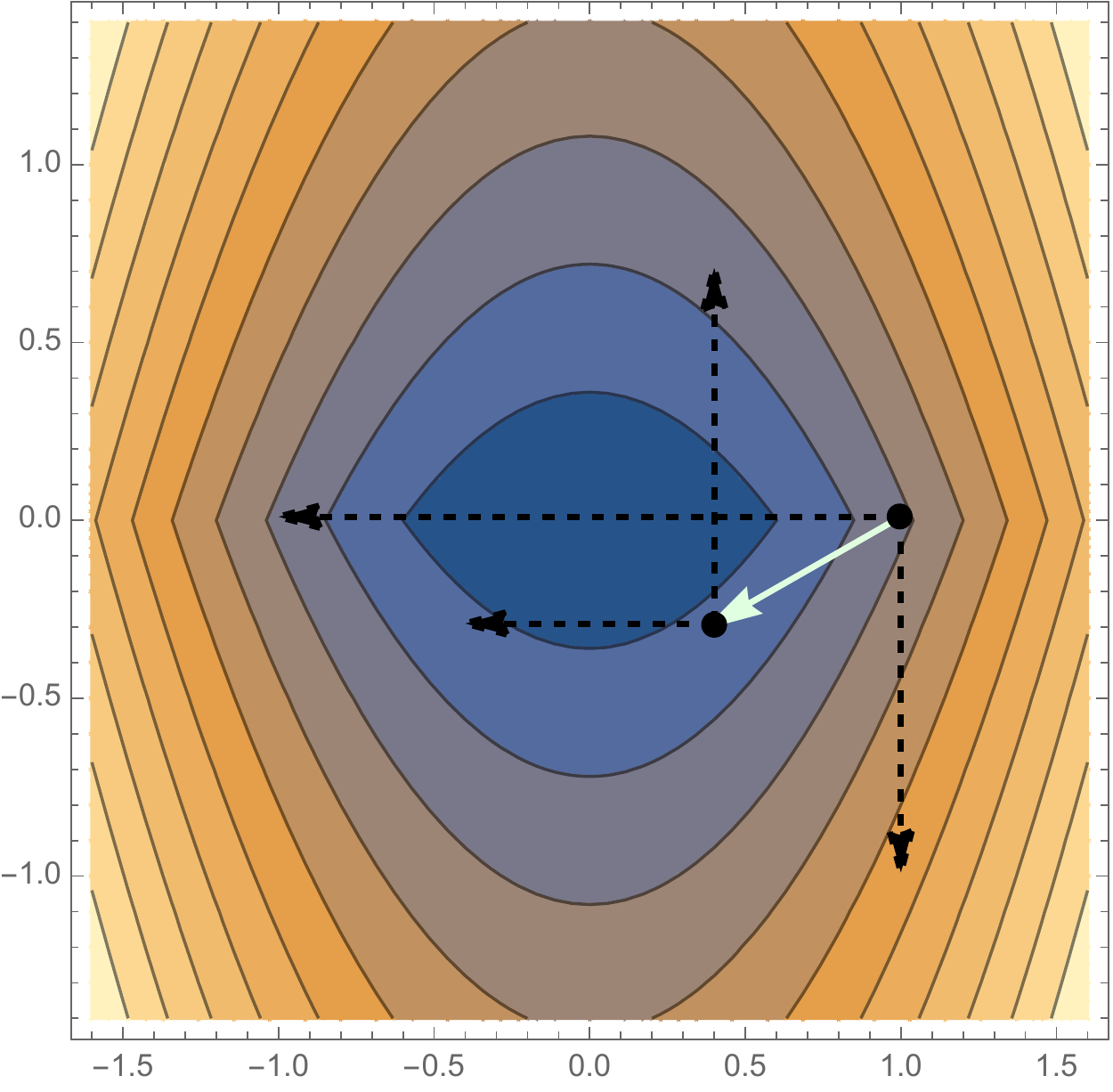}
     \end{subfigure}
     \hfill
     \begin{subfigure}[b]{0.45\textwidth}
         \centering
         \includegraphics[width=\textwidth]{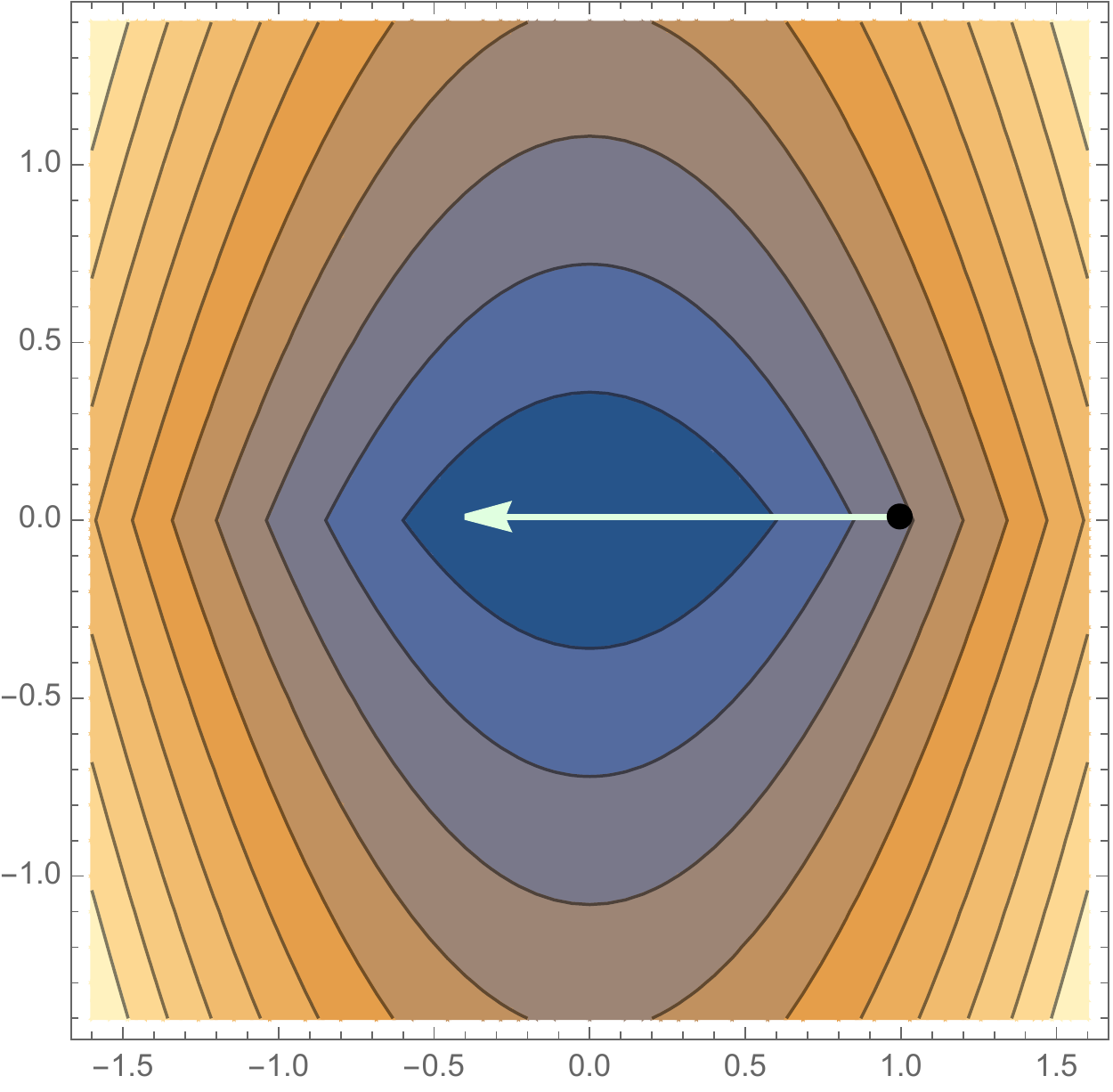}
     \end{subfigure}
        \caption{Contour plots for $f(u, v) = u^2 + |v|$. Left: The point $x = (1, .1)$ together with the approximate reflection $x_+ = x - .3\frac{\nabla f(x)}{\|\nabla f(x)\|}$ across the $u$ axis. The solid light green arrow is parallel to the negative gradient direction $-\nabla f(x)$. The dashed arrows denote the orthogonal decomposition of $-\nabla f(x)$, respectively $-\nabla f(x_+)$, into the vectors $-\nabla f_{\cU}(x)$ and $-\nabla f_{\cV}(x)$, respectively $-\nabla f_{\cU}(x_+)$ and $-\nabla f_{\cV}(x_+)$. From the plot, we see $\nabla f_{\cV}(x) = -\nabla f_{\cV}(x_+)$. Right: The point $x$ with estimate $-\frac{1}{2}(\nabla f(x) + \nabla f(x_+))$ of the vector $-\nabla f_{\cU}(x)$.}
        \label{fig:flipgradient}
\end{figure}

\subsubsection{Two $\minnorm$ methods: $\ngoldstein$ and $\tgoldstein$}

To generalize the strategy outlined in the previous section, we will prove that the minimal norm Goldstein subgradients of typical problems similarly split into tangent and normal components just as in Section~\ref{sec:simpleexample}.
Then, we introduce two $\minnorm$ type methods for ``normal" and ``tangent" steps.

For $(x, \sigma)$ in the normal region, we use a small modification of the $\minnorm$ type method of~\cite{davis2021gradientsampling}. 
We call this method \emph{Normal Descent} ($\ngoldstein$) and describe it in Algorithm~\ref{alg:ngoldstein}. 
As in the simple example above, we will show that $\ngoldstein$ must terminate with an approximately minimal norm Goldstein subgradient in finitely many steps, provided $\sigma$ lies within an appropriate range. We will show that this subgradient is a descent direction satisfying~\eqref{eq:descentintroapprox}.

\begin{algorithm}[H]
	\caption{$\ngoldstein(x, g, \sigma, T)$}
	\label{alg:ngoldstein}
	\begin{algorithmic}[1]
		\State {\bfseries Set}   $g_0 = g$ and $t= 0$.
		 \While {$T-1\geq \innerg$, $\|g_{\innerg}\| >0$,  and $\frac{\sigma}{8} \|g_\innerg\|\geq f(x)-f\left(x-\sigma \frac{g_\innerg}{\|g_\innerg\|}\right)$}
		 \State \label{alg:ngoldsteinperturb} Choose any $r$ satisfying $ 0<r< \sigma\|g_t\|$.  
		\State Sample $\zeta_\innerg ~\textrm{uniformly~from}~\mathbb{B}_{r}(g_\innerg)$.
		\State Choose $y_\innerg$ uniformly at random in the segment $\left[x, x-\sigma \frac{\zeta_\innerg}{\|\zeta_\innerg\|}\right]$.\; 
		\State  Choose $\hat g_\innerg \in \partial f(y_t)$.\;		
		\State $g_{\innerg+1}=\argmin_{z\in [g_\innerg,\hat g_t]} \|z\|_2$.\;
		\State $t = t+1$.
		\EndWhile
		\State \Return $g_{t}$.
	\end{algorithmic}
\end{algorithm}

{\color{blue} 
We illustrate the principle behind $\ngoldstein$ as follows. 
Suppose we are given a vector $g \in \partial_{\sigma} f(x)$ not satisfying the descent condition, i.e., with $u := \frac{g}{\|g\|}$, we have
$$
f\left(x - \sigma u\right) - f(x) \geq -\frac{\|g\|}{8}.
$$
Then by Lebourg mean value theorem~\cite[Theorem 2.4]{clarke2008nonsmooth} (provided that $f$ is differentiable along the line segment between $[x, x']$, which can be ensured by adding a small perturbation to $g$; we ignore this in our discussion), we may assume that 
$$
f\left(x - \sigma u\right) - f(x) = \sigma \int_0^1 -\left\langle \nabla f\left(x - \sigma tu\right), u\right\rangle dt = -\sigma\dotp{v, u},
$$
where $v := \int_{0}^1\nabla f\left(x - t u\right)dt \in \partial_{\sigma}f(x)$. Consequently, $\dotp{v, g} \leq \|g\|^2/8$.
While it is not possible to compute $v$, we can compute a \emph{random} element  of the Goldstein subdifferential, satisfying the same inequality in expectation. Indeed, 
defining $v' = \nabla f(y)$ where $y$ is uniformly sampled from the line segment $[x, x - \sigma u]$ (with end points $x$ and $x-\sigma u$), we have $\dotp{\EE_{y}[v'], g} \leq \|g\|^2/8$.
Based on this bound, a quick calculation shows that the minimal norm element $g_+$ of the line segment $[g, v']$ satisfies the bound 
$$
\EE_y\|g_+\|^2 \leq \|g\|^2 - \frac{\|g\|^4}{16L^2}.
$$
Moreover $g_+ \in \partial_{\sigma} f(x)$.
Thus, repeating this process yields a decreasing sequence of Goldstein subgradients which tend to zero as long as the descent condition is not met. In general, the norms of the subgradients generated by this process decay at a rate of $1/k$. 
However, we will prove that $\dist(0, \partial_{\sigma} f(x))$ is bounded below by a fixed constant when $(x, \sigma)$ is in the normal region described in Section~\ref{sec:normal_tangent_regions}.
Consequently, the loop must exist in finite time with descent (with high probability), for otherwise we will have found a subgradient norm strictly smaller than $\dist(0, \partial_{\sigma} f(x))$; see Proposition~\ref{lemma:normal}. The reader interested in the formal calculations may consult~\cite{zhang2020complexity,davis2021gradientsampling}.}

On the other hand, for $(x,\sigma)$ in the tangent region, we develop a new $\minnorm$ type method, which likewise relies on an approximate reflection property. 
We call this method \emph{Tangent Descent} ($\tgoldstein$) and present it in Algorithm~\ref{alg:tgoldstein}.
Given an input point $x$, stepsize $\sigma > 0$, and initial subgradient $g_0 \in \partial f(x)$,  $\tgoldstein$ repeats the following steps
\begin{align*}
\text{Choose: }& \hat g_k \in \partial f\left(x - \sigma \frac{g_k}{\|g_k\|}\right);\\
\text{Update: } &g_{k+1} = \argmin_{g \in [g_{k}, \hat g_k]} \|g\|,
\end{align*}
until it achieves descent $f(x - \sigma \frac{g_k}{\|g_k\|}) \leq f(x) - \frac{\sigma}{8}\|g_k\|$ or runs over budget.
%Here the interval $[g_{k}, \hat g_k]$ denotes the line-segment with endpoints $g_k$ and $\hat g_k$.

\begin{algorithm}[H]
	\caption{$\tgoldstein(x, g, \sigma,T)$}
	\label{alg:tgoldstein}
	\begin{algorithmic}[1]
		\State {\bfseries Set}   $g_0 = g$ and $t=0$.
		 \While {$T-1\geq \innerg$, $\|g_{t}\| >0$,  and $\frac{\sigma}{8} \|g_t\|_2\geq f(x)-f\left(x-\sigma \frac{g_t}{\|g_t\|}\right)$}
		\State  Choose $\hat g_\innerg \in \partial f(x- \sigma \frac{g_\innerg}{\|g_\innerg\|})$.
		\State $g_{t+1} =  \argmin_{z \in [g_t,\hat g_t]} \|z\|$ .
		\State $t = t+1$.
		\EndWhile
		\State \Return $g_{t}$.
	\end{algorithmic}
\end{algorithm}

The motivation for this method is that for typical problems the step $x - \sigma \frac{g_k}{\|g_k\|}$ is locally an approximate reflection across $\cM$ that  ``flips" the normal component of the Goldstein subgradient. Indeed, let $y := P_{\cM}(x)$ denote the projection of $x$ onto $\cM$ and let $N := \normal(y)$ denote the normal space to $\cM$ at $y$. Then we will prove that for all $k$, we have
$$
\dotp{P_{N} g_k, \hat g_k} \leq - C\|P_{N} g_k\| + O(\|y - \bar x\|^2),
$$
for some $C > 0$, provided $\sigma$ lies within an appropriate range.
This inequality ensures that each step of the $\tgoldstein$ geometrically decreases the ``normal component" of $g_k$, {\color{blue} until we arrive at a Goldstein subgradient with normal component on the order of $O(\|y - \bar x\|^2)$; see Section~\ref{sec:approximate_reflection}.
Moreover, given $g \in \partial_{\sigma} f(x)$ satisfying 
$$
\|P_N(g)\| \leq \cfour\|y - \bar x\|^2
$$
for a particular problem dependent constant $\cfour > 0$, we will prove the descent condition
$$
f\left(x - \sigma \frac{g}{\|g\|}\right) \leq f(x) - \frac{\sigma\|g\|}{8}
$$
holds; see Lemma~\ref{lem: decentwithsmallnormal}. 
Combining these two facts shows that $\tgoldstein$ will rapidly terminate with descent.
}

\subsection{The $\algname$ algorithm}\label{sec:algdescription}

We call the main algorithm of this work \emph{Normal Tangent Descent} ($\algname$) and present it in Algorithm~\ref{alg:mainalg}. 
At a high level the method is an approximate implementation of Goldstein's conceptual subgradient method as in~\eqref{eq:approximateminnorm}, using $\ngoldstein$ and $\tgoldstein$ as $\minnorm$ type methods.
As input it takes three parameters: an initial point $x$; a sequence of grid-sizes $\{\gridsize_k\}$ for the line search on $\sigma$; and a sequence of budgets $\{T_k\}$ for the $\minnorm$ type methods $\ngoldstein$ and $\tgoldstein$. 
Later we will show that the user may simply set $T_k = G_k = k+1$ for all $k \geq 0$.

\begin{algorithm}[H]
	\caption{$\linesearch(x, g, s, \gridsize, T)$}
	\label{alg:linesearch}
	\begin{algorithmic}[1]
	\State {\bfseries Set}  $v_{0} = g$.
	\For {$i = 0, \ldots, \gridsize-1$}\label{line:linesearchbegin}
		\State $\sigma_i = 2^{-(\gridsize - i )}$.
		\State \textcolor{blue}{$u_{i} =  \tgoldstein(x, v_{i}, \sigma_i,T)$.} \label{line:argmin} 
		\State \textcolor{blue}{$v_{i + 1} =  \ngoldstein(x, u_{i}, \sigma_i,T)$}.
	\EndFor\label{line:linesearchend}
	\State $\tilde x := \argmin \{f(x') \colon x' \in \{x\} \cup \{x - \sigma_i \frac{v_{i+1}}{\|v_{i+1}\|} \colon \sigma_i \leq \frac{\|v_{i+1}\|}{s}, i = 0, \ldots, \gridsize-1\} \}$.\label{line:trustregion} 
	\State \Return $\tilde x$.
	\end{algorithmic}
\end{algorithm}

\begin{algorithm}[H]
	\caption{$\algname(x, g,  {\color{blue}\sscale}, \{\gridsize_k\}, \{T_k\})$}
	\label{alg:mainalg}
	\begin{algorithmic}[1]
	\Require s$g \neq 0$, {\color{blue}$\sscale \in (0, 1]$}
	\State {\bfseries Set}   $x_0 = x$ and $g_0 = g$.
	\For{$k = 0, 1, \ldots$}	
	\State $x_{k+1} = \linesearch(x_k, g_k, {\color{blue}\max\{\|g_k\|, \sscale\|g_0\|\}}, \gridsize_k, T_k)$. 
	\State {Choose} $g_{k+1} \in \partial f(x_{k+1})$.
	\EndFor
	\end{algorithmic}
\end{algorithm}

The workhorse of $\algname$ is the line search procedure in Algorithm~\ref{alg:linesearch} ($\linesearch$).
Let us briefly comment on the structure of this method.
Lines~\ref{line:linesearchbegin} through~\ref{line:linesearchend} of Algorithm~\ref{alg:linesearch}  implement a line search on $\sigma$.
Line~\ref{line:trustregion} chooses the Goldstein subgradient that provides the most descent while enforcing the trust-region constraint  $\sigma_i \leq \frac{\|v_{i+1}\|}{s}$.
Line~\ref{line:trustregion} also ensures the $\algname$ is a descent method.
Within the line search procedure, we evaluate $\tgoldstein$ and $\ngoldstein$ a total of $\gridsize$ times each.
Not all of the calls to $\tgoldstein$ and $\ngoldstein$ will succeed with descent within the allotted budget $T$, but we will show that for typical problems, at least one will generate sufficient descent provided $x_k$ is close enough to a local minimizer and $T$ is sufficiently large.
{\color{blue}The reason at least one will succeed with descent is that given any $x$ sufficiently near the solution and parameters $G$ and $T$ sufficiently large, $\linesearch$ will find a $\sigma$ such that $(x, \sigma)$ is in either the normal or tangent region described in Section~\ref{sec:normal_tangent_regions}.}
The line search allows the possibility that $\sigma$ is as large as $1/2$, which might force $x_{k+1}$ to leave the region surrounding the minimizer $\bar x$. 
This concern is what motivates the somewhat unusual structure of the line search method wherein the $\minnorm$-type methods are nested. 
Indeed, on the one hand, the nesting ensures the norms of the Goldstein subgradients $\|v_{i+1}\|$ are decaying as $\sigma_i$ increases. 
On the other hand, the trust region constraint ensures that $\sigma_i$ is not chosen too large, {\color{blue} which we need for two technical reasons in our analysis: (i) it prevents $x_{k+1}$ from leaving a small neighborhood around the minimizer where our regularity assumptions hold; (ii) we can only ensure $\tgoldstein$ terminates quickly when $\sigma \leq \delta_{\text{Grid}}$, for a certain radius $\delta_{\text{Grid}}$ defined in Lemma~\ref{lem: approximatereflection2}, which may be substantially smaller than $1/2$.

Computationally, it may at first seem desirable to drop the trust region constraint. 
Figure~\ref{fig: lowerboundtrust} shows this may not be the case.
We suspect the reason is two-fold: First, the trust region constraint allows us to cut off a range of $\sigma$ from our search, which might otherwise waste oracle calls; indeed, since $\|v_{i+1}\|$ is nonincreasing in $i$, and $\sigma_i$ is increasing, once the trust region is violated, it will be violated for all larger $i$.
Second, although we may take longer steps by disabling the trust region constraint, the amount of descent we expect is on the order of $\Omega(\sigma_i \|v_{i+1}\|)$. Thus, since the norms $\|v_{i+1}\|$ are nonincreasing, larger stepsizes $\sigma_i$ do not necessarily translate to larger descent.
}

{\color{blue} Finally, we comment on our motivation for choosing the scaling $s_k = \max\{\|g_k\|, \sscale\|g_0\|\}$  in the trust region constraint.
First, note that it is possible to prove, using identical techniques, that the $\algname$ converges when one replaces $s_k$ by any positive sequence that is bounded from above and below by positive constants. 
For our particular choice of $s_k$, the term $\sscale\|g_0\|$ ensures the sequence is bounded below, while the local Lipschitz continuity of $f$ ensures that $s_k$ is bounded above.
Second, we wish for the trust region constraint to be unaffected by rescalings of $f$. 
Our particular choice of $s_k$ guarantees scaling invariance, since the subgradients of $af$ are simply the subgradients of $f$ scaled by $a$ for any positive constant $a$.
One might introduce other schemes for choosing $s$, but we did not explore such strategies.
Finally, we found that performance of $\algname$ is relatively insensitive to the choice of $\sscale > 0$, and any $\sscale 
\in \{10^{-i} \colon i = 0, 2, 4, 6\}$ yielded adequate performance; see Figures~\ref{fig:3e} and~\ref{fig:3f}. 
}

 \begin{figure}[H]
 	\centering
 	\begin{subfigure}[b]{0.45\textwidth}
 		\centering
 		\includegraphics[width=\textwidth]{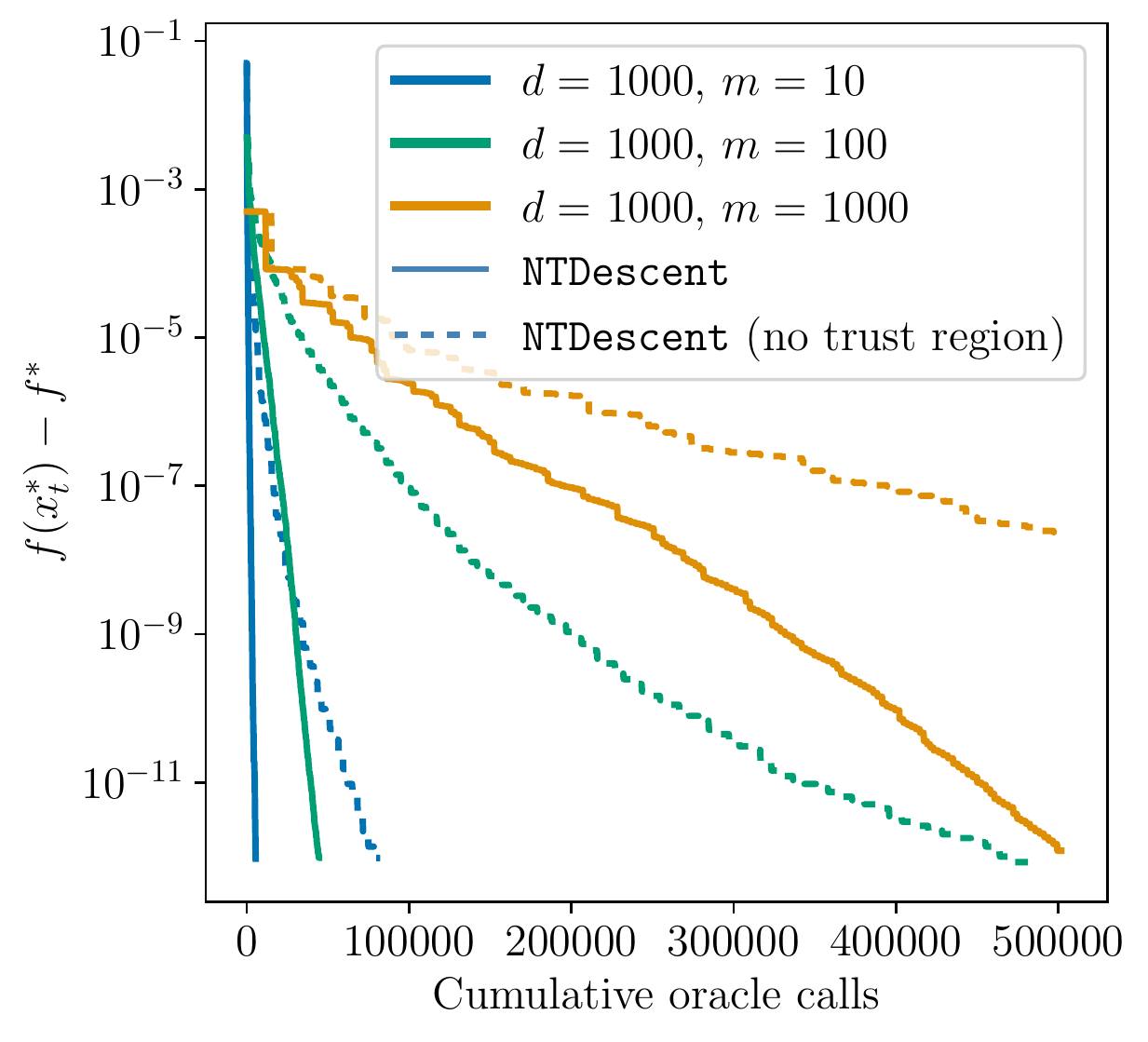}
 		\caption{}
 		\label{fig:lbasec1.4}
 	\end{subfigure}
 	\hfill
 	\begin{subfigure}[b]{0.45\textwidth}
 		\centering%\vspace{-10pt}
 		\includegraphics[width=1\textwidth]{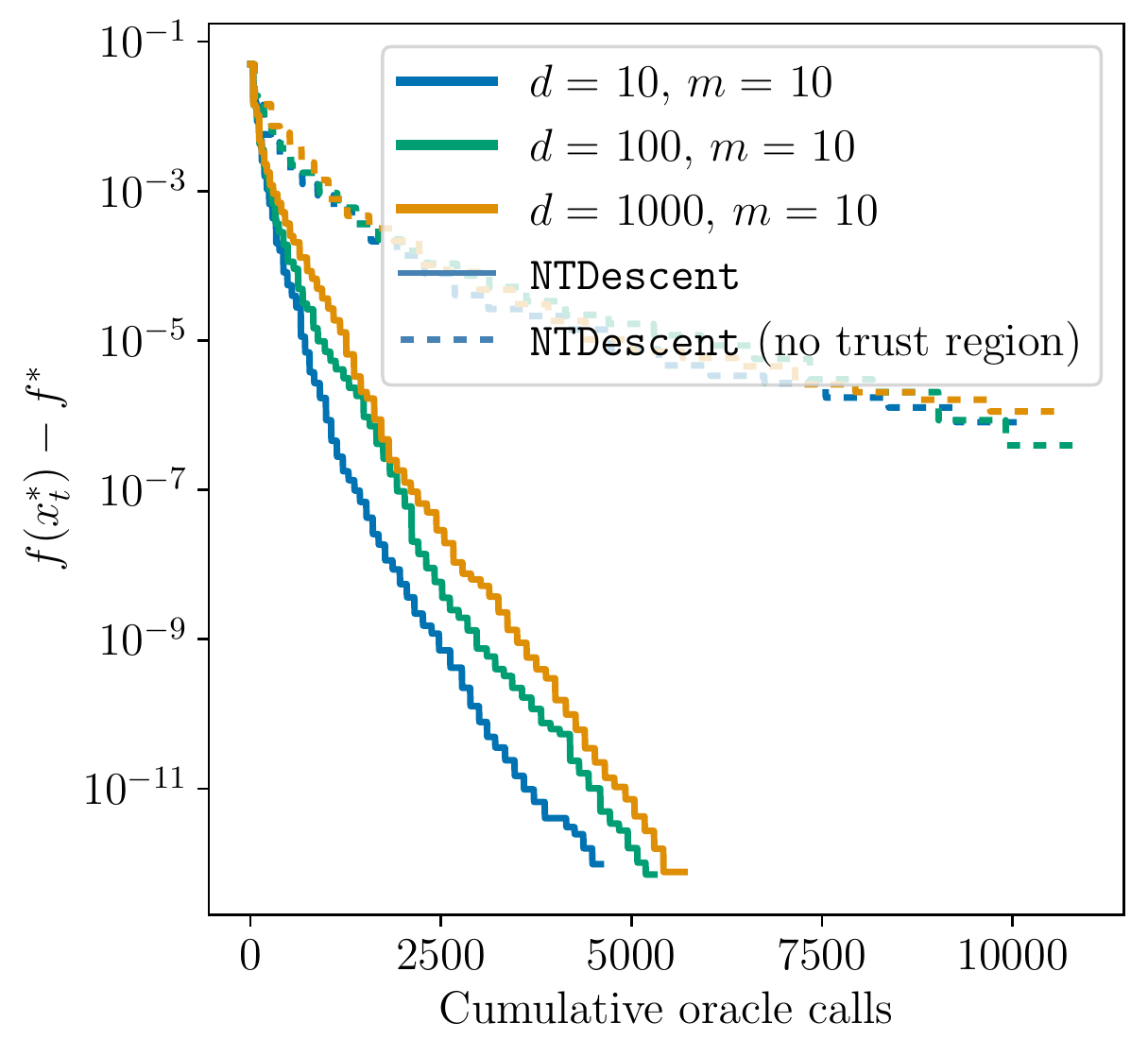}
 		\caption{}
 		\label{fig:lbbsec1.4}
 	\end{subfigure}	
 	\caption{{\color{blue}Comparison of $\algname$ on Problem~\eqref{eq:lbexample} with the trust region constraint in Line~\ref{line:linesearchend} of Algorithm~\ref{alg:linesearch}  removed. Left: we fix $d$ and vary $m$; Right: we fix $m$ and vary $d$. We invite the reader to compare these plots with Figure~\ref{fig: lowerbound}.}}\label{fig: lowerboundtrust} 
 \end{figure}

\subsection{Main convergence guarantees for $\algname$}

The main contribution of this work is a local, nearly linear convergence rate for $\algname$. 
The local rate holds under a key structural assumption -- Assumption~\ref{assumption:mainfinal} -- 
which formalizes the concept of typical structure and mirrors the structure of the simple function considered in Section~\ref{sec:simpleexample}.
While we formally describe  Assumption~\ref{assumption:mainfinal} in Section~\ref{sec:mainassumption}, for now, we mention that it holds for max-of-smooth and properly $C^p$ decomposable functions, provided the local minimizer $\bar x$ is a strong local minimizer that satisfies a strict complementarity condition;
this class includes the max-of-smooth setting considered in~\cite{han2021survey}.
Assumption~\ref{assumption:mainfinal} also holds for generic linear tilts of semialgebraic functions: if $f$ is semialgebraic, then for a full Lebesgue measure set of $w \in \RR^d$, Assumption~\ref{assumption:mainfinal} holds at every local minimizer $\bar x$ of the tilted function $f_w \colon x \mapsto f(x) + w^\T x$. We now present the theorem.

\begin{thm}[Main convergence theorem]\label{thm:maintheoremsemi}
Let $f \colon \RR^d \rightarrow \RR$ satisfy Assumption~\ref{assumption:mainfinal} at a local minimizer $\bar x \in \RR^d$.
Fix {\color{blue} scalar $\sscale \in  (0, 1]$}, budget $\{T_k\}$ and grid size $\{G_k\}$ sequences satisfying $$\min\{T_k, G_k\} \geq k+1 \qquad \text{ for all $k \geq 0$.}$$
Suppose that for initial point $x_0 \in \RR^d$, there exists a subgradient $g_0 \in \partial f(x_0)$ such that $g_0 \neq 0$.
Consider iterates $\{x_k\}$ generated by  $\algname(x_0, g_0, {\color{blue} \sscale}, \{G_k\},\{T_k\})$.
For any $q, k_0, C > 0$,  let $E_{k_0,q,C}$ denote the event:
$$
f(x_k) - f(\bar x) \leq \max\{(f(x_{k_0}) - f(\bar x))q^{k - k_0}, Cq^{k}\}  \text{ for all $k \geq k_0$.}
$$
Then there exists $q \in (0, 1)$, $C, C' > 0$, and a neighborhood $U$ of $\bar x$ depending solely on $f$ such that   for any failure probability $p \in (0, 1)$ and all $k_0 \geq C'\max\{\log(1/p), 1\}$, we have 
$$
P(E_{k_0,q,C} \mid x_{k_0} \in U) \geq 1- p,
$$
{\color{blue}provided $P(x_{k_0} \in U) > 0$}. Moreover, if $f$ is convex, we have 
$$
P(E_{k_0,q,C}) \geq 1- p.
$$
\end{thm}
The theorem, which is justified in Theorems~\ref{thm:maintheorem} and~\ref{thm:mainconvexsetting}, bounds the function gap and distance by a quantity that geometrically decays in $k$. 
Let us examine the local complexity. 
Recall that each outer iteration of $\algname$ requires at most $2T_kG_k$ first-order oracle evaluations. 
Thus, if $T_k = G_k = k+1$ for all $k \geq 0$, the total number of oracle evaluations of $K$ steps of $\algname$ is at most $O(K^3)$. 
In other words, the local complexity of achieving an $\varepsilon$ optimal solution is $O(\log^3(1/\varepsilon))$ for all sufficiently small $\varepsilon > 0$, {\color{blue} where the big-$O$ notation hides terms depending on the local conditioning of $f$; see Lemma~\ref{lem:contractionlowerbound}.}
Therefore theorem establishes a local nearly linear rate of convergence for $\algname$.

\subsection{Outline}

The outline of this paper is as follows. 
In Section~\ref{sec:notation} we present notation and basic constructions.
This section describes a key structure -- the active manifold -- and cannot be skipped. 
In Section~\ref{sec:sublinear}, we present the sublinear convergence guarantees, which will be useful in the convex setting. 
This section also introduces key properties of the $\ngoldstein$ method, which will be used later in the work.
In Section~\ref{sec:mainassumption}, we introduce our main structural assumption -- Assumption~\ref{assumption:mainfinal} -- and show that it is satisfied for the generic semialgebraic and decomposable problem classes.
In Section~\ref{sec:GKLinequality}, we show that Assumption~\ref{assumption:mainfinal} implies the gradient inequality~\eqref{eq:introgkl}.
In Section~\ref{sec:minnorm} we show that the $\tgoldstein$ and $\ngoldstein$ methods terminate rapidly under appropriate conditions.
In Section~\ref{sec:locallinear}, we use the gradient inequality~\eqref{eq:introgkl} and Assumption~\ref{assumption:mainfinal} to prove that $\algname$ 
locally nearly linearly converges. 
Finally, in Section~\ref{sec:numerical} we provide a brief numerical illustration.

%%%%End Intro%%%%%%%

%%%%Preliminaries%%%%%%
% !TEX root = template.tex

\subsection{Notation and basic constructions}\label{sec:notation}

We use standard  convex analysis notation as set out in the monographs~\cite{RW98,clarke2008nonsmooth}.
Throughout,  $\RR^d$ denotes a $d$-dimensional Euclidean space with the inner product
$\dotp{\cdot,\cdot}$ and the induced norm $\norm{x} = \sqrt{\dotp{x,x}}$. We denote the open ball of
radius $\varepsilon>0$ around a point $x\in \RR^d$ by the symbol $B_{\varepsilon}(x)$. We use the symbol $\closedball$ to denote the closed unit ball at the origin.
For any set $\cX\subseteq\RR^d$, the  {\em distance function} and the {\em projection map} are defined by
\begin{equation*}
\dist(x,\cX):=\inf_{y\in \cX} \|y-x\|\qquad \textrm{and}\qquad
P_\cX(x):=\argmin_{y\in \cX} \|y-x\|,
\end{equation*}
respectively. {\color{blue} Note that the function $\dist(\cdot, \cX)$ is $1$-Lipschitz for any set $\cX$.} 
For any set $\cX \subseteq \RR^d$, all $\bar x \in \cX$, all $x \in \RR^d$, and all $y \in P_{\cX}(x)$, we have 
$$
\|y - \bar x\| \leq 2\|x - \bar x\|.
$$
{\color{blue} We denote the diameter of a set $\cX$ by 
$$
\diam(\cX) = \sup_{x,y \in \cX} \|x - y\|.
$$}
We call a function $h\colon \RR^d \rightarrow \RR$ \emph{sublinear} if its epigraph is a closed convex cone, and in that case we define
$$
\text{Lin}(h) := \{ x\in \RR^d\colon  h(x) = - h(-x)\}
$$
to be its \emph{lineality space.} 
Given a mapping $F \colon \RR^d \rightarrow \RR^m$ and a point $\bar x \in \RR^d$, we define 
$$
\lip_{F}(\bar x) := \limsup_{\substack{x, x' \rightarrow \bar x \\ x \neq x'}} \frac{\|F(x) - F(x')\|}{\|x - x'\|}.
$$
Given a mapping $F \colon \RR^d \rightarrow \RR^{m \times n}$ into the space of $m \times n$ matrices and a point $x \in \RR^d$ then we define 
$$
\lip^{\text{op}}_{F}(\bar x) := \limsup_{\substack{x, x' \rightarrow \bar x \\ x \neq x'}} \frac{\|F(x) - F(x')\|_{\text{op}}}{\|x - x'\|}, 
$$
where $\|\cdot\|_{\mathrm{op}}$ denotes the operator norm defined on $\RR^{m \times n}$.

\paragraph{Semialgebraicity.} We call a set $\cX \subseteq \RR^d$ \emph{semialgebraic} if it is the union of finitely many sets defined by finitely many polynomial inequalities. Likewise, we call a function $f \colon \RR^d \rightarrow \RR$ semialgebraic if its graph $\gph(f) = \{(x, f(x)) \colon x\in \RR^d\}$ is semialgebraic. 

\paragraph {Subdifferentials.}
Consider a locally Lipschitz function $f\colon\RR^d\to\RR^d$ and a point $x \in \RR^d$. 
The Clarke subdifferential is the convex hull of limits of gradients evaluated at nearby points of differentiability:
$$
\partial f(x) = \text{conv} \left\{ \lim_{i \rightarrow \infty} \nabla f(x_i) \colon x_i \stackrel{\Omega}{\rightarrow} x\right\},
$$
where $\Omega \subseteq \RR^d$ is the set of points at which $f$ is differentiable (recall Radamacher's theorem). 
\textcolor{blue}{If $f$ is $L$-Lipschitz on a neighborhood $U$, then 
\begin{align*}
\text{for all $x \in U$ and $v \in \partial f(x)$, we have $\|v\| \leq L$.}
\end{align*}
This fact will be used throughout the paper.} A point $\bar x$ satisfying $0 \in \partial f(x)$ is said to be critical for $f$.
The Goldstein subdifferential, which appears in~\eqref{eq:goldsteinsubgradientdef}, will be a central object throughout. 
An important fact is that $\partial_{\sigma} f(x)$ is a closed convex set for any $x\in \RR^d$ and $\sigma > 0$.

\paragraph{Manifolds.}
We will need a few basic results about smooth manifolds, which can be found in
the references~\cite{boumal2020introduction,lee2013smooth}. A set $\cM \subseteq
\RR^d$ is called a $C^p$-smooth manifold around $\bar x$ (with $p \geq 1$) if there exists a
natural number $m$, an open neighborhood $U$ of $\bar x$,  and a $C^p$ smooth mapping
$F \colon U \rightarrow \RR^m$ such that the Jacobian $\nabla F(x)$ is surjective
and $\cM \cap U = F^{-1}(0)$. The tangent and normal spaces to $\cM$ at
$x \in \cM$ near $\bar x$ are defined to be $T_\cM(x) = \ker(\nabla F(x))$ and $N_{\cM}(x) =
T_{\cM}(x)^\perp = \range(\nabla F(x)^\ast)$, respectively. 
If $\cM$ is a $C^2$-smooth manifold around a point $\bar x$, then there exists $C > 0$ such that $y - x \in T_{\cM}(x) + C\|y - x\|^2 \closedball$ for all $x, y \in \cM$ near $\bar x$. 
We also have that $x - P_{\cM}(x) \in \normal(P_{\cM}(x))$ for all $x$ near $\bar x$.
Moreover, the projection mapping $P_{\cM} \colon \RR^d \rightarrow \RR$ is $C^{p-1}$ smooth on a neighborhood of $\bar x$ and satisfies $\nabla P_{\cM}(x) = P_{\tangent(x)}$ for all $x \in \cM$ near $\bar x$.  

\paragraph{Covariant gradients and smooth extension} Let $\cM\subset \RR^d$ be a $C^p$-manifold around a point $x$ for some $p\geq 1$. Then a function $f\colon \cM\to\RR$ is called $C^q$-smooth (with $q \geq 1$) around the point $x$ if there exists a $C^q$ function $\hat f\colon U\to \RR$ defined on an open neighborhood $U$ of $x$ and that agrees with $f$ on $U\cap \cM$. 
In that case, the projection of $\nabla \hat f(x)$ onto $\tangent(x)$ is independent of the choice of $\hat f$.
We call this projection the {\em covariant gradient of $f$ at $x$} and denote it by
$$\nabla_{\cM} f(x):= P_{\tangent(x)}(\nabla \hat f(x) ).$$
For example, the smooth extension 
$$
f_{\cM} := f \circ P_{\cM}
$$ of $f$ is $C^{\min\{p-1, q\}}$ smooth on a neighborhood of $x$ and agrees with $f$ along $\cM$.
Thus, we will use the identification: $\nabla_{\cM} f(x):= \nabla f_{\cM}(x).$

\paragraph{Active Manifolds.}

In this work, we will assume the local minimizer of interest lies on an \emph{active manifold.} 
Informally, an active manifold is a smooth manifold along which the function varies smoothly and off of which the function varies sharply. 
We adopt the formal model of activity explicitly used in~\cite{drusvyatskiy2014optimality}.
Related models exist, e.g., 
identifiable surfaces \cite{wright1993identifiable}, manifolds of partial smoothness \cite{lewis2002active}, $\mathcal{VU}$-structures \cite{lemarecha2000,vuoriginal}, and $g\circ F$ decomposable functions \cite{shapiroreducible}. 

\begin{definition}[Active manifold]\label{defn:ident_man}{\rm 
Consider a function $f\colon\RR^d\to\RR$ 
	and 
fix a set $\mathcal{M} \subseteq \RR^d$ 
	containing 
	a point $\bar x$ satisfying $0\in \partial f(\bar x)$. 
Then $\mathcal{M}$ 
	is called an
{\em  active} $C^p${\em-manifold around} $\bar x$  
	if 
		there exists 
	a neighborhood $U$ of $\bar x$ such that the following are true:
\begin{itemize}
\item {\bf (smoothness)}  
The set $\mathcal{M}$ 
	is 
a $C^p$-smooth manifold near $\bar x$ and  the restriction 
	of 	$f$ 
	to $\mathcal{M}$ 
	is $C^p$-smooth near $\bar x$.
\item {\bf (sharpness)} 
The lower bound holds:
$$\inf \{\|v\|: v\in \partial f(x),~x\in U\setminus \cM\}>0.$$
\end{itemize}}
\end{definition}

In the introduction, we provided two examples of functions that admit an active manifold at their minimizers. 
For example, function  $f(u, v) = u^2 + |v|$ admits the active manifold $\cM := \{(u, 0) \colon u \in \RR\}$ around the origin. On the other hand, the function $f$ from~\eqref{eq:lbexample} admits the active manifold $\cM = \{ x \in \RR^d \colon x_1 = x_2 = \ldots x_m\}$ around the origin.
{\color{blue} To draw a distinction with partial smoothness property of~\cite{lewis2002active}, the function $f(x,y) = \max\{x,0\} + y^2$ is partly smooth along the $x$ axis, but the $x$-axis is not an active manifold for $f$ around the origin; indeed, $f$ does not satisfy the sharpness condition at the origin.}
We now turn to sublinear convergence guarantees.

%%%%%%%%%%%%%%%%%

%%%%% GKL %%%%%%%%%
% !TEX root = template.tex

\section{Global sublinear convergence of $\algname$}\label{sec:sublinear}

The main goal of this work is to show that $\algname$ locally converges nearly linearly for ``typical" nonsmooth optimization problems.
A natural question is whether $\algname$ also possesses global nonasymptotic convergence guarantees. 
In this section, we prove two such guarantees: 
First, for arbitrary Lipschitz functions, we analyze the rate at which $\dist(0, \partial_{\sigma_i} f(x_k))$ tends to zero as a function of $k$.
Second, for convex Lipschitz functions, we analyze the rate at which $f(x_k)$ tends to $\inf f$.

{\color{blue} In the proofs of this section, the $\tgoldstein$ loop is ignored as we can only prove it terminates with descent near the minimizer. Instead, the global convergence guarantees follow from the properties of $\ngoldstein$. Thus, our analysis essentially follows that of \cite{davis2021gradientsampling}, where a nearly identical $\minnorm$ method was introduced. The main difference between the $\ngoldstein$ and the method of \cite{davis2021gradientsampling} lies in the perturbation radius in Line~\ref{alg:ngoldsteinperturb} of Algorithm~\ref{alg:ngoldstein}: while the radius of $\ngoldstein$ can be computed with access only to $\sigma\|g_t\|$, the radius in \cite{davis2021gradientsampling} requires knowledge of the Lipschitz constant of $f$, which we do not assume. Finally, we mention that \cite{davis2021gradientsampling} did not consider rates of convergence for convex problems.
}

Before stating the main result, we recall three key Lemmas, which underlie the proof. 
The first lemma shows that the vectors $u_i$ and $v_i$ generated by  $\linesearch$ are Goldstein subgradients of decreasing norm. 
\begin{lem}[Properties of $\linesearch$]\label{lem:linesearchproperties}
Let $f \colon \RR^d \rightarrow \RR$ be a locally Lipschitz function. 
Fix $x \in \RR^d$, subgradient $g\in \partial f(x)$, budget $T$, and  grid size $\gridsize$.
Let $u_i$ and $v_i$ be generated by $\linesearch(x, g, \gridsize, T)$. 
Then 
\begin{align}
u_i, v_{i+1} \in \partial_{\sigma_i} f(x)
\qquad \text{ and } \qquad \|v_{i+1}\| \leq \|u_i\| \leq \|v_{i}\|  \label{eq:decreasingnormsub}
\end{align}
for all $i = 0, \ldots, G-1$.
\end{lem}
\begin{proof}
The proof follows by induction. 
We prove the base case only, since the induction is straightforward.
First note that the inclusion $v_0 \in \partial f(x)$ implies that $u_0 \in \partial_{\sigma_0} f(x)$, since $\tgoldstein$ constructs $u_0$ as a convex combinations of subgradients evaluated in the ball $\overline B_{\sigma_0}(\bar x)$. 
Likewise, due to the $\argmin$ operation on line~\ref{line:argmin} of Algorithm~\ref{alg:tgoldstein}, the subgradients generated by $\tgoldstein$ are decreasing in norm. Consequently, we have $\|u_0\| \leq \|v_0\|$.
A similar argument shows that $v_1 \in \partial_{\sigma_0} f(x)$ and $\|v_1\| \leq \|u_0\|$. This completes the proof.
\end{proof}

The next lemma shows that when $f$ is convex, the minimal norm Goldstein subgradient may be used to bound the function values.
We place the proof in Appendix~\ref{sec:prooflem:goldsteinsubgradientinequality}, since it follows from a standard argument.
\begin{lem}[Subgradient inequality]\label{lem:goldsteinsubgradientinequality}
Suppose that $f \colon \RR^d \rightarrow \RR$ is a continuous convex function.
Let $x, y \in \RR^d$. 
Let $L$ denote a Lipschitz constant for $f$ on the ball $B_{2\sigma}(x)$.
Then 
$$
 f(x) - f(y) \leq \|x - y\| \dist(0, \partial_{\sigma}f(x))  + 2\sigma L.
$$
\end{lem}

The final lemma provides conditions under which $\ngoldstein$ terminates with descent with high probability.
The result is closely related to~\cite[Corollary~2.6]{davis2021gradientsampling}, {\color{blue} but we take extra care to analyze the perturbation radius in Line~\ref{alg:ngoldsteinperturb} of Algorithm~\ref{alg:ngoldstein}.}
\begin{lem}[$\ngoldstein$ loop terminates with descent]\label{lemma:normalgold}
Let $f$ be a locally Lipschitz function. Fix initial point $x \in \RR^d$, radius $\sigma > 0$, subgradient $g \in \partial_{\sigma} f(x)$, and failure probability $p \in (0, 1)$.
Furthermore, let $L$ be a Lipschitz constant of $f$ on the ball $B_{2\sigma} (x)$.
Suppose that 
$$
\sigma \leq \frac{\dist(0, \partial_{\sigma} f(x))}{\sqrt{128}L}; \qquad \text{ and } \qquad 
T \geq \left\lceil \frac{64L^2}{\dist^2(0, \partial_{\sigma} f(x))} \right\rceil \left\lceil 2\log(1/p) \right\rceil.
$$
Define $g_+ := \ngoldstein(x, g, \sigma, T)$. 
Then $\|g_+\| \neq 0$ and the point $x_+ := x - \sigma\frac{g_+}{\|g_+\|}$ satisfies
$$
f(x_+) \leq f(x) - \frac{\sigma \dist(0, \partial_{\sigma} f(x))}{8} \qquad \text{with probability at least $ 1- p$.}
$$
\end{lem}
\begin{proof}
First note that $g_+ \in \partial_{\sigma} f(x)$, so $\|g_+\| \geq \dist(0, \partial_{\sigma} f(x)) > 0$.
Now, observe that $\ngoldstein$ is precisely \cite[Algorithm~1]{davis2021gradientsampling} with a different bound on the perturbation radius $r$. 
Indeed, in \cite[Algorithm~1]{davis2021gradientsampling}, $r$ must satisfy 
$$
r < \|g_t\|\sqrt{1 - \left(1 - \frac{\|g_t\|^2}{128L^2}\right)^2}
$$
for all $t \geq 0$. 
We now show that the constraint $r \leq \sigma\|g_t\|$ implies the above bound. 
To that end, define the univariate function $h\colon a \mapsto \sqrt{1 - (1- \frac{a^2}{128L^2})^2}$.
Then $h$ is increasing in $a$ for $a \leq L$. Moreover, for $a \in [0, L]$, we have 
$
h(a) \geq \frac{a}{\sqrt{128}L}.
$
Consequently, since $$\dist(0, \partial_{\sigma} f(x)) \leq \|g_t\| \leq L$$ for all $t \leq T$, we have
$$
r < \sigma\|g_t\| \leq \frac{\dist(0, \partial_{\sigma} f(x))\|g_t\|}{\sqrt{128}L} \leq  h(\dist(0, \partial_{\sigma}f(x)))\|g_t\| \leq h(\|g_t\|)\|g_t\|.
$$
Thus the proof is a direct application of \cite[Corollary~2.6]{davis2021gradientsampling}.
\end{proof}

Given these lemmata, we are now ready to state and prove our main sublinear convergence guarantee. 
\begin{thm}[Sublinear convergence]\label{thm:sublinear}
Let $f \colon \RR^d \rightarrow \RR$ be a locally Lipschitz function.
Fix initial point $x_0 \in \RR^d$ and subgradient $g_0 \in \partial f(x_0)$. 
Assume that $g_0 \neq 0$.
Let $L \in \RR\cup \{+\infty\}$ be any Lipschitz constant of $f$ over the widened sublevel set 
$$
S: = \{x+ u \colon f(x) \leq f(x_0) \text{ and } u \in \overline B(x)\}.
$$
Fix a {\color{blue} scalar $\sscale \in (0,1]$}, budget sequence $\{T_k\}$, grid size sequence $\{\gridsize_k\}$, and failure probability $p\in (0, 1)$.
Let $\{x_k\}$ be generated by $\algname(x, g, {\color{blue} \sscale}, \{\gridsize_k\},\{T_k\})$. 
Then for all $K > 0$, the following holds with probability at least $1-p$:
Define $G := \min_{K \leq k \leq 2K-1}\gridsize_k$ and  $T: = \min_{K \leq k \leq 2K - 1} T_k$. 
Then for all $i \leq G$, {\color{blue} the following bound holds with $\sigma_i := 2^{-(G-i)}$:}
$$
\min_{K \leq k \leq 2K-1}  \dist(0, \partial_{\sigma_i} f(x_k)) \leq  
\max\left\{\frac{8(f(x_{K}) - \inf f)}{\sigma_i K}, \frac{16L\sqrt{2 \log(KG/p)}}{\sqrt{T}}, \sqrt{128}L \sigma_i\right\}.
$$
Now suppose that $f$ is convex and define {\color{blue} $D := \diam(\{x\in \RR^d \colon f(x) \leq f(x_0)\})$.} Then
%\max_{K \leq k \leq 2K-1} \dist(x_k, \argmin f)$. 
\begin{align}\label{eq:functiongapboundsublinear}
f(x_{2K-1}) - \inf f \leq  
\min_{i \leq G}\left\{ D\max\left\{\frac{8(f(x_{K}) - \inf f)}{\sigma_i K}, \frac{16L\sqrt{2 \log(KG/p)}}{\sqrt{T}}, \sqrt{128}L\sigma_i\right\} + 2L\sigma_i  \right\}.
\end{align}
\end{thm}

\begin{proof}
Let us assume that $L < +\infty$; otherwise the result is trivial. Fix $K > 0$ and $i \leq G$. Define
$$
\epsilon_{i} :=  \max\left\{  \frac{16L\sqrt{2 \log(K\gridsize /p)}}{\sqrt{T}}, \sqrt{128}L\sigma_i\right\}.
$$
For every $K \leq k \leq 2K-1$, define 
$$
 x_{k, i} := x_k - \sigma_i\frac{v_{i+1}}{\|v_{i+1}\|}, \qquad \text{ where } v_{i+1} := \ngoldstein(x_k, u_i, \sigma_i, T_k), 
$$
and  $u_i$ appear in the definition of $\linesearch(x_k, g_k,  {\color{blue}\max\{\|g_k\|, \sscale\|g_0\|\}}, G_k, T_k)$; see Algorithm~\ref{alg:linesearch}.
Note that $v_{i+1} \in \partial_{\sigma_i} f(x_k)$ by Lemma~\ref{lem:linesearchproperties}. 
Thus, in the event $\{\dist(0, \partial_{\sigma_i} f(x_k)) \geq \epsilon_{i}\}$, we have 
\begin{enumerate}
\item $x_{k, i}$ is well-defined since $v_{i+1} \neq 0$; 
\item the trust region constraint $\sigma_i \leq \frac{\|v_{i+1}\|}{s}$ is satisfied \textcolor{blue}{for $s =\max\{ \|g_k\|, \sscale\|g_0\|\}$ (in Algorithm~\ref{alg:linesearch})}; indeed,  
$$
\frac{\|v_{i+1}\|}{s} \geq \frac{\dist(0, \partial_{\sigma_i}f(x_k))}{s} \geq \frac{\sqrt{128}L\sigma_i}{s} \geq \sigma_i,
$$
where the final inequality follows from the bound {\color{blue} $s  \leq L$}, a consequence of the inclusion $x_0 \subseteq \text{int}\; S$ and the Lipschitz continuity of $f$ on $S$.
\end{enumerate}
Finally, for every $K \leq k \leq 2K-1$, define
$$
A_{k, i} := \left\{f(x_{k, i}) - f(x_k) \geq -\frac{\sigma_i \dist(0, \partial_{\sigma_i} f(x_k))}{8} \right\}\cap \{\dist(0, \partial_{\sigma_i} f(x_k)) \geq \epsilon_{i}\}.
$$
Now we apply Lemma~\ref{lemma:normalgold}. 

To that end, observe that since $f(x_k)$ is nonincreasing and $\sigma_i \leq 1/2$, every iterate $x_k$ satisfies 
$
 B_{2\sigma_i}(x_k) \subseteq S.
$
Consequently, $L$ is a Lipschitz constant of $f$ on $ B_{2\sigma_i}(x_k)$. 
Therefore, by Lemma~\ref{lemma:normalgold}, for every $K \leq k \leq 2K-1$, we have
\begin{align}
P(A_{k, i}) \leq P(A_{k, i} \mid \dist(0, \partial_{\sigma_i} f(x_k)) \geq \epsilon_{i} ) \leq \frac{p}{\gridsize K}. 
\end{align}
Thus, by a union bound, with probability at least 
$1-\frac{p}{G}$, at least one of the following must hold at every index $K \leq k \leq 2K-1$:
\begin{align*}
f(x_{k, i}) - f(x_k) \leq -\frac{\sigma_i \dist(0, \partial_{\sigma_i} f(x_k))}{8} \qquad \text{ or } \qquad \dist(0, \partial_{\sigma_i} f(x_k)) \leq \epsilon_{i}.
\end{align*}
If $\dist(0, \partial_{\sigma_i} f(x_k)) \leq \epsilon_i$ for some $k$ satisfying $K \leq k \leq 2K-1$, then the result follows. 
On the other hand, suppose that for all  $K \leq k \leq 2K-1$, we have $\dist(0, \partial_{\sigma_i} f(x_k)) > \epsilon_i$; in particular, we have $\dist(0, \partial_{\sigma_i} f(x_k)) > \sqrt{128}L\sigma_{i}$. 
Therefore, with probability at least $1-\frac{p}{G}$, we must have 
$$
f(x_{k+1}) \leq f(x_{k, i}) \leq f(x_k) - \frac{\sigma_i \dist(0, \partial_{\sigma_i} f(x_k))}{8}, \qquad \text{ for all $K \leq k \leq 2K-1$.}
$$
where the first inequality follows since the trust region constraint is satisfied for $x_{k,i}$.
Iterating this inequality, we have with probability at least $1-\frac{p}{G}$, the bound
$$
\min_{K \leq k \leq 2K-1}  \dist(0, \partial_{\sigma_i} f(x_k)) \leq \frac{1}{K}\sum_{k = K}^{2K-1}  \dist(0, \partial_{\sigma_i} f(x_k)) \leq \frac{8(f(x_{K}) - f(x_{2K}))}{\sigma_i K}.
$$
This proves the result for $i$. 
Taking a union bound over $i$ then yields the bound for minimal norm Goldstein subgradient for all $i \leq G$.

To prove~\eqref{eq:functiongapboundsublinear}, fix an $i \leq G$ and let $k_{i}$ be the index that attains the minimum. 
Then 
$$
f(x_{2K - 1}) - \inf f \leq f(x_{k_i}) - \inf f \leq \dist(x_{k_i}, \cX_\ast) \min_{K \leq k \leq 2K-1}  \dist(0, \partial_{\sigma_i} f(x_k)) + 2\sigma_i L,
$$
where the first inequality follows since  $f(x_k)$ is nonincreasing and the second inequality follows from Lemma~\ref{lem:goldsteinsubgradientinequality}. 
{\color{blue}The proof then follows from the upper bound $\dist(x_{k_i}, \cX) \leq D$.}
\end{proof}

theorem provides bounds on the minimal norm Goldstein subgradient within any window of indices $K \leq k \leq 2K-1$. 
Let us briefly investigate the setting $T_k = k+1$ for all $k \geq 0$.
In this case, theorem implies that with probability at least $1-p$, we have 
$$
\min_{K \leq k \leq 2K-1}  \dist(0, \partial_{\sigma_i} f(x_k)) \leq  \max\left\{\frac{8(f(x_{K}) - \inf f)}{\sigma_i K}, \frac{16L\sqrt{2 \log(KG/p)}}{\sqrt{2K}}, \sqrt{128}L\sigma_i\right\}
$$
for all $i \leq G$. 
Let us now suppose $G$ is large enough that there exists $i \leq G$ satisfying  $(1/2)K^{-1/2} \leq \sigma_i \leq K^{-1/2}$, e.g., we may assume $G_k = \Omega(\log(k^{1/2}))$ for all $k > 0$. Then, we find that at most $O(KTG) = O(K^2G)$ first-order oracle evaluations are needed to find a point $x_k$ satisfying $$\dist(0, \partial_{K^{-1/2}} f(x_k))  = \tilde O(K^{-1/2}),$$ where $\tilde O$ hides logarithmic terms in $G, K$ and $p$. 
Let's consider two settings for $G_k$.
\begin{enumerate}
\item {\bf Setting 1: $G_k = O(\log(k^{1/2}))$.}  In this case, $\algname$ finds a point $x_k$ satisfying \break $\dist(0, \partial_{\varepsilon} f(x_k)) \leq \varepsilon$ using at most $\tilde O(\varepsilon^{-4})$ first-order oracle evaluations.
\item {\bf Setting 2: $G_k = k+1$.} In this case,  $\algname$ finds a point $x_k$ satisfying \break $\dist(0, \partial_{\varepsilon} f(x_k)) \leq \varepsilon$ using at most $\tilde O(\varepsilon^{-6})$ first-order oracle evaluations.
\end{enumerate}
The complexity of Setting 1 is more favorable than the complexity of Setting 2.
Nevertheless, when we establish our local rapid convergence guarantees, we will work in Setting $2$, which has more favorable local convergence properties. 
Before moving on, we note that the above guarantees likewise apply in the convex setting, namely $\algname$ finds a point $x_k$ with $f(x_k) - f^\ast \leq \varepsilon$ using at most $\tilde O(\varepsilon^{-4})$, respectively  $\tilde O(\varepsilon^{-6})$, first-order oracle evaluations in Setting 1, respectively Setting 2.

{\color{blue}
In addition to the nonasymptotic guarantees of Theorem~\ref{thm:sublinear}, the reader may wonder whether a given limit point $\bar x$ of $\algname$ is Clarke critical, meaning $0 \in \partial f(\bar x)$. We prove that this is indeed the case under a bounded sublevel set condition. We place the proof in Appendix~\ref{app:cor:clarkestationarity} since it follows a similar line of reasoning as Theorem~\ref{thm:sublinear}.
	\begin{cor}[Limiting points are Clarke critical]\label{cor:clarkestationarity}
		Let $f \colon \RR^d \rightarrow \RR$ be a locally Lipschitz function.
		Fix  initial point $x_0 \in \RR^d$ and subgradient $g_0 \in \partial f(x_0)$. 
		Assume that $g_0 \neq 0$.
		Suppose the  sublevel set $\{x \colon f(x) \leq f(x_0) \}$ is bounded.  
		Fix {\color{blue} scalar $\sscale \in (0,1]$}, budget sequence $\{T_k\}$, grid size sequence $\{\gridsize_k\}$ such that $\{G_k\}$ tends to infinity and $T_k \ge k$.
		Let $\{x_k\}$ be generated by $\algname(x, g, \sscale, \{\gridsize_k\},\{T_k\})$. 
		Then with probability one, all the limiting points of $\{x_k\}$ are Clarke critical. 
	\end{cor}
}

This concludes our sublinear convergence guarantees for $\algname$. 
In the following section, we describe the key structural assumptions needed to ensure that $\algname$ locally rapidly converges. 

\section{Main assumption, examples, and consequences}\label{sec:mainassumption} 

In this section, we introduce our key structural assumption -- Assumption~\ref{assumption:mainfinal}.
In Section~\ref{sec:examples} we show that Assumption~\ref{assumption:mainfinal} holds for generic semialgebraic functions and certain properly $C^p$ decomposable functions.
Then, in Section~\ref{sec:consequencesofA}, we extract several key consequences of Assumption~\ref{assumption:mainfinal}. 
These consequences will be instrumental in proving the gradient inequality~\eqref{eq:introgkl} and rapid convergence of $\algname.$ 
We now turn to the assumption.

\begin{assumption}\label{assumption:mainfinal}
	{\rm 
	Function $f \colon \RR^d \rightarrow \RR$ is locally Lipschitz with local minimizer $\bar x \in \RR^d$. 
	\begin{enumerate}[label=$\mathrm{(A\arabic*)}$]
			\item {\bf (Quadratic Growth)}\label{assum: quad}
			There exists $\scc > 0$ such that 
			$$
			f(x) - f(\bar x) \geq \frac{\scc}{2} \|x - \bar x\|^2 \qquad \text{for all $x$ near $\bar x$.}
			$$
		\item {\bf (Active Manifold)}\label{assum: smooth} Function $f$ admits a $C^{4}$-smooth active manifold $\cM$ around $\bar x$.
			\item {\bf (Strong-$(a)$ regularity)}\label{assum: stronga} 
			There exists $\ver > 0$ such that 
			$$
			\|P_{T_{\cM}(y)} (v - \nabla_\cM f(y))\| \leq \ver \|x - y\| \qquad \text{ for all $x \in \RR^d$, $v \in \partial f(x)$, and $y \in \cM$ near $\bar x$.}
			$$
			\item {\bf (($b_{\leq}$)-regularity)}\label{assum: bregularity} The following inequality holds  
			$$
			f(y) \geq  f(x) + \dotp{v, y - x} + o(\|y - x\|) \qquad  \text{ as $y \stackrel{\cM}{\rightarrow} \bar x$ and $x \rightarrow \bar x$ with $v \in \partial f(x)$},
			$$
			 where $o(\cdot)$ is any univariate function satisfying $\lim_{t \rightarrow 0} o(t)/t = 0$.
		\end{enumerate}
	}
\end{assumption}

Some comments are in order. 
Assumption~\ref{assum: quad} 
is a classical regularity condition that ensures local linear convergence of gradient methods for smooth convex functions.
Assumptions~\ref{assum: smooth},~\ref{assum: stronga}, and~\ref{assum: bregularity} describe the interaction of $f$ and a distinguished smooth manifold $\cM$.
Assumption~\ref{assum: smooth} requires $\cM$ to be an active manifold for $f$ around $\bar x$ in the sense of Definition~\ref{defn:ident_man}. 
In particular, along the manifold $\cM$, the function $f$ is  $C^4$ smooth with {covariant gradient} $\nabla_{\cM} f$; see Section~\ref{sec:notation} for a definition.
Assumption~\ref{assum: stronga} shows that in tangent directions the covariant gradient along the manifold approximates the subgradients of $f$ up to a linear error. 
This property recently appeared in~\cite{davis2021subgradient,bianchi2021stochastic}, where it was used to study saddle avoidance properties of the subgradient method for nonsmooth optimization.
Finally, Assumption~\ref{assum: bregularity} is a restricted lower smoothness property, showing that linear models of $f$ off the manifold are underapproximators of $f$ on the manifold up to first-order.
Note that the property is automatic if $f$ is weakly convex, meaning the mapping $x \mapsto f(x) + \frac{\rho}{2}\|x\|^2$ is convex for some $\rho \geq 0$. 
The weakly convex class is broad and contains all compositions of convex functions with smooth mappings that have Lipschitz Jacobians; see the survey~\cite{SIAGOPTnewletter2020} for an introduction.
We mention that the name ``$(b_{\leq})$-regularity" is motivated by ``uniform semismoothness" property of~\cite{davis2021subgradient}, which was called the ``$(b)$-regularity property."

In the following section, we provide examples of functions \textcolor{blue}{satisfying} Assumption~\ref{assumption:mainfinal}.

\subsection{Examples of Assumption~\ref{assumption:mainfinal}}\label{sec:examples}

In this section, we show that the aforementioned problems satisfy Assumption~\ref{assumption:mainfinal}.
The most important example is the class of generic semialgebraic functions.
The following theorem is essentially contained in~\cite{drusvyatskiy2016generic,davis2021subgradient}, but we provide a proof for completeness.

\begin{thm}[Generic semialgebraic functions]\label{thm:semialgebraic}
Consider a locally Lipschitz semialgebraic function $f \colon \RR^d \rightarrow \RR$.  
Then for a full Lebesgue measure set of $w \in \RR^d$, the tilted function $f_w \colon  x \mapsto f(x) + w^\T x$ satisfies Assumption~\ref{assumption:mainfinal} at every local minimizer.
\end{thm}
\begin{proof}
\textcolor{blue}{The proof is a consequence of~\cite[Theorem 3.31]{davis2021subgradient} and~\cite[Corollary 4.8, Theorem 4.16]{drusvyatskiy2016generic}.}
\textcolor{blue}{A combination of Corollary 4.8 and Theorem 4.16 in \cite{drusvyatskiy2016generic} }shows that for a full Lebesgue measure set of $w \in \RR^d$, the following hold:
every local minimizer $\bar x$ of $f_w$ lies on a $C^4$ active manifold $\cM$, verifying~\ref{assum: smooth}; and
the quadratic growth condition~\ref{assum: quad} holds at $\bar x$. {\color{blue}Next, \cite[Theorem 3.31]{davis2021subgradient}} shows that $f_w$ also satisfies the strong $(a)$ property~\ref{assum: stronga} along $\cM$; 
applying{\color{blue}~\cite[Theorem 3.11 and Theorem 3.4]{davis2021subgradient}}, we deduce that $f_w$ also satisfies the $(b_{\leq})$-regularity property \ref{assum: bregularity} along $\cM$ at $\bar x$.
\end{proof}

Turning to our second class, we introduce so-called \emph{properly $C^p$ decomposable} functions, originally proposed and analyzed in~\cite{shapiroreducible}. 
At a high-level, the class consists of functions that are locally the composition of a sublinear function with a smooth mapping, which together satisfy a transversality condition.
\begin{definition}[Decomposable functions]\label{def:properlydecomposable}
	{\rm 
		A function $f\colon\RR^d\to \RR$ is called {\em properly $C^p$ decomposable at $\bar x$ as }$h\circ c$ if near $\bar x$ it can be written as 
		$$f(x)=f(\bar x)+h(c(x))$$
		for some $C^p$-smooth mapping $c\colon\RR^d\to \RR^m$ satisfying $c(\bar x)=0$ and some proper, closed sublinear function $h\colon \RR^m\to\RR$ satisfying the transversality condition:
		$$\Lin(h)+ \range(\nabla c(\bar x))= \RR^m.$$	
}
\end{definition}

The following theorem shows that decomposable functions satisfy Assumption~\ref{assumption:mainfinal} near local minimizers if they also satisfy a strict complementarity condition and a quadratic growth bound.
The proof is a consequence of results found in  works~\cite{davis2021subgradient,lewis2002active,shapiroreducible,drusvyatskiy2014optimality}. 

\begin{thm}[Properly decomposable functions]\label{thm:decomposable}
Consider a locally Lipschitz function $f \colon \RR^d \rightarrow \RR$.
Let $\bar x$ be a local minimizer of $f$ and suppose that $f$ is properly $C^4$ decomposable at $\bar x$.
Furthermore, suppose that 
\begin{enumerate}
\item {\bf (Strict Complementarity)} We have that $0 \in \ri \partial f(\bar x)$.
\item {\bf (Quadratic growth)} There exists $\scc > 0$ such that 
$$
f(x) - f(\bar x) \geq \frac{\scc}{2} \|x - \bar x\|^2 \qquad \text{for all $x$ near $\bar x$.}
$$
\end{enumerate}
Then $f$ satisfies Assumption~\ref{assumption:mainfinal} at $\bar x$.
\end{thm}
\begin{proof}
To set the notation for the proof, recall that since $f$ is properly $C^4$ decomposable, there exist functions $h$ and $c$ satisfying the conditions of Definition~\ref{def:properlydecomposable}. 
The discussion in~\cite[p.\ 683-4]{shapiroreducible} then shows that the set 
$$
\cM := c^{-1}(\Lin(h))
$$
is a so-called $C^4$ \emph{manifold of partial smoothness} for $f$ around $\bar x$ in the sense of Lewis~\cite{lewis2002active}. 
{\color{blue}Moreover, $f$ is prox-regular at $\bar x$ for $0$ in the sense of~\cite[Definition 1.1]{poliquin1996prox}, since by definition it is \emph{strongly amenable}~\cite[Definition 2.4]{poliquin1996prox} at $\bar x$; see~\cite[Proposition 2.5]{poliquin1996prox}.
Thus, according to \cite[Theorem 5.3]{hare2004identifying}, partial smoothness, prox-regularity, and strict complementarity ensure that the sharpness condition of Definition~\ref{defn:ident_man} holds. Consequently, $\cM$ is a $C^4$ smooth active manifold around $\bar x$, verifying~\ref{assum: smooth}.}
{\color{blue}In addition, \cite[Corollary 3.24]{davis2021subgradient}} ensures that $f$ satisfies the~\ref{assum: stronga} and~\ref{assum: bregularity} properties along $\cM$.
\end{proof}

A popular class of decomposable objectives arises from pointwise maxima of smooth functions that satisfy an affine independence property. 
For example, this class was considered in the work of Han and Lewis~\cite{han2021survey}.
As an immediate corollary of Theorem~\ref{thm:decomposable}, we show that such functions satisfy Assumption~\ref{assumption:mainfinal}.

\begin{cor}[Max-of-smooth functions]\label{cor:maxofsmooth}
Consider a locally Lipschitz function $f$ and a family of $C^4$ smooth functions $f_i \colon \RR^d \rightarrow\RR$ indexed by a finite set $i \in I$.
Fix a local minimizer $\bar x$ of $f$ and suppose the set $\{\nabla f_i(\bar x)\}_{i \in I}$ is affinely independent. 
Suppose furthermore that $f$ is locally expressible as 
$$
f(x) := \max_{i \in I} f_i(x) \qquad \text{for all $x$ near $\bar x$.}
$$
Then provided the strict complementarity and quadratic growth conditions of Theorem~\ref{thm:decomposable} hold, the function $f$ satisfies Assumption~\ref{assumption:mainfinal} at $\bar x$.
\end{cor}
\begin{proof}
To prove the result, note that the affine independence property is simply a restatement of the transversality condition of Definition~\ref{def:properlydecomposable} for the smooth mapping $x \mapsto (f_i(x))_{i \in I}$ and the sublinear function $y \mapsto \max_{i \in I}y_i$.
\end{proof}

We now turn our attention to the key consequences of Assumption~\ref{assumption:mainfinal}.

\subsection{Key consequences of Assumption~\ref{assumption:mainfinal}}\label{sec:consequencesofA}

The following proposition summarizes the key consequences of Assumption~\ref{assumption:mainfinal}. 
The proof of the result is straightforward but technical, so we place it in Appendix~\ref{sec:prop:proof:consequencesofA}.

\begin{proposition}[Consequences of Assumption~\ref{assumption:mainfinal}]\label{prop:proof:consequencesofA}
Suppose $f$ satisfies Assumption~\ref{assumption:mainfinal} at $\bar x$. 
Then there exists $\deltaA > 0$ such that on the ball $B_{2\deltaA}(\bar x)$, the projection operator $P_{\cM}$ is $C^3$ with Lipschitz Jacobian and the smooth extension $f_{\cM} := f \circ P_{\cM}$ is $C^3$ with Lipschitz gradient. 
Moreover, the following bounds hold:
\begin{enumerate}
\item \label{prop:proof:consequencesofA:item:quadgrowth} {\bf (Quadratic growth)} The quadratic growth bound~\ref{assum: quad} holds throughout $\overline B_{2\deltaA}(\bar x)$.
\item \label{prop:proof:consequencesofA:item:smoothproj} {\bf (Smoothness of $P_{\cM}$)} For all $x \in B_{\deltaA}(\bar x)$ and $x' \in B_{2\deltaA}(\bar x)$, we have
\begin{align}\label{eq:smoothnessprojectorproof}
\|P_\cM(x') - P_\cM(x) -P_{\tangent(P_{\cM}(x))}(x' - x)\| \leq \shape(\dist^2(x, \cM) + \|x - x'\|^2),
\end{align}
where $\shape := 2\lip^{\mathrm{op}}_{\nabla P_{\cM}}(\bar x)$.
\item\label{prop:proof:consequencesofA:item:smoothfunction} {\bf (Bounds on $\nabla_{\cM} f$)} For all $x \in B_{\deltaA}(\bar x)$, we have  
\begin{align}\label{eq:uplowersmoothnessprop}
\frac{\scc}{2}\|P_{\cM}(x) - \bar x\| \leq \|\nabla_{\cM} f(P_{\cM}(x))\| \leq \beta\|P_{\cM}(x) - \bar x\|,
\end{align}
where $\beta := 2\lip_{\nabla f_{\cM}}(\bar x)$.
\item \label{prop:proof:consequencesofA:item:stronga}{\bf (Consequence of strong $(a)$)} For all $x \in B_{\deltaA}(\bar x)$ and $\sigma \leq \deltaA$, we have
\begin{align}
{\color{blue}\sup_{g \in \partial_{\sigma} f(x)}} \|P_{T_{\cM}(P_{\cM}(x))} (g - \nabla_\cM f({\color{blue}P_{\cM}(x))})\| &\leq \ver(\dist(x, \cM) + \sigma); \label{eq:consequencestrongaeasy}\\
\sup_{g \in \partial_{\sigma} f(x)} \|P_{\tangent(P_{\cM}(x))} g\| &\leq \ver(\dist(x, \cM) + \sigma) + \beta \|{\color{blue}P_{\cM}(x)}- \bar x\|\label{eq:consequencestronga};\\
\sup_{g,g' \in \partial_{\sigma} f(x)} \|P_{\tangent(P_{\cM}(x))}(g - g')\| &\leq 2\ver(\dist(x, \cM) + \sigma).\label{eq:consequencestronga2}
\end{align}
\item\label{prop:proof:consequencesofA:item:sharpaim} {\bf (Aiming)} For all $x \in B_{\deltaA}(\bar x)$ and all $v \in \partial f(x)$, we have
\begin{align}\label{eq:sharpnessandaiming}
\dotp{v, x - P_{\cM}(x)} \geq \sharpc \, \dist(x, \cM),  
\end{align}
where $\sharpc := \frac{1}{4}\displaystyle\liminf_{\substack{x' \stackrel{\cM^c}{\rightarrow} \bar x }} \dist(0, \partial f(x')).$
\item {\bf (Subgradient bound)} \label{prop:proof:consequencesofA:item:Lipschitzbound}
For all $x \in B_{\deltaA}(\bar x)$ and $\sigma \leq \deltaA$, we have
$$
\sup_{g \in \partial_{\sigma} f(x)} \|g\| \leq L,
$$
where $L := 2\lip_f(\bar x)$
\item \label{prop:proof:consequencesofA:item:gap}{\bf (Function gap)} For all $x \in B_{\deltaA}(\bar x)$, we have
\begin{align}\label{eq:tangentnormalboundforfunction}
f(x)- f(\bar x) \leq L\dist(x, \cM) + \frac{\beta}{2}\|P_{\cM}(x) - \bar x\|^2. 
\end{align}	
\end{enumerate}
\end{proposition}

Let us briefly comment on the result. 
Item~\ref{prop:proof:consequencesofA:item:smoothproj} provides a crucial smoothness property of the projection operator of $\cM$.
Item~\ref{prop:proof:consequencesofA:item:smoothfunction} shows that the Riemannian gradient of $f$ is proportional to the distance of the projection $y$ to $\bar x$.
Item~\ref{prop:proof:consequencesofA:item:stronga} shows how the Goldstein subgradients inherit the strong $(a)$ property~\ref{assum: stronga} of Assumption~\ref{assumption:mainfinal}.
Indeed, Equation~\eqref{eq:consequencestronga} shows that Goldstein subgradients are ``small" in tangent directions and Equation~\eqref{eq:consequencestronga2} shows Goldstein subgradients vary in an approximate Lipschitz fashion in tangent directions.
Item~\ref{prop:proof:consequencesofA:item:sharpaim} shows that the subgradients of $f$ off of the manifold have a constant level of correlation with $ x- P_{\cM}(x)$, i.e., the direction $-v$ ``aims" towards the manifold. 
Note that $\mu > 0$ due to the active manifold Assumption~\ref{assum: smooth}.
The proof of Item~\ref{prop:proof:consequencesofA:item:sharpaim}  is based on Assumptions~\ref{assum: smooth} and~\ref{assum: bregularity}; a similar result appears in~\cite[Theorem D.2]{davis2019stochasticgeometric}.
Item~\ref{prop:proof:consequencesofA:item:Lipschitzbound} provides a bound on the Goldstein subgradients of $f$ near $\bar x$; 
we will appeal to this bound many times throughout the analysis without referencing this proposition.
Finally, Item~\ref{prop:proof:consequencesofA:item:gap} 
decomposes the function gap into a sum of two terms: the distance to the manifold and the squared distance of the projection to the solution.
The proof relies on the smoothness of $f$ along the manifold.
Note that the trivial upper bound $L\|x - \bar x\|$ for the gap can be weaker than~\eqref{eq:tangentnormalboundforfunction}.

This concludes our discussion of Assumption~\ref{assumption:mainfinal}.
The following three sections establish further consequences: 
the gradient inequality~\eqref{eq:introgkl} (Section~\ref{sec:GKLinequality}); rapid local convergence of $\ngoldstein$ and $\tgoldstein$ (Section~\ref{sec:minnorm}); and rapid local convergence of $\algname$ (Section~\ref{sec:locallinear}).
In all three sections, we use the notation and results introduced in Proposition~\ref{prop:proof:consequencesofA}. 
{\color{blue} 
Finally, the statements of the results in Section~\ref{sec:GKLinequality} and~\ref{sec:minnorm} contain several parameters/radii which we will use in Section~\ref{sec:locallinear} to determine the region of near linear convergence and the oracle complexity for $\algname$. For the readers' convenience, we have listed these parameters in Table~\ref{tab:constants}.
\begin{table}[ht]
\centering
\renewcommand{\arraystretch}{1.5}
\begin{tabular}{@{}ll@{}}
\toprule
{Parameter} & {Definition} \\
\midrule
$\done $ & $\frac{\mu}{8(\mu + L)}$\\% and $\frac{1}{2(1 + 3L/\sharpc)}$ \\
$\dtwo$ & $\frac{\mu}{2}$ \\
$\cone$ & $\frac{\scc}{4}$ \\
$\ctwo$ & $\min\left\{\frac{\scc}{8\ver}, \frac{\min\{1, 1/\deltaA\}}{2}\right\}$ \\
%$C_3$ & $\ctwo/2$ and $\min \left\{\ctwo, \frac{\cone}{8\ver}\right\}$ \\
$\cfour$ & $\frac{\cone^2}{8L}$ \\
$\csix$ & $\min\left\{\frac{\beta }{\ver(1+\deltaA)}, \frac{{\color{blue}\min\left\{\sharpc/\deltaA, \cfour\cfive/\beta\right\}}}{4(1 +  (1+\deltaA)\shape)  (\mu + L))}, {\color{blue}\frac{1}{2}}\right\}$ \\
$\cseven$ & $\min \left\{\frac{\beta}{2\ver}, \frac{\cfour\cfive}{32\ver\beta},\csix, \frac{\ctwo}{4}\right\}$ \\
%$\ceight$ & $\frac{2}{D_1}$ and $\frac{8(\sharpc + L)}{\sharpc}$ \\
$\deltaGKL$ & $\min \left\{\frac{\delta_A}{4}, \frac{\done }{\shape}\right\}$ \\
$\deltaND$ & $\min\left\{\deltaGKL, \frac{\dtwo}{\done L\sqrt{128}}\right\}$ \\
$\deltagrid$ & $\min\left\{\frac{\deltaA}{2}, \frac{1}{\shape(\ceight + 1)}, \frac{\mu}{8(\ver + \beta)}\right\}$ \\
\bottomrule
\end{tabular}
\caption{Parameters used throughout Sections~\ref{sec:GKLinequality} and~\ref{sec:minnorm}.}
\label{tab:constants}
\end{table}
%\begin{align*}
%\done  &= \frac{1}{2(1 + 3L/\sharpc)};   &  \dtwo &= \frac{\mu}{2}; \\
%\cone &= \frac{\scc}{4}; & \ctwo &= {\color{blue} \min\left\{\frac{\scc}{8\ver}, \frac{\min\{1, 1/\deltaA\}}{2}\right\}}; \\
%C_3 &= \min \left\{\ctwo, \frac{\cone}{8\ver}\right\};  & \cfour &= \frac{\cone^2}{8L}; 
%\end{align*}
%\begin{align*}
%\csix &= \min\left\{\frac{\beta }{\ver(1+\deltaA)}, \frac{{\color{blue}\min\{\frac{\sharpc}{\deltaA}, \cfour\cfive\}}}{4(1 +  (1+\deltaA)\shape)  (\mu + L)\beta)}, {\color{blue}\frac{1}{2}}\right\};   
%\end{align*}
%\begin{align*}
%\cseven &= \min \left\{\frac{\beta}{2\ver}, \frac{\cfour\cfive}{32\ver\beta},\csix, C_3\right\}; &  \ceight &= \frac{8(\sharpc + L)}{\sharpc}\\
%\deltaGKL &= \min \left\{\frac{\delta_A}{4}, \frac{\done }{\shape}\right\}; &\deltaND &= \min\left\{\deltaGKL, \frac{\dtwo}{\done L\sqrt{128}}\right\} ;\\
% \deltagrid &= \min\left\{\frac{\deltaA}{2}, \frac{1}{\ceight\shape(1/\ceight + 1)}, \frac{\mu}{8(\ver + \beta)}\right\}& 
% \end{align*}
}

\section{Verifying the gradient inequality~\eqref{eq:introgkl} under Assumption~\ref{assumption:mainfinal}}\label{sec:GKLinequality}

In this section, we establish the gradient inequality~\eqref{eq:introgkl} for functions satisfying Assumption~\ref{assumption:mainfinal}.
Throughout the section, we assume that Assumption~\ref{assumption:mainfinal} is in force. 
We also use the notation set out in Proposition~\ref{prop:proof:consequencesofA}.

We present the formal statement and the gradient inequality~\eqref{eq:introgkl} in Theorem~\ref{thm:gkl}, which appears at the end of this section. 
The proof is a consequence of the two lemmata. 
In the first lemma, we prove a constant-sized lower bound for $\dist(0, \partial_{\sigma} f(x))$, whenever $\sigma$ is sufficiently small. 
The proof of this bound relies on the active manifold assumption~\ref{assum: smooth} and the aiming inequality~\eqref{eq:sharpnessandaiming}. 
A consequence of the argument is that all elements of $\partial_{\sigma} f(x)$ are correlated with the normal direction $x - P_{\cM}(x) \in \normal(P_{\cM}(x))$. 
Later in Proposition~\ref{lemma:normal} we will also show that Algorithm~\ref{alg:ngoldstein} ($\ngoldstein$) terminates rapidly when $\sigma$ is in the region, motivating the name Normal Descent.
We now turn to the lemma.

\begin{lem}[Lower bound on Goldstein subgradients I]\label{lem:lbnormal}
Define 
$$
\done  := {\color{blue}\frac{\mu}{8(\mu + L)}};  \qquad   \dtwo := \frac{\mu}{2}; \qquad \text{ and } \qquad {\color{blue}\deltaGKL} := \min\left\{ \frac{\deltaA}{4},  \frac{\done }{\shape}\right\}.
$$
Then for all $x \in B_{\deltaGKL}(\bar x)$ and $0 < \sigma \leq \done  \dist(x, \cM)$,
we have 
$$
\dist(0, \partial_{\sigma} f(x)) \geq \dtwo.
$$
\end{lem}
\begin{proof}
We begin with some preliminary bounds.
Fix {\color{blue}$x \in B_{\deltaGKL}(\bar x)$} and $\sigma > 0$ satisfying the lemma assumptions. 
We observe that 
$$
\sigma \leq \done  \dist(x, \cM) \leq \dist(x, \cM) \leq  \|x - \bar x\| \leq  {\color{blue}\deltaGKL,}
$$
 where the second inequality follows since $\done  \leq 1$ and the third follows since $\bar x \in \cM$.
Consequently, 
\begin{align}\label{eq:quadraticshrinking}
L\shape(\sigma^2+ \dist^2(x, \cM)) &\leq {\color{blue}\deltaGKL} L\shape(\sigma+ \dist(x, \cM))\notag \\
&\leq 2L{\color{blue}\deltaGKL}\shape\dist(x, \cM) \notag\\
&\leq 2L\done \dist(x, \cM),
\end{align}
where the first inequality follows from the bound $\max\{\sigma,\dist(x, \cM)\} \leq {\color{blue}\deltaGKL}$ and the second follows from the bound $\sigma \leq \dist(x, \cM)$. We now turn to the proof.

Now, let $x' \in \overline B_{\sigma}(x) \subseteq B_{\deltaA}(\bar x)$ and observe that \textcolor{blue}{by aiming condition~\eqref{eq:sharpnessandaiming},}
$$
\dotp{v, x' - P_{\cM}(x')} \geq \sharpc \dist(x', \cM) \qquad \text{for all $v \in \partial f(x')$.}
$$
We claim that $\dotp{v, x - P_{\cM}(x)} \geq \dtwo \dist(x, \cM)$ for all $v \in \partial f(x')$. 
Indeed, for all $v \in \partial f(x')$ we may upper bound the inner product as follows:
\begin{align*}
&\dotp{v, x' - P_{\cM}(x')} \\
&\leq \dotp{v, x- P_{\cM}(x)} + \|v\|\|x' - P_{\cM}(x') - x - P_{\cM}(x)\|\\
&\leq \dotp{v, x- P_{\cM}(x)} + L\|(I - P_{\tangent(P_{\cM}(x))})(x - x')\| + L\shape(\sigma^2+ \dist^2(x, \cM)) \\
&\leq \dotp{v, x- P_{\cM}(x)} + 3L\done \dist(x, \cM), 
\end{align*}
where the second inequality follows from the bound $\|v\| \leq L$ and Item~\ref{prop:proof:consequencesofA:item:smoothproj} of Proposition~\ref{prop:proof:consequencesofA};  and the third inequality follows from $\|x - x'\| \leq \sigma\leq \done \dist(x, \cM)$ \textcolor{blue}{and~\eqref{eq:quadraticshrinking}}. Consequently, for all $v \in \partial f(x')$, we have
\begin{align}\label{eq:onestepsub}
 \dotp{v, x - P_{\cM}(x)}  &\geq \sharpc \dist(x', \cM) -  3L\done \dist(x, \cM) \notag\\
 &\geq \sharpc\dist(x, \cM) - \sharpc\sigma -  3L\done \dist(x, \cM)\notag \\
 &\geq \sharpc(1 - \done (1 + 3L/\sharpc))\dist(x, \cM) \notag\\
 &= \dtwo \dist(x, \cM),
\end{align}
where the second inequality follows from 1-Lipschitz continuity of $\dist(\cdot , \cM)$; {\color{blue}and  the final inequality follows from the bound $\done \leq \frac{1}{2(1 + 3L/\mu)}$}. This proves the claim.

Now, fix $g \in \partial_{\sigma} f(x)$. 
By definition of $\partial_{\sigma} f(x)$, there exists a family of coefficients $\lambda_i \in [0, 1]$, points $x_i \in \overline B_{\sigma}(x) \subseteq  B_{\deltaA}(\bar x)$, and subgradients $g_i \in \partial f(x_i)$ indexed by a finite set $i \in I$ such that $\sum_{i\in I} \lambda_i = 1$ and $g =  \sum_{i\in I} \lambda_i g_i$.
Thus, by~\eqref{eq:onestepsub}, we have
$$
 \dotp{g, x - P_{\cM}(x)} = \sum_{i\in I} \lambda_i \dotp{g_i, x - P_{\cM}(x)} \geq \dtwo \dist(x, \cM).
$$
Therefore, we have 
$$
\|g\|  \geq \frac{\dotp{g, x - P_{\cM}(x)}}{\dist(x, \cM)} \geq \dtwo,
$$
as desired.
\end{proof}

In the second lemma, we provide a lower bound for $\dist(0, \partial_{\sigma} f(x))$ on the order of $\|P_{\cM}(x) - \bar x\|$, provided $\sigma = O(\|P_{\cM}(x) - \bar x\|)$.
The proof of this bound relies on quadratic growth~\ref{assum: quad} and strong $(a)$-regularity~\ref{assum: stronga}.
A consequence of the argument is that the minimal norm element of $\partial_{\sigma} f(x)$ is close to the tangent vector $\nabla_{\cM}f(P_{\cM}(x))\in \tangent(P_{\cM}(x))$. 
Later in Proposition~\ref{lem:goldsteinterminatedescent} we will also show that Algorithm~\ref{alg:tgoldstein} ($\tgoldstein$) terminates rapidly when $\sigma$ is in the region, motivating the name Tangent Descent.
We now turn to the lemma.

\begin{lem}[Lower bound on Goldstein subgradients II]\label{lem:lowerboundgoldstein}
Define 
$$
\cone := \frac{\scc}{4}; \qquad \text{ and } \qquad {\color{blue}\ctwo := \min\left\{\frac{\scc}{8\ver}, \frac{\min\{1, 1/\deltaA\}}{2}\right\}}.%;% \qquad \text{ and } \qquad \delta_{2} := \deltaA\min\left\{\frac{1}{2\ctwo}, 1\right\}.
$$
Then for all {\color{blue}$x \in B_{\deltaA}(\bar x)$} and $\sigma \geq 0$ satisfying 
$$
\max\{\dist(x, \cM), \sigma\} \leq \ctwo \|P_{\cM}(x) - \bar x\|,
$$
we have 
$$
\|P_{\tangent(P_{\cM}(x))}(g)\| \geq \cone\|P_{\cM}(x) - \bar x\| \qquad \text{for all $g \in \partial_\sigma f(x)$}.
$$
\end{lem}
\begin{proof}
{\color{blue}For the purposes of this proof, the term $1/\deltaA$ in the definition of $\ctwo$ is unnecessary; however, it will be crucial in the proof of Theorem~\ref{thm:gkl}.}
Turning to the proof, fix {\color{blue}$x \in B_{\deltaA}(\bar x)$} and $\sigma \geq 0$ satisfying the lemma assumptions. Define $y = P_{\cM}(x)$. 
Note that  
$$
\sigma \leq \ctwo\|y - \bar x\| \leq 2\ctwo \|x - \bar x\| \leq \deltaA.
$$
Thus, by~\eqref{eq:consequencestrongaeasy}, for all $g \in \partial_{\sigma} f(x)$, we have
$$
\|P_{\tangent(y)}(g - \nabla_{\cM} f(y))\| \leq \ver (\dist(x, \cM) + \sigma) \leq \frac{\scc}{4}\|y - \bar x\|.
$$
In addition,  by~\eqref{eq:uplowersmoothnessprop}, we have $\|\nabla_\cM f(y) \| \ge \frac{\scc}{2} \|y - \bar x\|.$
	Therefore, for all $g \in \partial_{\sigma} f(x)$, we have
	\begin{align*}
		\|P_{\tangent(y)}(g)\|&\ge \|\nabla_\cM f(y) \| - \ver(\dist(x, \cM)  + \sigma) \geq \frac{\scc}{4} \|y - \bar x\|,
	\end{align*}
as desired.
\end{proof}

Given these lemmata, we are now ready to establish the gradient inequality~\eqref{eq:introgkl}. 
The following theorem verifies the bound 
$$
\sigma\dist(0, \partial_{\sigma} f(x)) \geq \eta(f(x) - f(\bar x)),
$$
for some $\eta > 0$ provided $x$ is sufficiently near $\bar x$ and \textcolor{blue}{$(x,\sigma)$} lies within one of two regions, described in Item~\ref{item:gkl:1} and Item~\ref{item:gkl:2} of Theorem~\ref{thm:gkl}.
Item~\ref{item:gkl:1} and Item~\ref{item:gkl:2} roughly correspond to the regions considered in Lemma~\ref{lem:lbnormal} and Lemma~\ref{lem:lowerboundgoldstein}, respectively.
Comparing with the statement of the {\color{blue}gradient inequality}~\eqref{eq:introgkl},
we see that gradient inequality of Theorem~\ref{thm:gkl} does not require knowledge of an explicit function $\sigma(x)$. 
Instead, we need only find some $\sigma$ proportional to $\done \dist(x, \cM)$ or $\ctwo\|P_{\cM}(x) - \bar x\|$ up to a factor of, say, 2. 
Later in Proposition~\ref{prop:onestep} we show that this flexibility allows us to find an appropriate $\sigma$ through the $\linesearch$ procedure.

\begin{thm}[Gradient inequality]\label{thm:gkl}
Suppose that function $f$ satisfies Assumption~\ref{assumption:mainfinal} at $\bar x \in \RR^d$. 
%Define a radius 
%$$
%{\color{blue}\deltaGKL := \delta_1.}
%$$
For any constants $a_1 \in (0, \done ]$ and $a_2 \in (0,\ctwo]$, we have 
$$
\sigma \dist(0, \partial_{\sigma} f(x)) \geq \min\left\{\frac{\scc a_2}{8\max\{4La_2^2, \beta\}},  \frac{\sharpc a_1}{4\max\{2L, \beta/a_2^2\}}\right\} (f(x) - f(\bar x)),
$$
whenever
$x \in  B_{\deltaGKL}(\bar x)$ and $\sigma >0$ satisfy Item~\ref{item:gkl:1} or Item~\ref{item:gkl:2}:
\begin{enumerate}
\item \label{item:gkl:1}
\begin{enumerate}
\item \label{item:gkl:1a}$\frac{a_1}{2}\dist(x, \cM) \leq \sigma \leq a_1\dist(x, \cM) $; 
\item \label{item:gkl:1b}$a_2^2\|P_{\cM}(x) - \bar x\|^2 \leq \dist(x, \cM)$.
\end{enumerate}
\item \label{item:gkl:2}
\begin{enumerate}
\item \label{item:gkl:2a} $\frac{a_2}{2}\|P_{\cM}(x) - \bar x\|  \leq \sigma \leq a_2\|P_{\cM}(x) - \bar x\|$;
\item\label{item:gkl:2b} $\frac{\dist(x, \cM)}{\sigma} \leq2a_2\|P_{\cM}(x) - \bar x\|$.
\end{enumerate}
\end{enumerate}
Moreover, for any $x\in B_{\deltaGKL}(\bar x)\backslash \{ \bar x\}$, there exists $\sigma > 0$ such that Item~\ref{item:gkl:1} or Item~\ref{item:gkl:2} is satisfied. 
\end{thm}
\begin{proof}
We first show that for any $x \in B_{\deltaGKL}(\bar x)\backslash\{\bar x\}$, there exists $\sigma > 0$ such that either Item~\ref{item:gkl:1} or Item~\ref{item:gkl:2} is satisfied.
We consider two cases. First, suppose $x \in \cM$. Then Item~\ref{item:gkl:2} is trivially satisfied for $\sigma = a_2\|P_{\cM}(x) - \bar x\|$. 
Second, suppose $x \notin \cM$ and Item~\ref{item:gkl:1} cannot be satisfied for any $\sigma > 0$. 
In this case, we have 
$$
\dist(x, \cM) \leq a_2^2\|P_{\cM}(x) - \bar x\|^2 = 2a_2\sigma\|P_{\cM}(x) - \bar x\| \qquad \text{with $\sigma := a_2\|P_{\cM}(x) - \bar x\|/2$.}
$$
Thus, Item~\ref{item:gkl:2} is satisfied.

Now we prove the gradient inequality is satisfied whenever $\sigma$ satisfies Item~\ref{item:gkl:1} or Item~\ref{item:gkl:2}. 
Let us suppose that Item~\ref{item:gkl:1} holds for some $x \in B_{\deltaGKL}(\bar x)$ and $\sigma > 0.$
From~\eqref{eq:tangentnormalboundforfunction}, we have the bound:
\begin{align*}
\frac{1}{\max\{2L, \beta/a_2^2\}}(f(x) - f(\bar x)) &\leq  \frac{1}{\max\{2L, \beta/a_2^2\}}\left(L\dist(x, \cM) + \frac{\beta}{2}\|P_{\cM}(x) - \bar x\|^2\right)\\
&\leq \frac{1}{2}\left(\dist(x, \cM) + a_2^2\|P_{\cM}(x) - \bar x\|^2\right) \\
&\leq \dist(x, \cM).
\end{align*}
Now observe that the assumptions of Lemma~\ref{lem:lbnormal} are satisfied since $x \in B_{\deltaGKL}(\bar x)$,  $a_1 \leq \done $, and $x$ and $\sigma$ satisfy Item~\ref{item:gkl:1}. Therefore, we have 
$$
\sigma \dist(0, \partial_{\sigma} f(x))\geq \sigma \dtwo \geq \frac{\sharpc a_1}{4}\dist(x, \cM) \geq \frac{\sharpc a_1}{4\max\{2L, \beta/a_2^2\}}(f(x) - f(\bar x)), 
$$
as desired.

Next, let us  suppose that Item~\ref{item:gkl:2} holds for some $x \in B_{\deltaGKL}(\bar x)$ and $\sigma > 0.$
From~\eqref{eq:tangentnormalboundforfunction}, we have the bound:
\begin{align*}
\frac{1}{\max\{4La_2^2, \beta\}}(f(x) - f(\bar x))& \leq \frac{1}{\max\{4La_2^2, \beta\}}\left( L\dist(x, \cM)+  \frac{\beta}{2}\|P_{\cM}(x) - \bar x\|^2\right) \\
&\leq \frac{1}{2}\left( \frac{\dist(x, \cM)}{2a_2^2} +  \|P_{\cM}(x) - \bar x\|^2\right) \\
&\leq \frac{1}{2}\left( \frac{\dist(x, \cM)\|P_{\cM}(x) - \bar x\|}{2a_2\sigma} +  \|P_{\cM}(x) - \bar x\|^2\right) \\
&\leq \|P_{\cM}(x) - \bar x\|^2.
\end{align*}
Now observe that since $a_2 \leq \ctwo$ and $x$ and $\sigma$ satisfy Item~\ref{item:gkl:2}, {\color{blue} we have 
$$\sigma \leq \ctwo\|P_{\cM}(x) - \bar x\| \leq 2\ctwo\deltaGKL \leq (1/\deltaA)(\deltaA/4) \leq 1,$$  
where we use the bound $\ctwo \leq 1/2\deltaA$.}
Consequently, we have 
$$
\dist(x, \cM) \leq 2\sigma \ctwo\|P_{\cM}(x) - \bar x\| \leq \ctwo \|P_{\cM}(x) - \bar x\|.
$$
Therefore, $\max\{\dist(x, \cM), \sigma\} \leq \ctwo \|P_{\cM}(x) - \bar x\|$, so the conditions of Lemma~\ref{lem:lowerboundgoldstein} are satisfied (recall $\deltaGKL \leq \deltaA$).
Thus, let $g$ denote the minimal norm element of $\partial_{\sigma} f(x)$ and let us apply Lemma~\ref{lem:lowerboundgoldstein}:
$$
\dist(0,\partial_{\sigma} f(x))=\|g\| \geq \|P_{\tangent(P_{\cM}(x))}(g)\| \geq \frac{\scc}{4}\|P_{\cM}(x) - \bar x\|.
$$
Consequently, we have
$$
\sigma \dist(0, \partial_{\sigma} f(x)) \geq  \frac{\sigma\scc}{4}\|P_{\cM}(x) - \bar x\|  \geq \frac{\scc a_2}{8\max\{4La_2^2, \beta\}}(f(x) - f(\bar x)), 
$$
where the last inequality follows from $\sigma \ge \frac{a_2}{2}\|P_{\cM}(x) - \bar x\|$. This completes the proof.
\end{proof}
 
 {\color{blue}
 \begin{remark}
Note that $a_1, a_2 \in (0, 1)$ as claimed in Section~\ref{sec:normal_tangent_regions}, where we introduced the \emph{normal and tangent regions} appearing in the statement of Theorem~\ref{thm:gkl}.
 \end{remark}
 }

This concludes the proof of the gradient inequality~\eqref{eq:introgkl} under Assumption~\ref{assumption:mainfinal}. 
In Section~\ref{sec:locallinear}, we will use the gradient inequality to establish rapid local convergence of $\algname.$
Before proving that, the following section analyzes $\tgoldstein$ and $\ngoldstein$ methods.

\section{Rapid termination of $\ngoldstein$ and $\tgoldstein$ under Assumption~\ref{assumption:mainfinal}}\label{sec:minnorm}

In this section, we analyze the $\ngoldstein$ and $\tgoldstein$ methods, 
showing that both methods rapidly terminate with descent in appropriate regions.
Throughout the section, we assume that Assumption~\ref{assumption:mainfinal} is in force. 
We also use the results and notation of Proposition~\ref{prop:proof:consequencesofA}, Table~\ref{tab:constants}, Lemma~\ref{lem:lbnormal}, and Lemma~\ref{lem:lowerboundgoldstein}.

The main results of this section are Propositions~\ref{sec:analysisndescent} and \ref{lem:goldsteinterminatedescent}, which analyze $\ngoldstein$ and $\tgoldstein$, respectively. 
Proposition~\ref{sec:analysisndescent} shows that $\ngoldstein$ terminates with descent in a constant number of iterations within the region considered in Item~\ref{item:gkl:1} of Theorem~\ref{thm:gkl}. 
Proposition~\ref{lem:goldsteinterminatedescent} shows that $\tgoldstein$ either terminates with descent in $O(\log^{-1}(f(x) - f(\bar x)))$ iterations or $f(x) - f(\bar x)$ is already exponentially small in $T$ within the region considered in Item~\ref{item:gkl:2} of Theorem~\ref{thm:gkl}.
These lemmata will be the basis of our main convergence theorem -- Theorem~\ref{thm:maintheorem} -- appearing in Section~\ref{sec:locallinear}.

\subsection{Analysis of $\ngoldstein$}\label{sec:analysisndescent}

The following proposition shows that $\ngoldstein$ locally terminates in finitely many iterations whenever $\sigma$ is sufficiently small.
The result is a simple consequence of Lemmas~\ref{lemma:normalgold} and~\ref{lem:lbnormal}.
\begin{proposition}[$\ngoldstein$ loop terminates with descent]\label{lemma:normal}
Define a radius 
$$
\deltaND := {\color{blue}\min\left\{\deltaGKL, \frac{\dtwo}{\done L\sqrt{128}}\right\}.}
$$
Then for all $x \in B_{\deltaND}(\bar x)$, radii $\sigma> 0$ with $\sigma \leq \done  \dist(x, \cM)$, subgradients $g \in \partial_{\sigma} f(x)$, failure probabilities $p \in (0, 1)$ and budgets $T > 0$ satisfying 
$$
T \geq \left\lceil\frac{64L^2}{\dtwo^2} \right\rceil \left\lceil 2\log(1/p) \right\rceil,
$$
the point 
$
x_+ := \ngoldstein(x,g, \sigma, T)
$
satisfies
$$
f(x_+) \leq f(x) - \frac{\sigma \dist(0, \partial_{\sigma} f(x))}{8} \qquad \text{with probability at least $ 1- p$.}
$$
\end{proposition}
\begin{proof}
Fix $x \in B_{\deltaND}(\bar x)$ and $\sigma > 0$ satisfying the lemma assumptions.
Observe that 
$$
\sigma \leq \done \dist(x, \cM) \leq \done  \deltaND \leq \min\left\{\deltaGKL, \frac{\dtwo}{ L \sqrt{128}}\right\}, 
$$
where the final inequality follows from the bound $\done  \leq 1$; {\color{blue}see Lemma~\ref{lem:condition_number_larger}.}
Thus, by Lemma~\ref{lem:lbnormal}, we have $\dist(0, \partial_{\sigma} f(x)) \geq \dtwo$ (recall $\deltaND \leq \deltaGKL$). Consequently, 
$$ \sigma \leq \frac{\dtwo}{ L \sqrt{128}} \leq \frac{\dist(0, \partial_{\sigma} f(x))}{ L \sqrt{128}}.$$
Therefore, $\sigma$ and $T$ satisfy the assumptions of Lemma~\ref{lemma:normalgold}. 
Hence, the desired descent condition is guaranteed with probability at least $1-p$.
\end{proof}

We now turn to the analysis of the $\tgoldstein$ step.

\subsection{Analysis of $\tgoldstein$}\label{sec:analysistdescent}

{\color{blue} In this section, we analyze $\tgoldstein$, proving two main results. First, we prove Proposition~\ref{lem:goldsteinterminatedescent},
which shows that $\tgoldstein$ terminates rapidly. Second, in Lemma~\ref{lem:smallsubgeneral} we show that the trust region constraint in Line~\ref{line:trustregion} of Algorithm~\ref{alg:linesearch} ($\linesearch$) prevents long steps. Thus, once the method enters a sufficiently small neighborhood of $\bar x$, it cannot leave.}

We begin with descent  Proposition~\ref{lem:goldsteinterminatedescent}, which relies on four technical lemmata that analyze the structure of Goldstein subgradients when $\sigma$ is sufficiently small and $x$ is sufficiently near $\bar x$:
Lemma~\ref{lem: decentwithsmallnormal} states that elements of Goldstein subdifferential with small normal components are descent directions. 
Lemmas~\ref{claim: approximatereflectionintermediate_lemma} and~\ref{lem: approximatereflection} show that normalized subgradient steps approximately reflect points across the active manifold.
Lemma~\ref{lem: normalshrinklinearly} uses the approximate reflection property to show that $\tgoldstein$ geometrically decreases the normal component of the input subgradient, ensuring that we rapidly find a descent direction. We now turn to the Lemmata.

\subsubsection{Descent with small normal part}
The first lemma shows that Goldstein subgradients with small normal components are descent directions.

\begin{lem}[Descent with small normal part]\label{lem: decentwithsmallnormal}
Define  \textcolor{blue}{
$$%C_3 := \min \left\{\ctwo, \frac{\cone}{8\ver}\right\}; \qquad \text{ and } \qquad 
\cfour := \frac{\cone^2}{8L}. %\qquad \text{ and } \qquad {\color{blue}\delta_3 := \min\left\{\deltaA, \delta_2, \frac{\deltaA}{2C_3}\right\}}.
$$
}
Then for all {\color{blue}$x \in B_{\deltaA}(\bar x)$}, 
$\sigma > 0$, and $g \in \partial_\sigma f(x)\backslash\{0\}$ satisfying
\begin{enumerate}
\item $
\max\{\dist(x,\cM) , \sigma \} \le \frac{\ctwo}{4}\|P_{\cM}(x) - \bar x\|$;
\item $ \|P_{\normal(P_{\cM}(x))}(g)\| \leq \cfour \|P_{\cM}(x) - \bar x\|^2,$
\end{enumerate}
we have 
$$
f\left(x - \sigma \frac{g}{\|g\|}\right) \leq f(x) - \frac{\sigma\|g\|}{8}.
$$
\end{lem}
\begin{proof}
We begin with preliminary notation and bounds. 
We fix {\color{blue}$x \in B_{\deltaA}(\bar x)$} and subgradient $g \in \partial_{\sigma} f(x) \backslash \{0\}$. We define $y := P_{\cM}(x)$, $T := \tangent(y)$, and  $N := \normal(y)$. 
\textcolor{blue}{We observe that
$$
\sigma \leq \frac{\ctwo}{4}\|y - \bar x\| \leq \frac{\ctwo}{2}\|x- \bar x\| \leq \ctwo\deltaA \leq \deltaA,
$$
where the final inequality follows since $\ctwo \leq 1$; see Lemma~\ref{lem:lowerboundgoldstein}.
}
We now turn to the proof.

The starting point of the proof is Lebourg's mean value Theorem~\cite[Theorem 2.4]{clarke2008nonsmooth}, which ensures that there exists $v \in \partial_{\sigma} f(x)$ such that
\begin{align*}
	f\left(x- \sigma\frac{g}{\|g\|}\right) - f(x)  = \dotp{v, -\sigma\frac{g}{\|g\|}}
	 = -\frac{\sigma}{\|g\|} \dotp{v, P_T(g)}- \frac{\sigma}{\|g\|} \dotp{v, P_N(g)}.
\end{align*}
In what follows, we will show that the first term satisfies $\dotp{v, P_T(g)} \geq \frac{3}{8}\|g\|^2$, while the second term satisfies $|\dotp{v, P_N(g)}|\leq \frac{1}{8}\|g\|^2$, yielding the result.

Indeed, beginning with $|\dotp{v, P_N(g)}|$, {\color{blue} we note that 
\begin{align}\label{eq:small_normal_for_descent}
\|P_N(g)\| \leq \cfour\|P_{\cM}(x) - \bar x\|^2 \leq \frac{\cfour}{\cone^2}\|g\|^2 =\frac{1}{8L}\|g\|^2,
\end{align}
where the second inequality follow from Lemma~\ref{lem:lowerboundgoldstein}.
Consequently, we have the bound
$|\dotp{v, P_N(g)}| \le L\|P_N(g)\| \le  \frac{1}{8}\|g\|^2,
$
where we first inequality relies on the estimate $\|v\| \leq L$; see Item~\ref{prop:proof:consequencesofA:item:Lipschitzbound} of  Proposition~\ref{prop:proof:consequencesofA}.}

Next, we prove a lower bound on $\dotp{v, P_T(g)}$.  
Since $v \in \partial_{\sigma} f(x)$, {\color{blue} 
\begin{align*}
\|P_{T}(v - g)\| \leq 2\ver(\dist(x, \cM) + \sigma ) \leq \ctwo\ver\|P_{\cM}(x) - \bar x\| \leq \frac{\ctwo\ver}{\cone}\|g\| \leq \frac{1}{2}\|g\|.
\end{align*}
where the first inequality follows from~\eqref{eq:consequencestronga2}; the second by assumption; the third follows from Lemma~\ref{lem:lowerboundgoldstein}; and the fourth follows from the bound $\ctwo \leq \frac{\cone}{2\ver}$.
Therefore, 
\begin{align*}%\label{eq:stronganormalpartsmall}
 \|P_T(v) - g\| &\leq \|P_{T}(v - g)\| + \|P_N(g)\| \leq \frac{1}{2}\|g\| + \frac{1}{8L}\|g\|^2 \leq \frac{5}{8}\|g\|, 
\end{align*}
where the second inequality follows from~\eqref{eq:small_normal_for_descent} and the third follows from the bound $\|g\| \leq L$.
Consequently, we have the bound
\begin{align*}
\dotp{v, P_T(g)} &= \dotp{P_T(v), g} \geq \|g\|^2 - \|P_T(v) - g\|\|g\| \geq \frac{3}{8}\|g\|^2.
\end{align*}
 This completes the proof.}
\end{proof}

{\color{blue}Note that the proof implies a slightly stronger bound than claimed, namely that we have $f\left(x - \sigma g/\|g\|\right) \leq f(x) - \sigma\|g\|/4$. For the sake of maintaining symmetry with Proposition~\ref{lemma:normal}, however, we use the constant $1/8$ throughout.}

\subsubsection{The approximate reflection property}\label{sec:approximate_reflection}

The next two lemmata prove the approximate reflection property that was described in the introduction.
The lemmas roughly show that normalized subgradient steps approximately ``flip the sign" of the normal component of the subgradient nearby the manifold; see Section~\ref{sec:simpleexample} for more intuition. 
{\color{blue}The first lemma proves the approximate reflection property up to a tolerance depending on the distance to the manifold and $\sigma$. This lemma will be used again in the proofs of Lemma~\ref{lem: approximatereflection} and Lemma~\ref{lem: approximatereflection2}.}

{\color{blue}
\begin{lem}[Approximate reflection inequality, general case] \label{claim: approximatereflectionintermediate_lemma}
	For all $x \in B_{\deltaA/2}(\bar x), \sigma \in (0, \deltaA/2], g \in \partial_{\sigma}f(x) \backslash \{0\}$ and $\hat g \in \partial f\left(x - \sigma \frac{g}{\norm{g}}\right)$, we have
\begin{align}
&\dotp{P_{\normal(P_{\cM}(x))}(\hat g), g}\notag \\
&\hspace{20pt}\leq -\mu\|P_{\normal(P_{\cM}(x))} g\|  + \frac{(\mu + L)\|g\| \dist(x, \cM)}{\sigma}+ \frac{(\mu+L)\|g\|\shape (\dist^2(x, \cM) +\sigma^2)}{\sigma}.\label{claim: approximatereflectionintermediate}
\end{align}
\end{lem}
\begin{proof}
We begin with preliminary notation and bounds. 
We fix {\color{blue}$x \in B_{\deltaA/2}(\bar x)$} and subgradient $g \in \partial_{\sigma} f(x) \backslash \{0\}$. We define $y := P_{\cM}(x)$, $T := \tangent(y)$, and $N := \normal(y)$.
Finally, define $u := \frac{g}{\|g\|}$. % and
%\begin{align*}
%S := \dist(x, \cM) + \shape(\dist^2(x, \cM) +\sigma^2).% \leq \sigma \csix(1+ \shape(1+\delta_4))\|y - \bar x\|.\label{eq:sbound}
%\end{align*}
Note that since $x \in B_{\deltaA/2}(\bar x)$ and $\sigma \leq \deltaA/2$, we have $ x - \sigma u \in B_{\deltaA}(\bar x)$. 

Therefore, by the aiming inequality~\eqref{eq:sharpnessandaiming}, we have 
	$$
	\underbrace{\dotp{\hat g, x-\sigma u - P_\cM\paren{x- \sigma u}}}_{=:A} \ge \underbrace{\sharpc\norm{ x-\sigma u - P_\cM\paren{x- \sigma u}}}_{=: B}.
	$$
We aim to simplify this inequality with~\eqref{eq:smoothnessprojectorproof}. To that end, first note that
\begin{align}\label{eqn: x+sigmauandx}
\|(x-\sigma u - P_\cM(x-\sigma u)) - (x- P_\cM(x)- \sigma P_N(u))\| &= \|P_\cM(x-\sigma u) - P_\cM(x) +\sigma P_T(u)\|\notag\\
&\le   \shape(\dist^2(x, \cM) +\sigma^2) .
\end{align}
Consequently, we have 
\begin{align*}
A \geq B \geq \mu\|x-P_\cM(x)- \sigma P_N(u)\| - \mu \shape(\dist^2(x, \cM) +\sigma^2)  \geq \sigma\mu  \|P_N(u)\| -\mu S, %\label{eq:Blowerbound}.
\end{align*}
where $S := \dist(x, \cM) + \shape(\dist^2(x, \cM) +\sigma^2)$.
In addition, by~\eqref{eqn: x+sigmauandx} we have $$
\dotp{\hat g, (x-\sigma u - P_\cM\paren{x- \sigma u}) + \sigma P_Nu} \leq LS.
$$
Therefore, we have
\begin{align}
\dotp{\hat g, \sigma P_N(u)} &= -A + \dotp{\hat g, (x-\sigma u - P_\cM\paren{x- \sigma u}) + \sigma P_Nu} 
%&\sigma \mu\|P_N(u)\| - 
%&\leq-A + \dotp{\hat g, x-P_\cM(x)} +  \|\hat g\| \shape S \notag\\
%&\leq - A + \|\hat g\|\dist(x, \cM) +  \|\hat g\| \shape S \notag \\
%&\leq - B + L\dist(x, \cM) +  L \shape S \notag\\
\leq -\sigma\mu  \|P_N (u)\| +(\mu + L)S.\label{eq:usedagain}
\end{align}
Inequality~\eqref{claim: approximatereflectionintermediate} then follows by multiplying both sides of inequality~\eqref{eq:usedagain} by $\|g\|/\sigma$.
\end{proof}}

{\color{blue}The second lemma is an application of Lemma~\ref{claim: approximatereflectionintermediate_lemma} nearby the manifold.}

\begin{lem}[Approximate reflection inequality near the manifold]\label{lem: approximatereflection}
 Define 
$$\csix := \min\left\{\frac{\beta }{\ver(1+\deltaA)}, \frac{{\color{blue}\min\{\sharpc/\deltaA, \cfour\cfive/\beta\}}}{4(1 +  (1+\deltaA)\shape)  (\mu + L)}, {\color{blue}\frac{1}{2}}\right\}.$$
%In addition, define tradius 
%$$
%\delta_4 = \min\left\{\frac{\deltaA}{2}, \frac{\deltaA}{4\csix},\frac{\mu}{4(\mu + L)\csix(1 +  (1+\deltaA)\shape)} \right\}.
%$$
Then for all {\color{blue}$x \in B_{\deltaA/2}(\bar x)$},  $\sigma > 0$,  and $g \in \partial_\sigma f(x)\backslash\{0\}$ satisfying 
$$
\max\left\{\frac{\dist(x,\cM)}{\sigma} , \sigma \right\} \le \csix\|P_{\cM}(x) - \bar x\|, %\le 1,
$$
we have 
$$
\dotp{P_{\normal(P_{\cM}(x))}(\hat g), g} \leq - \cfive\|P_{\normal(P_{\cM}(x))}g\| + \frac{\cfour\cfive}{2}\|P_{\cM}(x) - \bar x\|^2\qquad\text{for all $\hat g \in \partial f\left(x - \sigma \frac{g}{\|g\|}\right)$.}
$$
\end{lem}
\begin{proof}
We begin with preliminary notation and bounds. 
We fix {\color{blue}$x \in B_{\deltaA/2}(\bar x)$} and subgradient $g \in \partial_{\sigma} f(x) \backslash \{0\}$. We define $y := P_{\cM}(x)$, $T := \tangent(y)$, and $N := \normal(y)$.
\textcolor{blue}{We observe that 
%\begin{align}\label{eq:yboundtdescent}
%	\|y - \bar x\| \leq 2\|x - \bar x\| \leq 2\delta_4 \leq \frac{\mu}{2(\mu + L)\csix(1 +  (1+\deltaA)\shape)}.
%\end{align}
%We also have
$$
\sigma \leq \csix\|y - \bar x\| \leq 2\csix\|x - \bar x\| \leq  \csix \delta_A\leq \deltaA/2.
$$}
{\color{blue} Finally, we have
\begin{align}
S := \dist(x, \cM) + \shape(\dist^2(x, \cM) +\sigma^2) \leq \sigma \csix(1+ \shape(1+\deltaA))\|y - \bar x\|.\label{eq:sbound}
\end{align}
where the inequality follows from the bound $\dist(x, \cM) \leq \|x - \bar x\| \leq \deltaA$.

We now apply inequality~\eqref{claim: approximatereflectionintermediate}:
\begin{align*}
\dotp{P_N\hat g, g} &\leq -\mu\|P_{N} g\|  + \frac{(\mu + L)\|g\|S}{\sigma} \\
%&\leq -\mu\|P_{N} g\|  + (\mu + L)\csix\|g\|\|y - \bar x\| + (\mu+L)\csix\shape\|g\|\|y - \bar x\|(1 +  \delta)   \\
&\leq -\mu\|P_{N} g\|  +(1 +  (1+\deltaA)\shape) (\mu + L)\csix\|g\|\|y - \bar x\|\\
&\leq -\mu\|P_{N} g\|  +(1 +  (1+\deltaA)\shape) (\mu + L)\csix(\|P_T(g)\| +\|P_N(g)\|)\|y - \bar x\|\\
&\leq -\frac{\mu}{2}\|P_{N} g\|+\frac{\cfour\cfive}{4\beta}\|P_T(g)\|\|y - \bar x\|,
 \end{align*}
where the second inequality follows from~\eqref{eq:sbound}; the third inequality follows from triangle inequality; and the fourth inequality follows from the bound 
$$
(1 +  (1+\deltaA)\shape) (\mu + L)\csix\|y - \bar x\| \leq \frac{\mu/\deltaA}{4} \cdot  (2\|x - \bar x\|) \leq \frac{\mu/\deltaA}{4}\deltaA \leq  \mu/2.
$$}
The proof will be complete if we can show that 
$$
\|P_T(g)\| \leq 2\beta\|y - \bar x\|.
$$
To that end, we have
\begin{align*}
\|P_T(g)\| \leq \ver(\dist(x, \cM) + \sigma)+  \beta \|y - \bar x\| &\leq  (\csix\ver(1+\deltaA)+ \beta)\|y - \bar x\| \leq  2\beta\|y - \bar x\|,
\end{align*}
where the first inequality follows from~\eqref{eq:consequencestronga}; the second inequality   follows from the lemma assumptions and the bound $
\dist(x, \cM) \leq \csix\sigma\|y - \bar x\| \leq \csix\deltaA\|y - \bar x\|
$; and the third inequality follows from the bounds on $\csix$. This completes the proof. 
\end{proof}

\subsubsection{The normal component shrinks geometrically}

The following lemma shows that every step of $\tgoldstein$ geometrically shrinks the normal component of the subgradient, up to a tolerance of $O(\|P_{\cM}(x) - \bar x\|^2)$.

\begin{lem}[Normal component shrinks geometrically]\label{lem: normalshrinklinearly}
 Define 
$$
\cseven:= \min \left\{\frac{\beta}{2\ver}, \frac{\cfour\cfive}{32\ver\beta},\csix, \frac{\ctwo}{4}\right\}. %\qquad \text{ and } \qquad 
%\delta_5 = \min\left\{\frac{\deltaA}{2}, \frac{\deltaA}{2\cseven}, \textcolor{blue}{\frac{1}{2\cseven}}, \delta_4\right\}.
$$
\textcolor{blue}{Then for all $x \in B_{\deltaA/2}(\bar x)$}, $\sigma >0$, $g \in \partial_\sigma f(x)\backslash\{0\}$,
 and $\hat g \in \partial f(x-\sigma\frac{g}{\|g\|})\backslash\{0\}$ satisfying 
\begin{enumerate}
\item $ \|P_{\normal(P_{\cM}(x))}g\| \ge \cfour\|P_{\cM}(x) - \bar x\|^2$;
\item $
	\max\left\{\frac{\dist(x,\cM)}{\sigma} , \sigma \right\} \le \cseven\|P_{\cM}(x) - \bar x\|,% \le 1,
	$
\end{enumerate}	
	the vector $g'$ = $\argmin_{h \in [g,\hat g]}\|h\|$ satisfies:
	$$
    \|P_{\normal(P_{\cM}(x))}(g')\|^2 \le \paren{1- \frac{3\cfive^2}{64\lipf^2}} \|P_{\normal(P_{\cM}(x))}g\|^2.
	$$
\end{lem}
\begin{proof}
We begin with preliminary notation and bounds. 
We fix {\color{blue}$x \in B_{\deltaA/2}(\bar x)$} and subgradient $g \in \partial_{\sigma} f(x) \backslash \{0\}$. We define $y := P_{\cM}(x)$, $T := \tangent(y)$, and $N := \normal(y)$.
{\color{blue}We observe two bounds. %First, we have
%$$
%\|y - \bar x\| \leq 2\|x - \bar x\| \leq 2 \delta_5 \leq \frac{1}{\cseven}.
%$$
First, we have 
$$
\sigma \leq \cseven\|y - \bar x\|\leq 2\cseven\|x - \bar x\| \leq \cseven\deltaA \leq 1.%\min\{1, \deltaA\},
$$
where the final inequality follows since $\cseven \leq \ctwo/4 \leq 1/(8\deltaA)$.
Second, we have 
\begin{align}\label{eq:c7distbound}
\dist(x, \cM) \leq \cseven \sigma\|y - \bar x\| \leq \cseven \|y - \bar x\|, 
\end{align}
since $\sigma \leq 1$.}
We now turn to the proof. 

Consider the optimal weight 
$\lambda' := \argmin_{\lambda \in [0, 1]} \|g +\lambda (\hat g - g)\|$. 
By definition we have $g' = g + \lambda' (\hat g - g)$. 
Moreover, a quick calculation shows that 
$$
\lambda' = \max\left\{\min\left\{-\frac{\dotp{g,\hat g -g}}{\|\hat g - g\|^2},1 \right\}, 0\right\}.
$$
We claim that the following bound holds on $\lambda'$:
{\color{blue}\begin{align}\label{eq:lambdasandwhich}
\underbrace{-\frac{\dotp{P_N(g), \hat g - g}}{8L^2}}_{=:\lambda_1} \leq \lambda' \leq -\underbrace{\frac{3\dotp{P_N(g), \hat g - g}}{2\|P_N(\hat g - g)\|^2}}_{=:\lambda_2}.
\end{align}}
Note that~\eqref{eq:lambdasandwhich} is an immediate consequence of the following bound:
\begin{align}\label{eq:sandwichtry1}
0 \leq -\frac{1}{2}\dotp{P_N(g), \hat g - g} \leq -\dotp{g, \hat g - g} \leq -\frac{3}{2}\dotp{P_N(g), \hat g - g} .
\end{align}
Indeed, if~\eqref{eq:sandwichtry1} holds, then 
$
\lambda' = \min\left\{-\frac{\dotp{g,\hat g -g}}{\|\hat g - g\|^2},1 \right\}.
$
Thus, we obtain the upper bound
\begin{align*}
\lambda' \leq -\frac{\dotp{g,\hat g -g}}{\|\hat g - g\|^2} \leq -\frac{3}{2}\frac{\dotp{P_N(g), \hat g - g}}{\|g - \hat g\|^2} \leq -\frac{3}{2}\frac{\dotp{P_N(g), \hat g - g}}{\|P_N(g - \hat g)\|^2} {\color{blue}= \lambda_2 .}
\end{align*}
Likewise, we obtain the lower bound 
$$
\lambda' = \min\left\{-\frac{\dotp{g,\hat g -g}}{\|\hat g - g\|^2},1 \right\} \geq \min\left\{-\frac{\dotp{g,\hat g -g}}{4L^2},1 \right\} = -\frac{\dotp{g,\hat g -g}}{4L^2} \geq -\frac{\dotp{P_N(g), \hat g - g}}{8L^2} {\color{blue}= \lambda_1,}
$$
where the first inequality follows from the bound $\|\hat g - g\|^2 \leq 2(\|\hat g\|^2 + \|g\|^2) \leq 4L^2$;  and the second equality follows from the bound $|\dotp{g, \hat g - g}| \leq \|g\|\|\hat g - g\|\leq 2L^2$. 
Thus, we now prove~\eqref{eq:sandwichtry1}. 

To that end, note that \eqref{eq:sandwichtry1} is equivalent to the following bound:
\begin{align}\label{eqn: tangentg-hatg}
		\abs{\dotp{P_T(g),\hat g - g}} \le \frac{-\dotp{P_N(g), \hat g - g}}{2}. 
	\end{align}
Therefore, we first bound $\abs{\dotp{P_T(g),\hat g - g}}$:
\begin{align*}
	\abs{\dotp{P_T(g),\hat g - g}} &\le \|P_T(g)\| \|P_T(\hat g - g)\|\\
	&\le 2\ver (\dist(x, \cM) +\sigma)(\ver (\dist(x, \cM)+\sigma) + \lipg\|y - \bar x\|)\\
	&\le 4\ver \cseven(2\ver \cseven+ \lipg )\|y - \bar x\|^2\\	
	&\le \frac{\cfour\cfive}{4} \|y - \bar x\|^2,
\end{align*}
where the second inequality follows from~\eqref{eq:consequencestronga} and~\eqref{eq:consequencestronga2};  the third inequality follows from~\eqref{eq:c7distbound} and the bound $\sigma \leq \cseven \|y - \bar x\|$; and the fourth inequality follows from the definition of $\cseven$. To complete the proof of~\eqref{eqn: tangentg-hatg}, we show that $\frac{\cfour\cfive}{4} \|y - \bar x\|^2 \leq  -\frac{1}{2}\dotp{P_N(g), \hat g - g}$: 
\begin{align}
\frac{\cfour\cfive}{2}\|y - \bar x\|^2 &\leq  \cfive\|P_N (g)\| - \frac{\cfour\cfive}{2}\|y - \bar x\|^2  \notag\\
&\leq  -\dotp{P_N(g), \hat g} \notag \\
&\leq -\dotp{P_N(g), \hat g - g}, \label{eq:lowerbounddpnormalparta}
\end{align}
where the  first inequality follows from the {\color{blue} assumption} $\frac{\cfive}{2}\|P_N(g)\| \geq \frac{\cfour\cfive}{2}\|y - \bar x\|^2$; the second inequality follows from Lemma~\ref{lem: approximatereflection} (recall $\cseven \leq \csix$ and {\color{blue}$x \in B_{\deltaA/2}(\bar x)$}); \textcolor{blue}{and the third inequality follows from $\dotp{P_N(g), g} = \|P_N(g)\|^2 \ge 0$.} 
Thus, the equivalent bounds~\eqref{eqn: tangentg-hatg} and \eqref{eq:sandwichtry1} hold. Consequently, Equation~\eqref{eq:lambdasandwhich} holds.

Now we turn to the contraction argument. 
Consider the function $r \colon \RR \rightarrow \RR$ satisfying 
$$r(\lambda) = \|P_N(g)\|^2 + 2\lambda \dotp{P_N(g), \hat g- g} + \lambda^2 \|P_N(\hat g -g)\|^2 \qquad \text{for all $\lambda \in \RR$.}
$$ 
Observe that
$$
\|P_N(g')\|^2 = \|P_N(g)\|^2 + 2\lambda' \dotp{P_N(g), \hat g- g} + (\lambda')^2 \|P_N(\hat g -g)\|^2= r(\lambda').
$$
Therefore, by convexity of $r$ and~\eqref{eq:lambdasandwhich}, we have 
$$
\|P_N(g')\|^2 = r(\lambda') \leq {\color{blue}\max_{\lambda \in [\lambda_1, \lambda_2]} r(\lambda) \leq \max\left\{r\left(\lambda_1\right), r\left(\lambda_2\right)\right\}}.
$$
To complete the proof, we show each term in the ``max" is bounded by $\left(1 - \frac{3\cfive^2}{64L^2}\right)\|P_N(g)\|^2$.

To show this, we will use the following consequence of~\eqref{eq:lowerbounddpnormalparta}:
\begin{align}
-\dotp{P_N(g), \hat g - g} \geq \cfive\|P_N (g)\| - \frac{\cfour\cfive}{2}\|y - \bar x\|^2 \geq \frac{\cfive}{2}\|P_N(g)\|,
\label{eq:lowerbounddpnormalpartb}
\end{align}
where the final inequality follows \textcolor{blue}{from the assumption} $\frac{\cfour\cfive}{2}\|y - \bar x\|^2 \leq \frac{\cfive}{2}\|P_N(g)\|$.
Indeed, first observe that 
\begin{align*}
{\color{blue}r\left(\lambda_2\right)} &= \|P_{N}(g)\|^2 - \frac{3}{4}\frac{\dotp{P_N(g), \hat g - g}^2}{\|P_N(\hat g - g)\|^2}\\
&\leq \left(1 - \frac{3\cfive^2}{16\|P_N(\hat g - g)\|^2}\right)\|P_N(g)\|^2 \\
&\leq \left(1 - \frac{3\cfive^2}{64L^2}\right)\|P_N(g)\|^2,
\end{align*}
where the first inequality from~\eqref{eq:lowerbounddpnormalpartb} and the second inequality follows from the bound $\|P_N(\hat g - g)\|^2\le \|\hat g - g\|^2 \leq 4L^2$.
Likewise, observe that
\begin{align*}
{\color{blue} r\left(\lambda_1\right)} &= \|P_N(g)\|^2 - \frac{\dotp{P_N(g), \hat g- g}^2}{4L^2} + \frac{\dotp{P_N(g), \hat g- g}^2\|P_N(\hat g - g)\|^2}{64L^4} \\
&\leq \|P_N(g)\|^2 - \frac{\dotp{P_N(g), \hat g- g}^2}{4L^2} + \frac{\dotp{P_N(g), \hat g- g}^2}{16L^2} \\
&\leq \left(1 - \frac{3\cfive^2}{64L^2}\right)\|P_N(g)\|^2, 
\end{align*}
where the first inequality follows from the bound  $\|P_N(\hat g - g)\|^2\le \|\hat g - g\|^2 \leq 4L^2$;  and the second inequality follows from~\eqref{eq:lowerbounddpnormalpartb}. Therefore, the proof is complete. \end{proof} 

\subsubsection{$\tgoldstein$ terminates with descent}

The following proposition is the main result of this section. 
It shows that $\tgoldstein$ must either terminate with descent or $f(x) - f(\bar x)$ is already exponentially small in $T$. %{\color{blue} be careful here easy to confuse $T$ and tangent space.}

\begin{proposition}[$\tgoldstein$ loop terminates with descent]\label{lem:goldsteinterminatedescent}
%Define 
%$$
%\deltaTD := \min\left\{\deltaA, \frac{1}{2\cseven}, \delta_2, \delta_3, \delta_5 \right\}.
%$$
Fix $T \in \NN$. Then for all $x \in B_{\deltaA/2}(\bar x)$, 
 \textcolor{blue}{$v \in \partial_\sigma f(x)$}, and $\sigma >0$  satisfying 
$$
\textcolor{blue}{\max\left\{\frac{\dist(x,\cM)}{\sigma} , \sigma \right\} \le   \cseven\|P_\cM(x) - \bar x\| },
$$
at least one of the following holds: 
\begin{enumerate}
\item \label{lem:goldsteinterminatedescent:smallval} we have $$
f(x) - f(\bar x) \le \frac{(\cseven^2\lipf + \lipg)\lipf}{\cfour}\left(1-\frac{3\sharpc^2}{256\lipf^2}\right)^{T/2};
$$
\item \label{lem:goldsteinterminatedescent:descent} the vector $g := \tgoldstein(x, v, \sigma,T)$ satisfies  $\|g\| > 0$ and $$f\left(x- \sigma \frac{g}{\|g\|}\right) \le f(x) - \frac{\sigma\dist(0,\partial_{\sigma} f(x))}{8}.$$
\end{enumerate}
\end{proposition}
\begin{proof}
We begin with preliminary notation and bounds. 
We fix {\color{blue}$x \in B_{\deltaA/2}(\bar x)$} and subgradient \textcolor{blue}{$v \in \partial_\sigma f(x)$}. We define $y := P_{\cM}(x)$, and $N := \normal(y)$. 
Observe that 
{\color{blue}$$
\sigma \leq \cseven \|y - \bar x\| \leq 2\cseven\|x - \bar x\| \leq \cseven\deltaA \leq 1,
$$
where the final inequality follows by definition of $\cseven \leq \ctwo/4 \leq 1/(8\deltaA).$}
In addition, since $\cseven \leq \ctwo/4$, we have $\sigma \leq (\ctwo/4)\|y - \bar x\|$ and 
$$
\dist(x, \cM) \leq \sigma \cseven\|y - \bar x\| \leq \frac{\ctwo}{4}\|y - \bar x\|,
$$
{\color{blue}where the final inequality follows from $\sigma \leq 1$.}
Consequently, 
\begin{align}\label{eq:c3boundneededfornormalpart}
\max\{\sigma, \dist(x, \cM)\} \leq \frac{\ctwo}{4} \|y - \bar x\|.
\end{align}
We now turn to the proof.

Turning to the proof, note that since {\color{blue}$x \in B_{\deltaA/2}(\bar x)$}, Lemma~\ref{lem:lowerboundgoldstein} and~\eqref{eq:c3boundneededfornormalpart} ensure that 
$$
\dist(0, \partial_{\sigma} f(x)) \geq \cone\|y- \bar x\| > 0.
$$
Thus, if $\tgoldstein(x, v, \sigma,T)$ terminates at $ t < T$, then Item~\ref{lem:goldsteinterminatedescent:descent} must hold. 
For the remainder of the proof, we suppose that $\tgoldstein(x, v, \sigma,T)$ terminates at the final iteration $t = T$ and that Item~\ref{lem:goldsteinterminatedescent:descent} does not hold.
In this case,  Lemma~\ref{lem: decentwithsmallnormal} and~\eqref{eq:c3boundneededfornormalpart} ensure that the iterates $g_t$ of $\tgoldstein(x, v, \sigma,T)$ satisfy $\|P_{N}(g_t)\| > \cfour \|y - \bar x\|^2$ for all \textcolor{blue}{$0\le t\le T-1$}. 
Therefore, since {\color{blue}$x\in B_{\deltaA/2}(\bar x)$}, $\max\left\{\dist(x,\cM)/\sigma , \sigma \right\} \le \cseven\|y - \bar x\|$, and $\|P_{N}(g_t)\| > \cfour \|y - \bar x\|^2$,  Lemma~\ref{lem: normalshrinklinearly}, yields the contraction:
	$$
	\|P_{N}(g_{t+1})\|^2 \le \paren{1- \frac{3\cfive^2}{64\lipf^2}}\|P_{N}(g_t)\|^2,\qquad  \text{for all $0\le t\le T-1$}.
	$$
Unfolding this contraction, we see that $g_T$ is an exponentially small Goldstein subgradient:
$$
\|P_{N}(g_{T})\| \leq \paren{1- \frac{3\cfive^2}{64\lipf^2}}^{T/2}  \|P_{N}(g_0)\|.
$$	
	As a result, the projection $y$ is nearby $\bar x$: 
	\begin{align}\label{eq:closeprojectiontd}
		\|y - \bar x\|^2 &\le \frac{\|P_{N}(g_{T})\| }{\cfour} \le \frac{\|P_{N}(g_0)\|}{\cfour}\paren{1- \frac{3\cfive^2}{64\lipf^2}}^{T/2} \le   \frac{ \lipf}{\cfour}\paren{1- \frac{3\cfive^2}{64\lipf^2}}^{T/2}.
	\end{align}
Consequently, 
\begin{align*}
	f(x) -f(\bar x) &\le \lipf \dist(x, \cM) + \lipg\|y - \bar x\|^2\\
	&\leq ( \cseven^2\lipf + \lipg )\|y - \bar x\|^2\\
	&\le  \frac{( \cseven^2\lipf + \lipg) L}{\cfour}\paren{1- \frac{3\cfive^2}{64\lipf^2}}^{T/2},
\end{align*}
where the first inequality follows from~\eqref{eq:tangentnormalboundforfunction} (recall {\color{blue}$x \in B_{\deltaA/2}(\bar x)$}); the second inequality follows since $\dist(x, \cM) \leq \sigma \cseven\|y - \bar x\| \leq \cseven^2\|y - \bar x\|^2$; and the third inequality follows from~\eqref{eq:closeprojectiontd}. The proof then follows from the identity $\cfive = \frac{\mu}{2}$.
\end{proof}

\subsubsection{The ``trust region" constraint prevents long steps}\label{sec:trustregion}
%\textcolor{blue}{Reviewer thinks the section number is not good since we dont have 5.2.2}
Before ending this section, we must establish one final technical result for $\tgoldstein.$
Namely, in Lemma~\ref{lem:smallsubgeneral}, we show that for appropriate $\sigma$,  $\tgoldstein$ eventually generates small subgradients on the order of $O(\|x - \bar x\|)$. 
This property is intuitive because $\dist(0, \partial_{\sigma} f(x)) = 0$ whenever $\sigma \geq \|x - \bar x\|$.
This property will help us ensure that the iterates of $\algname$ (Algorithm~\ref{alg:mainalg}) cannot leave sufficiently small neighborhoods of $\bar x$.
Indeed, since the subgradients $v_{i+1}$ generated by Algorithm~\ref{alg:linesearch} ($\linesearch$) are decreasing in norm, we will show that the trust region constraint $\sigma_i \leq \frac{\|v_{i+1}\|}{s}$ in Line~\ref{line:trustregion} of Algorithm~\ref{alg:linesearch} must eventually be violated for large $i$.
This ensures large $\sigma_i$ are never chosen.

To prove this claim, we first establish a refinement of the approximate reflection property in Lemma~\ref{lem: approximatereflection}. 
Compared to Lemma~\ref{lem: approximatereflection}, the following lemma deals with a different range of parameters. 
{\color{blue} We place the proof in Appendix~\ref{appendix:lem: approximatereflection2} as it follows from a similar line of reasoning as Lemma~\ref{lem: approximatereflection}.}
\begin{lem}[Approximate reflection across manifold, large steps]\label{lem: approximatereflection2}
Define 
$$
%\ceight := \frac{8(\sharpc + L)}{\sharpc} \qquad \text{ and }  \qquad 
\deltagrid := \min\left\{\frac{\deltaA}{2}, \textcolor{blue}{\frac{1}{\shape(\ceight + 1)}}, \frac{\mu}{8(\ver + \beta)}\right\}.
$$
Then for all $x \in B_{\deltagrid}(\bar x)$, $\sigma > 0$, and $g \in \partial_\sigma f(x)\backslash\{0\}$ satisfying 
$$
\ceight\dist(x, \cM) \leq  \sigma \leq \deltagrid, 
$$
we have 
$$
\dotp{\hat g, g} \leq - \cfive\|g\| + 2\cfive\|P_{\tangent(P_{\cM}(x))}(g)\| \qquad \text{for all $\hat g \in \partial f\left(x - \sigma \frac{g}{\|g\|}\right)$.}
$$
\end{lem}

Finally, we prove that $\tgoldstein$ eventually generates small subgradients.  
\begin{lem}[$\tgoldstein$ yields small subgradients]\label{lem:smallsubgeneral}
Fix $T \in \NN$.  Then for all $x \in B_{\deltagrid}(\bar x)$, $\sigma >0 $, and $g \in \partial_\sigma f(x)\backslash\{0\}$ satisfying 
$$
\ceight\dist(x, \cM) \leq  \sigma \leq \deltagrid, 
$$
the vector  $g' := \tgoldstein(x, g, \sigma, T)$ satisfies
$$
\|g'\| \leq \max\left\{ \left(1 - \frac{\mu^2}{64L^2} \right)^{T/2}\|g\|, 4\ver\sigma + 4(\ver + 2\beta)\|x - \bar x\|, \frac{8(f(x) - f(\bar x))}{\sigma} \right\}.
$$
\end{lem}
\begin{proof}
We begin with preliminary notation and bounds. 
We fix $x \in B_{\deltagrid}(\bar x)$ and subgradient $g \in \partial_{\sigma} f(x) \backslash \{0\}$.
We define $y := P_{\cM}(x)$ and $T := \tangent(y)$.
We also define 
$
c := \ver (\dist(x, \cM) + \sigma) + \beta\|y - \bar x\|.
$
We have the following two bounds: First, we have 
\begin{align}\label{eq:cboundfinal}
c \leq \ver(\|x - \bar x\| + \sigma) + 2\beta\|x - \bar x\| \leq \ver \sigma +  (\ver + 2\beta)\|x - \bar x\|.
\end{align}
Second, \textcolor{blue}{by~\eqref{eq:consequencestronga}}, we have
\begin{align}\label{eq:tangentvectorlooseend}
\|P_{T}(v)\| \leq c \qquad \text{for all $v\in \partial_{\sigma} f(x).$}
\end{align}
We now turn to the proof.

Note that the result holds automatically if $g' = 0$. Thus, we first consider the case where $\tgoldstein$ terminates in descent, meaning  
$$
f(x_+) - f(x) \leq -\frac{\sigma\|g'\|}{8} \qquad \text{ where $x_+ := x - \sigma \frac{g'}{\|g'\|}$.}
$$
Since $\sigma \leq \deltagrid \leq  \deltaA/2$ and $x \in B_{\deltaA/2}(\bar x)$, it follows that $x_+ \in B_{\deltaA}(\bar x)$. Thus, by Item~\ref{prop:proof:consequencesofA:item:quadgrowth} of Proposition~\ref{prop:proof:consequencesofA}, we have
$$
f(x_+) \geq f(\bar x) + \frac{\scc}{2}\|x - \bar x\|^2 \geq f(\bar x).
$$ Consequently, we have 
$$
f(\bar x) - f(x) \leq- \frac{\sigma \|g'\|}{8}.
$$
Rearranging then gives the upper bound $\|g'\| \leq \frac{8(f(x) - f(\bar x))}{\sigma},$ as desired. 

Let us now suppose that $\tgoldstein$ does not terminate with descent or with $g' = 0$.
In this case, the iterates $g_0, \ldots, g_T$ of $\tgoldstein(x, g, \sigma, T)$ exist and satisfy $g_\innerg \in \partial_{\sigma} f(x)$ for all $\innerg \leq T$. We consider two cases.

\paragraph{Case 1.} Now suppose $\|g_{\innerg}\| \le 4c$ for some $t$ satisfying $0 \leq \innerg \leq T$.
Since $\|g_\innerg\|$ is a decreasing sequence, it follows that $\|g'\| = \|g_{T}\| \leq 4c$. Recalling~\eqref{eq:cboundfinal}, yields the bound 
$$
\|g'\| \leq 4c \leq 4\ver \sigma +  4(\ver + 2\beta)\|x - \bar x\|, 
$$
as desired.

\paragraph{Case 2.}
Next suppose that for all $0\leq \innerg \leq T$ we have $4c <\|g_\innerg\|$. In this case, Lemma~\ref{lem: approximatereflection2} shows that for all $t \leq T$, we have
\begin{align}\label{eq:approxreflectiontangent}
\dotp{\hat g_\innerg, g_\innerg} \leq -\frac{\sharpc}{2}\|g_\innerg\|  + \sharpc\|P_{T}g_\innerg\| \leq -\frac{\sharpc}{2}\|g_\innerg\| + \sharpc c \leq -\frac{\sharpc}{4}\|g_\innerg\|. 
\end{align}	
We now use this bound to prove a one-step geometric improvement bound for $\|g_t\|^2$. 
To that end, fix any $\innerg \leq T-1$ and define the weight $\lambda := \frac{\sharpc\|g_\innerg\|}{16\lipf^2}$ and the vector $g_{\lambda} := g_\innerg + \lambda(\hat g_\innerg - g_\innerg)$.  Notice that $\lambda \in [0, 1]$, since
$$
\lambda = \frac{\sharpc\|g_\innerg\|}{16\lipf^2} \leq \frac{\sharpc}{16 L} \leq 1,
$$
where the first equation follows since $g_t \in \partial_{\sigma} f(x)$ and the second follows since $L \geq \mu$; {\color{blue}see Lemma~\ref{lem:condition_number_larger}.}
Thus 
\begin{align*}
\|g_{\innerg+1}\|^2 \leq \|g_{\lambda}\|^2
&=\|g_\innerg\|^2 +2\lambda\dotp{g_\innerg, \hat g_\innerg -g_\innerg}+ \lambda^2\|\hat g_\innerg - g_\innerg\|^2\\
&\leq \|g_\innerg\|^2 +2\lambda\dotp{g_\innerg, \hat g_\innerg} -  2\lambda\|g_\innerg\|^2 + 4\lipf^2\lambda^2\\
&\leq  \|g_\innerg\|^2 - \frac{\lambda\sharpc}{2}\|g_\innerg\|+ 4\lipf^2\lambda^2\\
&= \left(1 - \frac{\mu^2}{64L^2} \right)\|g_t\|^2,
\end{align*}
where the first inequality follows by definition of $g_{\innerg + 1}$;  the second inequality follows from the fact that $L$ is a local Lipschitz constant of $f$ near $\bar x$; and the third inequality follows from~\eqref{eq:approxreflectiontangent}.
Thus, to complete the proof, simply unfold this recursion to get the bound
$$
\|g'\| = \|g_{T}\| \leq \left(1 - \frac{\mu^2}{64L^2} \right)^{T/2}\|g_0\|^2,
$$
as desired.
\end{proof}

\section{Rapid local convergence of $\algname$}\label{sec:locallinear}

In this Section, we present our main convergence guarantees for the $\algname$ method under Assumption~\ref{assumption:mainfinal}. 
The main results of the section are Theorem~\ref{thm:maintheorem} and Theorem~\ref{thm:mainconvexsetting}, which analyze the nonconvex and convex settings respectively. 
In the nonconvex setting, we prove that iterates of $\algname$ locally nearly linearly converge, provided some iterate reaches a sufficiently small neighborhood of $\bar x$.
In the convex setting, we strengthen this guarantee, showing that for any initial starting point $x_0$ and any failure probability $p$, there exists some index $K_p$ after which $\algname$ nearly converges linearly with probability at least $1-p$.
Both results are a consequence of the local one-step improvement bound of Proposition~\ref{prop:onestep}. 
This proposition shows that with high probability, the following hold locally for $\linesearch$: 
its output is nearby its input; and
the function gap geometrically decreases whenever it is larger than a quantity that is exponentially small in the inner loop budget and the grid size.  
The former property will help ensure that the iterates of $\algname$ do not escape a local neighborhood of $\bar x$.

\subsection{Assumptions and notation}\label{sec:rapid_local}
Throughout this section, we assume the following assumptions and notations are in force. 
We assume that
\begin{enumerate}
\item the budget $T_k$ and grid size $G_k$ satisfy $\min\{T_k, G_k\} \geq k+1$ for all $k \geq 0$.
\item We fix an initial we an initial point $x_0 \in \RR^d$ and $g_0 \in \partial f(x_0)$. We assume that $g_0 \neq 0$. We assume Assumption~\ref{assumption:mainfinal} is in force at a point $\bar x \in \RR^d$ and use the notation of Proposition~\ref{prop:proof:consequencesofA} throughout. 
{\color{blue}We let $\{x_k\}$ denote the sequence of iterates generated by $\algname(x_0, g_0, \sscale, \{G_k\},\{T_k\})$ when applied to $f$. } %We denote the trust region parameter by {\color{blue}$s_k:= \max\{\sscale\|g_0\|, \|g_k\|\}$.}
\end{enumerate}
{\color{blue}Turning to notation, we now summarize in Table~\ref{tab:parameters} the main constants used in this section.
\begin{table}[!ht]
\centering
\renewcommand{\arraystretch}{1.5}
\begin{tabular}{@{}ll@{}}
\toprule
Parameter & Definition \\ \midrule
$\slb$ & $\sscale\|g_0\|$\\
$a_1$ & $\min\{\done , \dtwo/{\color{blue}L}\}$ \\
$a_2$ & $\frac{\min\left\{\cone/{\color{blue}L}, \cseven\right\}}{2}$ \\
$\deltaLS$ & $\min\left\{\frac{\deltaA}{2}, \deltaGKL, \deltaND,  \deltagrid, \frac{1}{2(a_1+2a_2)}, \frac{\scc D_1^2\min\{\deltagrid/2, 1/4\}^2}{2L}, 1 \right\}$ \\
$\cnine$ & $\max\left\{1,\frac{8(\ver   + 2\beta + 2\ver \ceight)}{{\color{blue}\slb}},2 \ceight, \frac{4\scc D_1}{{\color{blue}\slb}}\right\}$\\
$\epsilon_{1, T}$ & $\max\left\{\frac{(\cseven^2\lipf + \lipg)\lipf}{\cfour}\left(1-\frac{3\sharpc^2}{256\lipf^2}\right)^{T/2}, \left(1 - \frac{\mu^2}{64L^2} \right)^{T/2}L\right\}$ \\
$\epsilon_{2, G}$ & $\max\left\{\frac{L}{\min\{1, a_1\}}  + \frac{\beta}{2\min\{1,a_1\}a_2^2}, 8 \ver, {\color{blue} L}\right\}2^{-G}$ \\
$\rho$ & $1 - \frac{1}{8}\min\left\{\frac{\scc a_2}{8\max\{4La_2^2, \beta\}},  \frac{\sharpc a_1}{4\max\{2L, \beta/a_2^2\}}\right\}$ \\\bottomrule
\end{tabular}
\caption{Parameters used throughout Section~\ref{sec:locallinear}; see also Table~\ref{tab:constants}.}
\label{tab:parameters}
\end{table}

 In the following, we lower and upper bound the trust region parameter in $\linesearch$:
\begin{align}\label{eq:sbound_final}
\slb \leq \max\{\|g_k\|, \sscale\|g_0\|\} \leq L,
\end{align}
where the lower bound follows by definition, and the upper bound follows from  Part~\ref{prop:proof:consequencesofA:item:Lipschitzbound} of Proposition~\ref{prop:proof:consequencesofA}.
In addition, we apply Theorem~\ref{thm:gkl} with the constants $a_1, a_2$. These constants are derived from the parameters $\done $, $\dtwo$, $\cone$, and $\cseven$ which are defined in Lemmas~\ref{lem:lbnormal}, \ref{lem:lowerboundgoldstein}, and \ref{lem: normalshrinklinearly} respectively. We also define a neighborhood $B_{\deltaLS}(\bar x)$ for which $\linesearch$ results in geometric improvement. Here, the radius $\deltaLS$ is derived from the parameters $\deltaA$, $\deltaGKL$, $\deltaND$, $\deltagrid$,  and $\scc$ which appear in Proposition~\ref{prop:proof:consequencesofA} and Lemmas~\ref{lem:lbnormal}, \ref{lemma:normal}, \ref{lem: approximatereflection2}, and \ref{lem:smallsubgeneral}. In addition, the constant $\cnine$ will appear in an upper bound on the steplength of $\linesearch$.

We then define three terms $\epsilon_{1, T}$, $\epsilon_{2, G}$, and $\rho$ which appear in our convergence rate analysis. These terms are defined for all $T, G> 0$ and are derived from the parameters $\cseven$, $\cfour$, $a_1$, $a_2$, $\lipf$, $\lipg$, $\ver$, $\scc$, and $\sharpc$ which appear in Lemma~\ref{lem: decentwithsmallnormal}, Lemma~\ref{lem: normalshrinklinearly}, Proposition~\ref{prop:proof:consequencesofA}, and Assumption~\ref{assumption:mainfinal}.

Finally, in the following propositions, the constant $\rho \in (0, 1)$ plays the role of a local contraction factor, while the terms $\epsilon_{1,T}$ and $\epsilon_{2, G}$ are upper bounds for function gap of $\algname$.}

We now turn to the one-step improvement argument.

\subsection{One step improvement}\label{sec:onestep}

The following proposition presents our one-step improvement bound. %The bound involves the constant, which $\ceight$ appears in Lemma~\ref{lem: approximatereflection2}
\begin{proposition}[One step improvement]\label{prop:onestep}
{\color{blue}Assume the assumptions of Section~\ref{sec:rapid_local} are satisfied. Recall the notation in Table~\ref{tab:parameters}.} Then the following holds
for all $x \in B_{\deltaLS}(\bar x)$, subgradients $g \in \partial f(x)$,  and grid sizes $\gridsize > \lceil \log_2(1/\deltagrid)\rceil $:  
Fix {\color{blue}a scalar $s \in [\slb, L]$}, a failure probability $p \in (0, 1)$ and budget $T$ satisfying
$$
T \geq \left\lceil\frac{256L^2}{\mu^2} \right\rceil \left\lceil 2\log(1/p) \right\rceil. 
$$
Then with probability at least $1-p$, the point 
$
\tilde x = \linesearch(x, g, s, \gridsize, T)
$
satisfies
\begin{enumerate} 
\item 
$
f(\tilde x) - f(\bar x) \leq \max\{\rho (f(x) - f(\bar x)), \epsilon_{1, T}, \epsilon_{1, G}\}; 
$
\item $\|\tilde x - x\| \leq \cnine\max\left\{\epsilon_{1, T}/{\color{blue}\slb}, \epsilon_{2, G}/{\color{blue}\slb}, \sqrt{2(f(x) - f(\bar x))/\min\{{\color{blue}\slb}, \scc\}}\right\},$
\end{enumerate}
%$
%\cnine := \max\left\{1,\frac{8(\ver + \ver \ceight + 2\beta)}{{\color{blue}\slb}},2 \ceight, \frac{4\scc}{{\color{blue}\slb}\ceight}\right\}.
%$
\end{proposition}
\begin{proof}
We fix $x \in B_{\deltaLS}(\bar x)$, define $y := P_{\cM}(x)$, and choose a subgradient $g \in \partial f(x)$.   
Throughout we may freely use the results of Proposition~\ref{prop:proof:consequencesofA} since $\deltaLS \leq \deltaA$.
We will first establish the first item of the Proposition. 
To that end, let us assume that 
$$
f(x) - f(\bar x) > \max\{\epsilon_{1, T}, \epsilon_{2, \gridsize}\}; 
$$ otherwise the proof is trivial. 
In this case, we claim that $x$ must satisfy either Item~\ref{item:gkl:1} or Item~\ref{item:gkl:2} of Theorem~\ref{thm:gkl} for at least one $\sigma_i$ with $i \leq G-1$. 
To derive a contradiction, suppose that both items are not satisfied for $x$ with any choice of $\sigma_i$ with $i = 0, \ldots, G-1$.
We will show that neither Item~\ref{item:gkl:1b} nor its complement can be satisfied, leading to a contradiction.

Throughout the following argument, we will use the following bound: 
\begin{align*}
\max\{a_1\dist(x, \cM), a_2\|y - \bar x\|\} \leq \left(a_1 + 2a_2\right) \deltaLS \leq \frac{1}{2} = \sigma_{G-1}.
\end{align*}
Now suppose that Item~\ref{item:gkl:1b} holds, i.e., $a_2^2\|y - \bar x\|^2 \leq \dist(x, \cM)$.
Then by assumption, Item~\ref{item:gkl:1a} must fail for any $\sigma_i$. 
 We claim that this failure ensures that $\sigma_0 > a_1\dist(x, \cM)$. 
 Indeed, if $\sigma_0 \leq a_1 \dist(x, \cM)$, we must have 
$$
\sigma_0 \leq (a_1/2)\dist(x, \cM) \leq a_1\dist(x, \cM) \leq \sigma_{G-1}, 
$$
since $\sigma_0$ cannot satisfy Item~\ref{item:gkl:1a}. 
Thus, there exists some $j \leq G-1$ such that $\sigma_j = 2^{j}\sigma_0 $ satisfies  Item~\ref{item:gkl:1a}, a contradiction. 
Therefore, we have
$$
\sigma_0 > a_1\dist(x, \cM) \geq a_1a_2^2\|y - \bar x\|^2.
$$
In this case, by~\eqref{eq:tangentnormalboundforfunction}, we have 
\begin{align*}
f(x) - f(\bar x) \leq L\dist(x, \cM) + \frac{\beta}{2} \|y - \bar x\|^2 \leq \left(\frac{L}{a_1}  + \frac{\beta }{2a_2^2a_1}\right)\sigma_0 \leq \epsilon_{2, G}, 
\end{align*}
which is a contradiction.
Therefore, Item~\ref{item:gkl:1b} cannot hold, so we have $a_2^2\|y - \bar x\|^2 > \dist(x, \cM)$. 

Next, for the sake of contradiction, suppose that there exists $\sigma_i$ satisfying Item~\ref{item:gkl:2a}. 
In this case, since $\sigma_i \geq (a_2/2)\|y - \bar x\|$, we have
$$
\dist(x, \cM) < a_2^2\|y - \bar x\|^2 \leq 2a_2\sigma_i \|y - \bar x\|, 
$$
i.e., $\sigma_i$ also satisfies Item~\ref{item:gkl:2b}, which is a contradiction. 
Therefore no $\sigma_i$ satisfies Item~\ref{item:gkl:2a}. 
We claim that this ensures $\sigma_0 > a_2\|y - \bar x\|$.
 Indeed, if $\sigma_0 \leq a_2\|y - \bar x\|$, we must have 
$$
\sigma_0 \leq (a_2/2)\|y - \bar x\| \leq a_0\|y - \bar x\| \leq \sigma_{G-1}, 
$$
since $\sigma_0$ cannot satisfy Item~\ref{item:gkl:2a}. 
Thus, there exists some $j \leq G-1$ such that $\sigma_j = 2^{j}\sigma_0 $ satisfies  Item~\ref{item:gkl:2a}, a contradiction. 
Therefore, we have
$$
\sigma_0 > a_2\|y - \bar x\| \geq \sqrt{\dist(x, \cM)}.
$$
In this case, by~\eqref{eq:tangentnormalboundforfunction}, we have 
\begin{align*}
f(x) - f(\bar x) \leq L\dist(x, \cM) + \frac{\beta}{2} \|y - \bar x\|^2 \leq \left(L  + \frac{\beta }{2a_2^2}\right)\sigma_0^2 \leq \epsilon_{2, G}, 
\end{align*}
which is a contradiction. Therefore, there must exist $\sigma_i$ satisfying either Item~\ref{item:gkl:1} or Item~\ref{item:gkl:2} of Theorem~\ref{thm:gkl}.

Let us now fix a $\sigma_i$ satisfying either Item~\ref{item:gkl:1} or Item~\ref{item:gkl:2} of Theorem~\ref{thm:gkl}. 
Then, by Theorem~\ref{thm:gkl}, we have the bound 
$$
\sigma_i\dist(0, \partial_{\sigma_i} f(x)) \geq 8(1-\rho) (f(x) - f(\bar x)).
$$
In what follows, we will use the above bound to prove that with probability at least $1-p$, we have  
$
f(\tilde x) - f(\bar x) \leq \rho(f(x) - f(\bar x)) 
$
whenever $f(x) - f(\bar x) > \max\{ \epsilon_{1, T}, \epsilon_{2, \gridsize}\}$.

\paragraph{Contraction case 1: normal step.}
We first suppose that there exists $\sigma_i$ satisfying Item~\ref{item:gkl:1}.
In the interest of analyzing $v_{i+1} \in \partial_{\sigma_i}f(x)$, let us show that $x$, $\sigma_i$, and $T$ satisfy the conditions of Proposition~\ref{lemma:normal}: 
First $x\in B_{\deltaND}(\bar x)$ since $\deltaLS \leq \deltaND$. 
Second, by Item~\ref{item:gkl:1a} of Theorem~\ref{thm:gkl}, we have 
\begin{align}\label{eq:conditionsfornormalstep}
0 < \sigma_i \leq a_1\dist(x, \cM) \leq \done \dist(x, \cM).
\end{align}
Finally, from the definition $\dtwo = \mu/2$, it follows that $T$ satisfies the conditions of Proposition~\ref{lemma:normal}.
Therefore, with probability at least $1-p$, we have 
$$
f\left(x - \sigma_i \frac{v_{i+1}}{\|v_{i+1}\|}\right) - f(\bar x) \leq f(x) - f(\bar x)- \frac{\sigma_i}{8}\dist(0, \partial_{\sigma_i} f(x)) \leq \rho(f(x) - f(\bar x)).
$$

Next, we show that $v_{i+1}$ and $\sigma_i$ satisfy the trust region condition $\sigma_i \leq \frac{\|v_{i+1}\|}{s}$. 
To that end, note that the conditions of Lemma~\ref{lem:lbnormal} are met: We have $x \in B_{\deltaGKL}(\bar x)$ since $\deltaLS \leq \deltaGKL$. We also have bound $\sigma_i \leq \done \dist(x, \cM)$ from~\eqref{eq:conditionsfornormalstep}.
Therefore, it follows that the minimal norm Goldstein subgradient is lower bounded: $\dist(0, \partial_{\sigma_i} f(x)) \geq \dtwo$.
Consequently, we have 
$$
\sigma_i \leq a_1\dist(x, \cM) \leq \frac{\dtwo\deltaLS}{s} \leq \frac{\dist(0, \partial_{\sigma_i} f(x))\deltaLS}{s} \leq \frac{\|v_{i+1}\|}{s},
$$
where the second inequality follows from the definition of in Table~\ref{tab:parameters} and {\color{blue}the inequality $s \leq L$}; and the fourth inequality follows from the bound $\deltaLS \le 1$. Therefore, since the trust region constraint $\sigma_i \leq \frac{\|v_{i+1}\|}{s}$ is satisfied, the following holds with probability at least $1-p$:
$$
f(\tilde x) - f(\bar x)\leq f\left(x - \sigma_i \frac{v_{i+1}}{\|v_{i+1}\|}\right) - f(\bar x) \leq \rho(f(x) - f(\bar x)).
$$
 Thus, the first item of the proposition follows. 

\paragraph{Contraction case 2: tangent step.}
Next, we suppose that there exists $\sigma_i$ satisfying Item~\ref{item:gkl:2} of Theorem~\ref{thm:gkl}. 
In the interest of analyzing $u_{i} \in \partial_{\sigma_i} f(x)$, let us show that $x$, $\sigma_i$, and $T$ satisfy the conditions of Proposition~\ref{lem:goldsteinterminatedescent}:
 $x\in B_{\deltaA/2}(\bar x)$ since $\deltaLS \leq \deltaA/2$.
Second, by Item~\ref{item:gkl:2a} of Theorem~\ref{thm:gkl}, we have 
$$
\sigma_i \leq a_2\|y - \bar x\| \leq \cseven\|y - \bar x\|.
$$
Finally, by Item~\ref{item:gkl:2b}  of Theorem~\ref{thm:gkl}, we have $$\dist(x, \cM)/\sigma_i \leq 2a_2 \|y - \bar x\| \leq  \cseven  \|y - \bar x\|.$$
Therefore, since $f(x) - f(\bar x) > \epsilon_{1, T}$, \textcolor{blue}{Proposition~\ref{lem:goldsteinterminatedescent}} implies that 
$$
f\left(x - \sigma_i \frac{u_{i}}{\|u_{i}\|}\right) - f(\bar x)  \leq f(x) - f(\bar x )- \frac{\sigma_i}{8}\dist(0, \partial_{\sigma_i} f(x)) \leq \rho(f(x) - f(\bar x)).
$$

Next, we show that $u_{i}$ and $\sigma_i$ satisfy the trust region condition $\sigma_i \leq \frac{\|u_{i}\|}{s}$.  
To show this, we first note that $\sigma_i$ and $x$ satisfy the conditions of Lemma~\ref{lem:lowerboundgoldstein}: 
First $x\in B_{\deltaA/2}(\bar x)$ since $\deltaLS \leq \deltaA/2$. Second, by Item~\ref{item:gkl:2a} of Theorem~\ref{thm:gkl}, we have 
$$
\sigma_i \leq a_2\|y - \bar x\|\leq \ctwo\|y - \bar x\|.
$$
Finally, by Item~\ref{item:gkl:2}  of Theorem~\ref{thm:gkl}, we have
$$
\dist(x, \cM) \leq 2a_2\sigma_i\|y - \bar x\| \leq 2a_2^2\|y - \bar x\|^2  \leq \ctwo\|y - \bar x\|,
$$
where the third inequality follows from the bounds $\|y - \bar x\|\leq 2\deltaLS \leq 1/a_2$ and $a_2 \leq \ctwo/2$ (recall that $\cseven \leq \ctwo$).
Therefore, by Lemma~\ref{lem:lowerboundgoldstein} we have $\|u_{i}\| \geq \|P_{\tangent(y)}u_i\| \geq \cone\|y - \bar x\|$. 
Consequently, we have 
$$
\sigma_i \leq a_2\|y - \bar x\| \leq \frac{\cone\|y - \bar x\| }{s} \leq \frac{\|u_{i}\|}{s},
$$
\textcolor{blue}{where the second inequality follows from the definition of $a_2$ in Table~\ref{tab:parameters} and the inequality $s \leq L$.}
To complete the proof, observe that $v_{i+1} = u_i$:  since the sufficient descent condition is met, namely $f(x - \sigma_iu_i/\|u_i\|) \leq f(x) - \sigma\|u_i\|$, $\ngoldstein$ terminates at the first iteration. Therefore, we must have 
$$
f(\tilde x) - f(\bar x)\leq f\left(x - \sigma_i \frac{v_{i+1}}{\|v_{i+1}\|}\right) - f(\bar x) \leq \rho(f(x) - f(\bar x)),
$$
as desired.

Having proved the desired contraction $f(\tilde x) - f(\bar x) \leq \rho(f(x) - f(\bar x))$, we now turn to the bound on $\|\tilde x - x\|$.

\paragraph{Stepsize bound.}
We now no longer assume that $f(x) - f(\bar x) > \max\{\epsilon_{2, \gridsize}, \epsilon_{1, T}\}$.
We claim that we have 
\begin{align}\label{eq:stepsizeboundmaingoal}
\max_{0 \leq i \leq G-1}\left\{\sigma_i \colon \sigma_i \leq \|v_{i+1} \|/s\right\} \leq \cnine\max\left\{\epsilon_{1, T}/{\color{blue}\slb}, \epsilon_{2, G}/{\color{blue}\slb}, \sqrt{2(f(x) - f(\bar x))/\min\{{\color{blue}\slb}, \scc\}}\right\}.
\end{align}
Note that inequality~\eqref{eq:stepsizeboundmaingoal} immediately yields the second item of the proposition, since
$$
\|\tilde x - x\| \leq \max_{0 \leq i \leq G-1}\left\{\sigma_i \colon \sigma_i \leq \|v_{i+1} \|/s\right\}.
$$
To prove~\eqref{eq:stepsizeboundmaingoal}, we will apply Lemma~\ref{lem:smallsubgeneral}. 

To that end, first note that $x \in B_{\deltagrid}(\bar x)$ since $\deltaLS \leq \deltagrid.$
Next, we verify that there exists an index $i$ such that $\sigma_i$ satisfies a slightly stronger version of the assumptions of Lemma~\ref{lem:smallsubgeneral}.
Indeed, recall that by the quadratic growth condition~\ref{assum: quad}, we have the bound 
\begin{equation}\label{eqn: disttosolbound}
\dist(x, \cM) \leq \|x - \bar x\| \leq \sqrt{2(f(x) - f(\bar x))/\scc}.
\end{equation}

Thus, to satisfy the assumptions of Lemma~\ref{lem:smallsubgeneral}, we prove that there exists $i$ such that 
\begin{align}\label{eq:goalbound}
R_x:= \ceight\sqrt{2(f(x) - f(\bar x))/\scc} \leq \sigma_i \leq \deltagrid.
\end{align}
Indeed, first notice that 
$
\sigma_0 \leq \deltagrid
$
since $G \geq \lceil \log_2(1/\deltagrid) \rceil$.
Thus, if $\sigma_0 \geq R_x,$ the bound~\eqref{eq:goalbound} holds for $\sigma_0$.
If instead $\sigma_0 < R_x$, we have 
$$
\sigma_0 < R_x\leq \ceight\sqrt{2L\deltaLS/\scc} \leq \min\{\deltagrid/2, 1/4\}\leq \min\{\deltagrid, 1/2\} \leq 1/2 = \sigma_{G-1},
$$
where the second inequality follows since $\|x - \bar x\| < \deltaLS$ and $f$ is $L-$Lipschitz continuous on $B_{\deltaLS}(\bar x)$; and the third inequality follows since $\deltaLS \leq \scc D_1^2\min\{\deltagrid/2, 1/4\}^2/(2L)$.
Thus, there exists $i$ such that $\sigma_i \in [ \min\{\deltagrid/2, 1/4\},  \min\{\deltagrid, 1/2\}]$. 
Since  $ \min\{\deltagrid/2, 1/4\} \geq R_x$, inequality~\eqref{eq:goalbound} follows.

Now let $i_\ast$ be the minimal such index such that~\eqref{eq:goalbound} is satisfied for $i = i_\ast$. 
If $i_\ast \neq 0$, the bound $\sigma_{i_\ast - 1} \leq R_x$ holds. 
In particular, $\sigma_{i_\ast} \leq 2R_x$. Therefore, considering the cases $i_\ast = 0$ and $i_\ast \neq 0$ separately, we have
\begin{equation}\label{eqn: sigmaistarupperbound}
R_x \leq \sigma_{i_\ast} \leq \max\left\{\sigma_0, 2R_x\right\}.
\end{equation}
Now we bound the step length $\|x - \tilde x\|$ by considering two cases. 

First suppose that $\sigma_{i_\ast} > \|u_{i_\ast}\|/s$. In this case,~\eqref{eq:decreasingnormsub} ensures $\sigma_{i_\ast} > \|v_{i_\ast + 1}\|/s$.
Then, since $\sigma_i$ is increasing in $i$, we have 
\begin{align*}
\max_{0 \leq i \leq G-1}\left\{\sigma_i \colon \sigma_i \leq \|v_{i+1} \|/s\right\} \leq \sigma_{i_\ast} &\leq  \max\left\{\sigma_0, 2R_x\right\} \\%{\color{blue} \text{don't divide by $s$ on the next line.}}\\
 &\leq \cnine\max\left\{\epsilon_{2, G}/{\color{blue}\slb},\sqrt{2(f(x) - f(\bar x))/\min\{{\color{blue}\slb}, \scc\}}\right\}.\\
 &\leq \cnine\max\left\{\epsilon_{1, T}/{\color{blue}\slb}, \epsilon_{2, G}/{\color{blue}\slb}, \sqrt{2(f(x) - f(\bar x))/\min\{{\color{blue}\slb}, \scc\}}\right\}, 
\end{align*}
which verifies~\eqref{eq:stepsizeboundmaingoal}.
We now consider the alternative case.

Next suppose that $\sigma_{i_\ast} \leq \|u_{i_\ast}\|/s$. 
We consider two subcases. 
First suppose that the following bound also holds:
\begin{align}\label{eq:descentboundgap}
\|u_{i_\ast}\| \leq \frac{8(f(x) - f(\bar x))}{\sigma_{i_\ast}}.
\end{align}
Then, since $\sigma_{i_\ast} \geq R_x$, we have 
$$
\|u_{i_\ast}\| \leq \sqrt{32\scc D_1^2 (f(x) - f(\bar x))}.
$$
Second, suppose that~\eqref{eq:descentboundgap} does not hold. 
Let us apply Lemma~\ref{lem:smallsubgeneral} \textcolor{blue}{to $\sigma = \sigma_{i_\ast}$}:
\textcolor{blue}{
\begin{align*}
\|u_{i_\ast}\| &\leq \max\left\{ \left(1 - \frac{\mu^2}{64L^2} \right)^{T/2}L, 4 \ver \max\{\sigma_0, 2R_x\}+ 4(\ver + 2\beta)\|x - \bar x\| \right\}\\
%&\leq \max\left\{ \left(1 - \frac{\mu^2}{64L^2} \right)^{T/2}L, 4\ver\sigma_0 +4(\ver + 2\beta) ,4(\ver + 2\ceight\verL, 8 \ver\sigma_0, 8(\ver + 2\beta)\|x - \bar x\|+8\ver R + 2\beta)\|x - \bar x\| \right\} \\
&\leq \max\left\{ \left(1 - \frac{\mu^2}{64L^2} \right)^{T/2},  8 \ver\sigma_0,  8(\ver + 2\beta)\|x - \bar x\| + 16\ver R_x \right\} \\
&\leq \max\left\{ 1\cdot \epsilon_{1, T}, 1 \cdot \epsilon_{2, G}, 8(\ver   + 2\beta + 2\ver \ceight) \sqrt{2(f(x) - f(\bar x))/\gamma} \right\} \\
&\leq s\cnine\max\{\epsilon_{1,T}/\slb, \epsilon_{2,G}/\slb,\sqrt{2(f(x) - f(\bar x))/\gamma}\},
\end{align*}
where first inequality follows from Lemma~\ref{lem:smallsubgeneral} and bound~\eqref{eqn: sigmaistarupperbound}; the second inequality follows from the bound: $\max\{a, b\} + c \leq a + b + c \leq 2\max\{a, b + c\}$ for all $a, b, c \geq 0$; %considering cases $\ver \sigma_0 \le (\ver+2\beta)\|x-\bar x\|$ and  $\ver \sigma_0 >(\ver+2\beta)\|x-\bar x\|$; 
 the third inequality follows by definition of $\epsilon_{1, T}$, $\epsilon_{2, G}$ and $R_x$, and~\eqref{eqn: disttosolbound}; and the last inequality follows since $\cnine \geq \max\{1,8(\ver   + 2\beta + 2\ver \ceight)/\slb\}$.
}

Therefore, as long as $\sigma_{i_\ast} \leq \|u_{i_\ast}\|/s$, we have 
\begin{align*}
\|u_{i_\ast}\|/s  &\leq \max\left\{\cnine\max\{\epsilon_{1,T}/\slb, \epsilon_{2,G}/\slb, \sqrt{2(f(x) - f(\bar x))/\gamma}\}, \sqrt{32\scc D_1^2(f(x) - f(\bar x))/\slb^2}\right\} \\
&\leq  \cnine\max\left\{\epsilon_{1, T}/\slb, \epsilon_{2, G}/\slb,  \sqrt{2(f(x) - f(\bar x))/\gamma}\right\} \\
&\leq \cnine\max\left\{\epsilon_{1, T}/\slb, \epsilon_{2, G}/\slb, \sqrt{2(f(x) - f(\bar x))/\min\{\slb, \scc\}}\right\},
\end{align*}
where second inequality follows from the bound $\cnine \geq 4\gamma D_1/(\slb)$; and the third inequality follows from the bound~\eqref{eqn: disttosolbound}.
To complete the proof of~\eqref{eq:stepsizeboundmaingoal}, recall that by~\eqref{eq:decreasingnormsub}, for all $j > i_\ast$, we have $\|v_{j}\| \leq \|u_{i_\ast}\|$. 
Consequently,  
\begin{align*}
&\max_{0 \leq i \leq G-1}\left\{\sigma_i \colon \sigma_i \leq \|v_{i+1} \|/s\right\} \\
&\leq \max\left\{\sigma_{i_\ast}, \|u_{i_\ast}\|/s\right\}\\
&=\|u_{i_\ast}\|/s\\
&\leq \cnine\max\left\{\epsilon_{1, T}/{\color{blue}\slb}, \epsilon_{2, G}/{\color{blue}\slb}, \sqrt{2(f(x) - f(\bar x))/\min\{{\color{blue}\slb}, \scc\}}\right\},
\end{align*}
which verifies~\eqref{eq:stepsizeboundmaingoal}.
\end{proof}

\subsection{Main convergence theorems}\label{sec:mainconvergencetheorem}

We are now ready to prove the main results of this work. 
The goal of this section is to prove that an event of the following form occurs with high probability.
\begin{definition}[$E_{k_0, q, C}$]\label{def:Event}
{\rm
For any $k_0 > 0$, $q \in (0, 1)$ and $C > 0$, let $E_{k_0, q, C}$ denote the event that for all $k \geq k_0$, we have the following two bounds:
\begin{align*}
f(x_{k}) - f(\bar x) &\leq \max\{(f(x_{k_0}) - f(\bar x))q^{k - k_0}, Cq^{k}\}; \\
\|x_k - \bar x\|^2 &\leq  \frac{2}{\scc}\max\{(f(x_{k_0}) - f(\bar x))q^{k - k_0}, Cq^{k}\}. 
\end{align*}}
\end{definition}
We will lower bound the probability of the event $E_{k_0, q, C}$ in both nonconvex and convex settings for a particular choice of $k_0,$ $q$, and $C$.
In the nonconvex setting, our result will lower bound the conditional probability of $E_{k_0, q, C}$, given that iterate $x_{k_0}$ enters a sufficiently small neighborhood of $\bar x$.
To prove the result, we will simply iterate the one-step improvement bound of Proposition~\ref{prop:onestep}.
In the convex setting, we will lower bound the unconditional probability of $E_{k_0, q, C}$. 
To prove this result, we will combine the conditional result with the sublinear convergence guarantee of Theorem~\ref{thm:sublinear}.

Before turning to the proofs, we introduce the main parameters that are common to both the nonconvex and convex settings.
\begin{table}[ht]
\centering
\renewcommand{\arraystretch}{1.5}
\begin{tabular}{@{}ll@{}}
\toprule
Parameter & Definition \\
\midrule
$C'$ & $\frac{2048L^2}{\mu^2}$ \\
$C$ & $\max\left\{\frac{(\cseven^2\lipf + \lipg)\lipf}{\cfour}, L, \frac{L}{\min\{1, a_1\}}  + \frac{\beta}{2\min\{1,a_1\}a_2^2}, 8 \ver \right\}$ \\
$q$ & $\max\left\{\rho, \sqrt{1-\frac{3\sharpc^2}{256\lipf^2}}, \frac{1}{2}\right\}$ \\
$\deltaNCVX$ & $\min\left\{ \frac{\deltaLS}{4}, \frac{\deltaLS^2 \min\{\slb, \scc\}(1-q^{1/2})^2}{32L\cnine^2}, \frac{\deltaLS \slb(1-q)}{4L\cnine} \right\}$ \\
$K_0$ & $\left\lceil  \max\left\{\log_q\left(\frac{\deltaLS^2 \min\{\slb, \scc\}(1-q^{1/2})^2}{32C\cnine^2}\right),\log_q\left(\frac{\deltaLS \slb(1-q)}{4C\cnine}\right), \log_2\left(\frac{1}{\deltagrid}\right) \right\} \right\rceil$ \\
\bottomrule
\end{tabular}
\caption{Parameters used throughout Section~\ref{sec:mainconvergencetheorem}; see also Tables~\ref{tab:constants} and~\ref{tab:parameters}.}\label{table:main_convergence_theorems}
\end{table}

\subsubsection{The nonconvex setting}
The following theorem is our main convergence theorem in the nonconvex setting. %We recall that $C_8$ is a constant appearing in Proposition~\ref{prop:onestep}.

\begin{thm}[Main Theorem: Nonconvex Setting]\label{thm:maintheorem}
{\color{blue} Assume the assumptions outlined at the start of Section~\ref{sec:locallinear} are satisfied.
Recall the notation of Table~\ref{table:main_convergence_theorems}.} 
%Define constants 
%$$
%C' := \frac{2048L^2}{\mu^2}; \quad  \text{ and } \quad C := \max\left\{\frac{(\cseven^2\lipf + \lipg)\lipf}{\cfour}, L, \frac{L}{\min\{1, a_1\}}  + \frac{\beta}{2\min\{1,a_1\}a_2^2}, 8 \ver\right\}.
%$$
%Define the contraction factor and radius 
%$$
%q := \max\left\{\rho, \sqrt{1-\frac{3\sharpc^2}{256\lipf^2}}, \frac{1}{2}\right\};   \text{ and } \deltaNCVX := \min\left\{ \frac{\deltaLS}{4}, \frac{\deltaLS^2 \min\{{\color{blue}\slb}, \scc\}(1-q^{1/2})^2}{32L\cnine^2}, \frac{\deltaLS {\color{blue}\slb}(1-q)}{4L\cnine} \right\}.
%$$
%Define the index
%$$
%K_0 :=  \left\lceil  \max\left\{\log_q\left(\frac{\deltaLS^2 \min\{{\color{blue}\slb}, \scc\}(1-q^{1/2})^2}{32C\cnine^2}\right),\log_q\left(\frac{\deltaLS {\color{blue}\slb}(1-q)}{4C\cnine}\right), \log_2\left(\frac{1}{\deltagrid}\right) \right\} \right\rceil.
%$$
{\color{blue} Fix a failure probability $p \in (0, 1)$ and an index $k_0 \geq \max\left\{K_0, C'\log\left(C'/p\right)\right\}$. Suppose $P(x_{k_0} \in B_{\deltaNCVX }(\bar x)) > 0$. Then,}
$$
P(E_{k_0, q, C} \mid x_{k_0} \in B_{\deltaNCVX }(\bar x)) \geq 1-p.% \qquad \text{ for all $k_0 \geq \max\left\{K_0, C'\log\left(\frac{C'}{p}\right)\right\}$.}
$$
\end{thm}
\begin{proof}
We begin with preliminary notation and bounds.
Fix $k_0 \geq \max\{K_0, C'\log(C'/p)\}$ and for all $k \geq k_0$, define the quantity 
$$
R_{k} := \max\{(f(x_{k_0}) - f(\bar x))q^{k - k_0}, Cq^{k}\}.
$$
Note that whenever $x_{k_0} \in B_{\deltaNCVX }(\bar x)$ we have the bound 
\begin{align}\label{eq:Lipschitzboundfinal}
R_k \leq \max\{L\deltaNCVX q^{k - k_0}, Cq^{k}\}, 
\end{align}
since $f$ is $L$-Lipschitz continuous on $B_{\delta}(\bar x)$.

Next, we prove that 
\begin{align}\label{eq:epsboundneededforfinal}
\max\{\epsilon_{1, T_k}, \epsilon_{2, G_k}\} \leq R_{k+1} \qquad \text{for all $k \geq 0$.}
\end{align}
Indeed, beginning with $\epsilon_{1, T_k}$, we have
\begin{align*}
\epsilon_{1, T_k} &= \max\left\{\frac{(\cseven^2\lipf + \lipg)\lipf}{\cfour}\left(1-\frac{3\sharpc^2}{256\lipf^2}\right)^{T_k/2}, \left(1 - \frac{\mu^2}{64L^2} \right)^{T_k/2}L\right\} \\
&\leq C \max\left\{\left(1-\frac{3\sharpc^2}{256\lipf^2}\right)^{\frac{T_k }{2}}, \left(1 - \frac{\mu^2}{64L^2} \right)^{\frac{T_k }{2}}\right\} \\
&\leq Cq^{k+1} \leq R_{k+1}, 
\end{align*}
where the first and second inequalities follow from the definitions of $C$ and $q$ together with the lower bound $T_k \geq k+1$. Turning to $\epsilon_{2, G_k}$, we have
\begin{align*}
\epsilon_{2, G_k} =  \max\left\{\frac{L}{\min\{1, a_1\}}  + \frac{\beta}{2\min\{1,a_1\}a_2^2}, 8 \ver, {\color{blue}L}\right\}2^{-G_k } 
\leq C 2^{-G_k} \leq Cq^{k+1} \leq R_{k+1}, 
\end{align*}
where the first and second inequalities follow from the definition of $C$ and $q$ together with the lower bound $G_k \geq k+1$. Thus~\eqref{eq:epsboundneededforfinal} holds.

Finally,  we analyze the quantity   
$$
D_{k_0, \deltaNCVX } :=  \sum_{k=k_0}^{\infty} \cnine\max\left\{\sqrt{2R_k/\scc' },R_{k+1}/{\color{blue}\slb}\right\}  \qquad \text{ where $\scc' := \min\{{\color{blue}\slb}, \gamma\}$.}
$$
We claim in particular that 
\begin{align}\label{eq:stayinneighborhodprebound}
D_{k_0, \deltaNCVX} + \deltaNCVX  \leq \deltaLS/2.
\end{align}
Since $\deltaNCVX  \leq \deltaLS/4$, it suffices to prove $D_{k_0, \deltaNCVX } \leq \deltaLS/4.$ 
To that end, we have
\begin{align*}
D_{k_0, \deltaNCVX } &= \sum_{k=k_0}^{\infty} \cnine\max\left\{\sqrt{2R_k/\scc' },R_{k+1}/{\color{blue}\slb}\right\} \\
&\leq \sum_{k=k_0}^{\infty} \cnine \max\left\{\sqrt{2\max\{L\deltaNCVX q^{k - k_0}, Cq^{k}\}/\scc'}, \max\{L\deltaNCVX q^{k - k_0}/\scc', Cq^{k}\}/{\color{blue}\slb}\right\} \\
&\leq \cnine\max\left\{\frac{\sqrt{2L\deltaNCVX }}{\sqrt{\scc'}(1-q^{1/2})}, \frac{\sqrt{2Cq^{k_0}}}{\sqrt{\scc'}(1 - q^{1/2})}, \frac{L\deltaNCVX }{{\color{blue}\slb}(1-q)}, \frac{Cq^{k_0}}{{\color{blue}\slb}(1 - q)}\right\}\\
&\leq \frac{\deltaLS}{4}, 
\end{align*}
where the first inequality follows from the bounds~\eqref{eq:Lipschitzboundfinal} and the bound $R_{k+1} \leq R_k$; the second inequality follows by summing the infinite series; and the third inequality follows from the definitions of $K_0$ and $\deltaNCVX $ together with the bound $k_0 \geq K_0$ . 
This proves~\eqref{eq:stayinneighborhodprebound}.

We now turn to the proof. Consider the following sequence defined for all $k \geq k_0$: 
$$
b_{k} := 
\deltaNCVX  + \sum_{j=k_0}^{k-1} \cnine\max\left\{\sqrt{2R_k/\scc' },R_{k+1}/{\color{blue}\slb}\right\}.  
$$
Note that~\eqref{eq:stayinneighborhodprebound} ensures that $b_k \leq \deltaLS/2$ for all $k \geq k_0$.
Now, define the event  
$$
F_{k_0} := \{x_{k_0} \in B_{\deltaNCVX }(\bar x)\}.
$$
In addition, define the following decreasing sequence of events 
$$
A_k := \bigcap_{j=k_0}^k\left\{f(x_{j}) - f(\bar x) \leq R_j \text{ and }  \|x_j - \bar x\| \leq b_j \right\}.
$$
We claim that
\begin{align}\label{eq:inductiveprobability}
P(A_{k+1}\mid A_k\cap F_{k_0}) \geq 1- \exp(-T_k/C') \qquad \text{ for all $k \geq k_0$}.
\end{align}
Indeed, Proposition~\ref{prop:onestep} implies that conditioned on $A_k\cap F_{k_0}$, the following four inequalities are satisfied with probability at least $1- \exp(-T_k/C')$:
\begin{enumerate}
\item $f(x_k) - f(\bar x) \leq R_k$
\item $\|x_k - \bar x\| \leq b_k$; 
\item $\|x_{k+1} - x_k\| \leq \cnine\max\left\{\epsilon_{1, T_k}/{\color{blue}\slb}, \epsilon_{2, G_k}/{\color{blue}\slb}, \sqrt{2(f(x_k) - f(\bar x))/\scc'}\right\}$; 
\item $f(x_{k+1}) - f(\bar x) \leq \max\{\rho(f(x_{k}) - f(\bar x)),\epsilon_{1, T_k}, \epsilon_{2, G_k}\}$.
\end{enumerate}
{\color{blue}(Note that in applying the Proposition~\ref{prop:onestep}, we use the scalar $s = \max\{\|g_k\|, \sscale\|g_0\|\}$ and the inclusion $s \in [\slb, L]$, which was proved~\eqref{eq:sbound_final}.)}
Thus, the bound~\eqref{eq:inductiveprobability} will follow by induction if we can prove that whenever the above four inequalities hold, we have $\|x_{k+1} - \bar x\| \leq b_{k+1}$ and $f(x_{k+1}) - f(\bar x) \leq  R_{k+1}$.

To that end, we first prove $\|x_{k+1} - \bar x\| \leq b_{k+1}$. 
Indeed,
\begin{align}
\|x_{k+1} - \bar x\| \notag &\leq  \|x_{k+1} - x_k\| + \|x_k - \bar x\|\notag \\
&\leq \cnine\max\left\{\epsilon_{1, T_k}/{\color{blue}\slb}, \epsilon_{2, G_k}/{\color{blue}\slb}, \sqrt{2(f(x_k) - f(\bar x))/\scc'}\right\} +  b_{k}\notag \\
&\leq \cnine\max\left\{\sqrt{2R_k/\scc'},R_{k+1}/{\color{blue}\slb}\right\}   +  b_{k} = b_{k+1},
\end{align}
\textcolor{blue}{where the second inequality follows from Proposition~\ref{prop:onestep}; and the third inequality follows from the bound~\eqref{eq:epsboundneededforfinal}.}
Next, we prove the bound on $f(x_{k+1}) - f(\bar x) \leq  R_{k+1}$. Indeed,  
\begin{align*}
f(x_{k+1}) - f(\bar x) &\leq \max\{\rho(f(x_{k}) - f(\bar x)),\epsilon_{1, T_k}, \epsilon_{2, G_k}\} \\
&\leq \max\{\rho \max\{(f(x_{k_0}) - f(\bar x))q^{k - k_0}, Cq^{k}\} , \epsilon_{1, T_k}, \epsilon_{2, G_k}\}\\
&\leq  R_{k+1}, 
\end{align*}
where the final inequality follows from~\eqref{eq:epsboundneededforfinal} and the bound $\rho \leq q$.
Consequently, the bound~\eqref{eq:inductiveprobability} holds. Moreover, due to the bound $T_k \geq k+1$, we have 
\begin{align}\label{eq:conclusionofprop}
P(A_{k+1}\mid A_k\cap F_{k_0}) \geq 1- \exp(-T_k/C') \geq 1- \exp(-(k+1)/C').
\end{align}

Now we relate $A_k$ to $E_{k_0, q, C}$.  
To that end, by the conditional law of total probability, for all $k \geq k_0$, we have
\begin{align*}
P(A_{k+1} \mid F_{k_0}) \geq P(A_{k+1} \mid A_{k} \cap F_{k_0})P(A_{k} \mid F_{k_0}) \geq P(A_{k} \mid F_{k_0})- \exp(-(k+1)/C').
\end{align*}
Therefore, for all $k \geq k_0$, we have 
$$
P(A_{k} \mid F_{k_0}) \geq P(A_{k_0} \mid F_{k_0}) - \sum_{j=k_0+1}^\infty \exp(-j/C') = 1 -  \frac{\exp(-\frac{k_0+1}{C'})}{1 - \exp(-\frac{1}{C'})} \geq 1- p , 
$$
where the equality follows since $P(A_{k_0} \mid F_{k_0}) = 1$; and the final inequality follows by definition of $k_0 \geq C'\log(C'/p)$. 
Now recall that $\sup_{k \geq k_0} b_k \leq \deltaLS/2$. 
Therefore, defining the event 
$$E_{k_0, q, C}' := \{f(x_k) - f(\bar x) \leq R_k \text{ for all $k \geq k_0$ and } x_k \in B_{\deltaLS}(\bar x)\},$$ 
 we have  
$$
P(E_{k_0, q, C}' \mid F_{k_0}) \geq \lim_{k \rightarrow \infty} P(A_{k} \mid F_{k_0}) \geq 1 -  p.
$$
Next, recall that since $\deltaLS \leq \deltaA$,   the quadratic growth bound~\ref{assum: quad} 
$$\|x_k - \bar x \|^2 \leq \frac{2}{\scc} (f(x_k) - f(\bar x)) \leq \frac{2}{\scc}R_{k} 
$$
holds for every $k \geq k_0$  within the event $E_{k_0, q, C}'$. Thus, $E_{k_0, q, C} \supseteq E_{k_0, q, C}'$. 
Therefore, we have 
$$
P(E_{k_0, q, C} \mid F_{k_0}) \geq P(E_{k_0, q, C}' \mid F_{k_0}) \geq 1 -  p,
$$
as desired.
\end{proof}

\subsubsection{The convex setting}

Now we turn to the convex setting. 
Our goal is to prove a lower bound on $P(E_{k_0, q, C})$ for $q$ and $C$ chosen as in Table~\ref{table:main_convergence_theorems} and all sufficiently large $k_0$.
Before stating the result, we recall a simple fact about convex functions satisfying Assumption~\ref{assumption:mainfinal}. {\color{blue}A similar result appears in~\cite[Section 2.4]{bolte2017error}, but for completeness we provide a proof in Appendix~\ref{sec:proof:sec:getinneighborhood}.}
\begin{lem}\label{sec:getinneighborhood}
In addition to the assumption set out at the start of the section, suppose that function $f$ is convex. 
Then for all $ a > 0$, we have 
$$
\{x \in \RR^d \colon f(x) - f(\bar x) \leq a\} \subseteq \overline B_{r_a}(\bar x) \qquad \text{ where } r_a := \max\left\{\frac{2a}{\scc\delta_A}, \sqrt{\frac{2a}{\scc}} \right\}.
$$
In particular, $f$ has bounded sublevel sets. 
\end{lem}

We now turn to our main theorem. 
\begin{thm}[Main Theorem: Convex setting]\label{thm:mainconvexsetting}
{\color{blue} Assume the assumptions of Section~\ref{sec:rapid_local} are satisfied. Recall the notation of Table~\ref{table:main_convergence_theorems}. In addition,} suppose that function $f$ is convex.
Consider the bounded set 
$$
S: = \{x+ u \colon f(x) \leq f(x_0) \text{ and } u \in \overline B(x)\}.
$$
Let $L'$ be a Lipschitz constant of $f$ on $S$.  Define the constants
$$
a := \min\left\{\frac{\scc\deltaA\deltaNCVX}{4}, \frac{\scc\deltaNCVX^2}{8}\right\}; \qquad\text{ and } \qquad b := \inf_{\alpha\in (0, 1)} \frac{\left(64L'\sqrt{\frac{2}{\alpha}}\right)^{\frac{2}{(1-\alpha)}}}{\left(\frac{a}{\mathrm{diam}(S)}\right)^{\frac{2\alpha}{(1-\alpha)}}}.
$$
Finally, define
$$
K_1 := \left\lceil \frac{4\mathrm{diam}^2(S)}{a^2}\min\left\{16^2(f(x_0) - \inf f)^2,\frac{b}{4}, 2048 L'^2\log\left(\frac{2}{p}\right), 128(L')^2\right\} + \frac{(4L')^2}{a^2}\right\rceil.
$$
Then, for every failure probability $p \in (0, 1)$, we have
$$
P(E_{k_0, q, C}) \geq 1 - p \qquad \text{for all $k_0 \geq \max\left\{K_0, C'\log\left(\frac{2C'}{p}\right), 2K_1 - 1\right\}$.}
$$
%where $q, C,$ and $C'$ are defined as in Table~\ref{eq:main_convergence}.
\end{thm}
\begin{proof}
Theorem~\ref{thm:maintheorem} shows that 
\begin{align}\label{eq:dimdep1}
P(E_{k_0, q, C} \mid x_{k_0} \in  B_{\deltaNCVX}(\bar x)) \geq 1 - p/2 \qquad \text{ for all $k_0 \geq \max\left\{K_0, C'\log\left(\frac{2C'}{p}\right)\right\}$.}
\end{align}
We claim that 
\begin{align}\label{eq:lowerboundgetinneighborhood}
P(x_{k_0} \in B_{\deltaNCVX}(\bar x)) \geq 1- p/2 \qquad \text{ for all $k_0 \geq 2K_1 - 1$.}
\end{align}
Note that this yields the proof, since in that case
$$
P(E_{k_0, q, C}) \geq P(E_{k_0, q, C} \mid x_{k_0} \in U) P(x_{k_0} \in B_{\deltaNCVX}(\bar x)) \geq 1 + p^2/4 - p \geq 1-p,
$$
for all $k_0 \geq \max\left\{K_0, C'\log\left(\frac{2C'}{p}\right), 2K_1 - 1\right\}$.

Observe that~\eqref{eq:lowerboundgetinneighborhood} will follow if 
\begin{align}\label{eq:lowerboundgetinneighborhood2}
P(f(x_{k_0}) - f(\bar x) \leq a) \geq 1- p/2 \qquad \text{ for all $k_0 \geq 2K_1 - 1$.}
\end{align}
Indeed, by Lemma~\ref{sec:getinneighborhood}, we have. 
$$
\{x \in \RR^d \colon f(x) - f(\bar x) \leq a\} \subseteq \overline B_{\deltaNCVX/2}(\bar x) \subseteq B_{\deltaNCVX}(\bar x).
$$
To prove~\eqref{eq:lowerboundgetinneighborhood2}, we apply Theorem~\ref{thm:sublinear}.
To that end, note that $\{x\in \RR^d \colon f(x) \leq f(x_0)\}$ and the widened sublevel set $S$ are indeed bounded, due to Lemma~\ref{sec:getinneighborhood}.
Therefore {\color{blue}$D$ and} the Lipschitz constant $L'$ of $f$ on $S$ are finite.
%Moreover, $\argmin f\subseteq S$. 
%Consequently, the bound holds:
%$$
%{\color{blue} D \diam \leq \mathrm{diam}(S).
%$$
Now observe $G := \min_{K_1 \leq k \leq 2K_1 - 1}\{ G_k\} \geq K_1$ since $G_k \geq k+1$ for all $k$.
Thus, there exists $i \leq G$ such that 
$$(1/2)K_1^{-1/2}\leq \sigma_i \leq K_1^{-1/2}.$$
Therefore, applying Theorem~\ref{thm:sublinear} (in particular~\eqref{eq:functiongapboundsublinear}) with this $\sigma_i$, we have 
\begin{align}\label{eq:dimdep2}
f(x_{2K_1 -1}) - f(\bar x) \leq {\color{blue}D}\max\left\{\frac{16(f(x_{K_1}) - \inf f)}{ K_1^{1/2}}, \frac{16L'\sqrt{2 \log(2K_1^2/p)}}{K_1^{1/2}}, \frac{\sqrt{128}L'}{K_1^{1/2}}\right\} +  \frac{2L'}{K_1^{1/2}}
\end{align}
with probability at least $ 1- p/2$.  Thus, to complete the proof, we show that the left-hand side of~\eqref{eq:dimdep2} is smaller than $a$.
Indeed, it is straightforward to check that
$$
 \max\left\{\frac{2L'}{K_1^{1/2}}, \frac{16{\color{blue}D}(f(x_{K_1}) - \inf f)}{ K_1^{1/2}},\frac{\sqrt{128}{\color{blue}D}L'}{K_1^{1/2}}  \right\} \leq \frac{a}{2}.
$$
Thus, the proof will follow if
\begin{align}\label{eq:finalneededbound}
\frac{16{\color{blue}D}L'\sqrt{2 \log(2K_1^2/p)}}{K_1^{1/2}} \leq \frac{a}{2}.
\end{align}
We perform this calculation in Appendix~\ref{sec:finalneededbound}. 
Thus, the proof is complete. 
\end{proof}

\subsubsection{Local oracle complexity}

Thus, we have established a local nearly linear convergence rate for $\algname$.
To understand the overall complexity of the method, we must derive an upper bound on the contraction factor $q$.
The following lemma, which is proved in Appendix~\ref{sec:contractionlowerbound}, provides one {\color{blue} that depends on a worst-case condition number of $f$.}
\begin{lem}\label{lem:contractionlowerbound}
Suppose that without loss of generality that $\deltaA \leq 1$. Define the condition number 
\textcolor{blue}{
$$
\kappa = \frac{\max\{L,\beta, \ver\}}{\min\{\gamma, \mu\}}. 
$$
}
Then there exists a universal constant $\eta > 0$ independent of $f$ such that 
$$
q \leq 1 - \frac{\eta}{\kappa^8(1+\shape)^2}.
$$
where $q$ is defined as in Table~\ref{table:main_convergence_theorems}.
\end{lem}
With this upper bound on $q$, it is straightforward to derive a local complexity estimate for $\algname$: the method locally produces a point $\hat x$ satisfying $f(\hat x) - f(\bar x) \leq \varepsilon$ with at most  
\textcolor{blue}{
$$
O\left(\left(\kappa^8(1+\shape)^2 \log(1/\varepsilon)\right)^3\right),
$$
}
first-order oracle evaluations.
This bound may be pessimistic since we did not attempt to optimize the constants $C_i$ or $a_i$. 
We leave the improvement of this complexity as an intriguing open question.

Before moving to a brief numerical illustration, we explain how Theorem~\ref{thm:maintheoremsemi} from the introduction follows from the above results.
\begin{remark}[Establishing Theorem~\ref{thm:maintheoremsemi}]
{\rm 
Theorem~\ref{thm:maintheoremsemi} from the introduction immediately follows from Theorems~\ref{thm:maintheorem} and~\ref{thm:mainconvexsetting}.
Indeed, first the event $E_{k_0, q, C}$ from Theorems~\ref{thm:maintheorem} and~\ref{thm:mainconvexsetting} contains the corresponding event $E_{k_0, q, C}$ from Theorem~\ref{thm:maintheoremsemi} for particular $q$ and $C$, which depend solely on $f$.
Second, from the statement of theorems, we see that the neighborhood of local nearly linear convergence, $B_{\deltaNCVX}(\bar x)$, depends solely on $f$.
}
\end{remark}
%%%%%%%%%%%%%%%%%

%%%%%Experiments%%%%%%
% !TEX root = template.tex

\section{Numerical illustration}\label{sec:numerical}
In this section, we briefly illustrate the numerical performance of $\algname$ on two nonsmooth objective functions, borrowed from~\cite{lewis2019simple,lewis2013nonsmooth,burer2007solving,anstreicher2004masked}.
In both experiments, we compare $\algname$ to the subgradient method with the popular Polyak stepsize (\texttt{PolyakSGM})~\cite{Polyak69}, which iterates 
$$
x_{k+1} = x_k - \frac{f(x_k) - \inf f}{\|w_k\|^2} w_k \qquad \text{for some $w_k \in \partial f(x_k).$}
$$
In the first example, $\inf f$ is known, in the second, we estimate $\inf f$ from multiple runs of $\algname.$
We choose to compare against the subgradient method because it is a simple first-order method with strong convergence guarantees in convex~\cite{Polyak69} and nonconvex settings~\cite{davis2020stochastictame}. 
Importantly, \texttt{PolyakSGM} accesses the objective solely through function and subgradient evaluations.
Thus, we compare the accuracy achieved by $\polyak$ and $\algname$ after a fixed number of oracle calls, i.e., evaluations of $\partial f$.

Let us comment on the implementation of $\algname$. 
First, in all experiments, {\color{blue}unless otherwise noted}, we do not tune parameters of $\algname$. 
Instead, we {\color{blue}simply choose scaling constant $\sscale = 10^{-6}$ and loop size parameters}
$$
T_k = k+1 \qquad \text{ and } \qquad G_k = \min\{k+1 , \lceil \log_2(10^{-16})\rceil\} \quad \text{for all $k \geq 0$.}
$$ 
Second, we attempt to save first-order oracle calls by breaking the loop on Lines~\ref{line:linesearchbegin} through~\ref{line:linesearchend} of Algorithm~\ref{alg:linesearch} 
whenever we find that $\sigma_i > \|v_{i+1}\|/s$. 
Since $\sigma_i$ is increasing in $i$ and $\|v_{i+1}\|$ is nonincreasing in $i$, this does not affect the iterates $x_k$ of $\algname$; see Lemma~\ref{lem:linesearchproperties}. {\color{blue} Finally, in all problems, we initialize $\algname$ and $\polyak$ at a random vector $a z$ where $z$ is sampled from the uniform distribution on the unit sphere. For all problems, we use $a = 1$ unless otherwise noted. 
Note that in the problems of Section~\ref{sec:maxofsmoothsection} and ~\ref{sec:matrix_sensing}, the solution is known, while in the problem of  Section~\ref{eq:eigenvalueproduct}, the solution is unknown.

The purpose of this section is not to argue that $\algname$ is a substitute for standard subgradient methods in most problems. 
Instead, we only wish to point out some scenarios where standard first-order methods are known to perform poorly, yet $\algname$ asymptotically accelerates.
We are also not arguing that $\algname$ has fast global rates: indeed, we previously mentioned that the $\algname$'s global rate is $O(\eps^{-6})$ which is much worse than $\polyak$'s $O(\eps^{-2})$ rate for general convex problems. In practice, one could devise schemes that couple $\algname$ with $\polyak$, eventually switching to $\algname$ when it begins to outperform $\polyak$.
While we leave a more thorough numerical study to future work, the reader may download and run our PyTorch~\cite{paszke2019pytorch} implementation of $\algname$ at the following url: \href{https://github.com/COR-OPT/ntd.py}{https://github.com/COR-OPT/ntd.py}

}
We now turn to the examples.

\subsection{A max-of-smooth function}\label{sec:maxofsmoothsection}
In this example, $f$ takes the following form
\begin{align}\label{eq:maxofsmootheqexp}
f(x) = \max_{i=1,\ldots, m}\left\{ g_i^\top x + \frac{1}{2} x^T H_i x\right\},
\end{align}
where we generate a random vector $\lambda \in \RR^m$ in $\{\lambda>0 \colon \sum_{i=1}^{m}\lambda_i=0\}$, a random positive semi-definite matrix $H_i$, and a random vector $g_i$ satisfying that $\sum_{i=1}^{m}\lambda_i g_i = 0$.
In this case, one can show that with probability 1, $f$ satisfies Assumption~\ref{assumption:mainfinal} at its unique minimizer $0$.
 \begin{figure}[H]
	\centering
	\begin{subfigure}[b]{\widthforbig\textwidth}
		\centering
		\includegraphics[width=\textwidth]{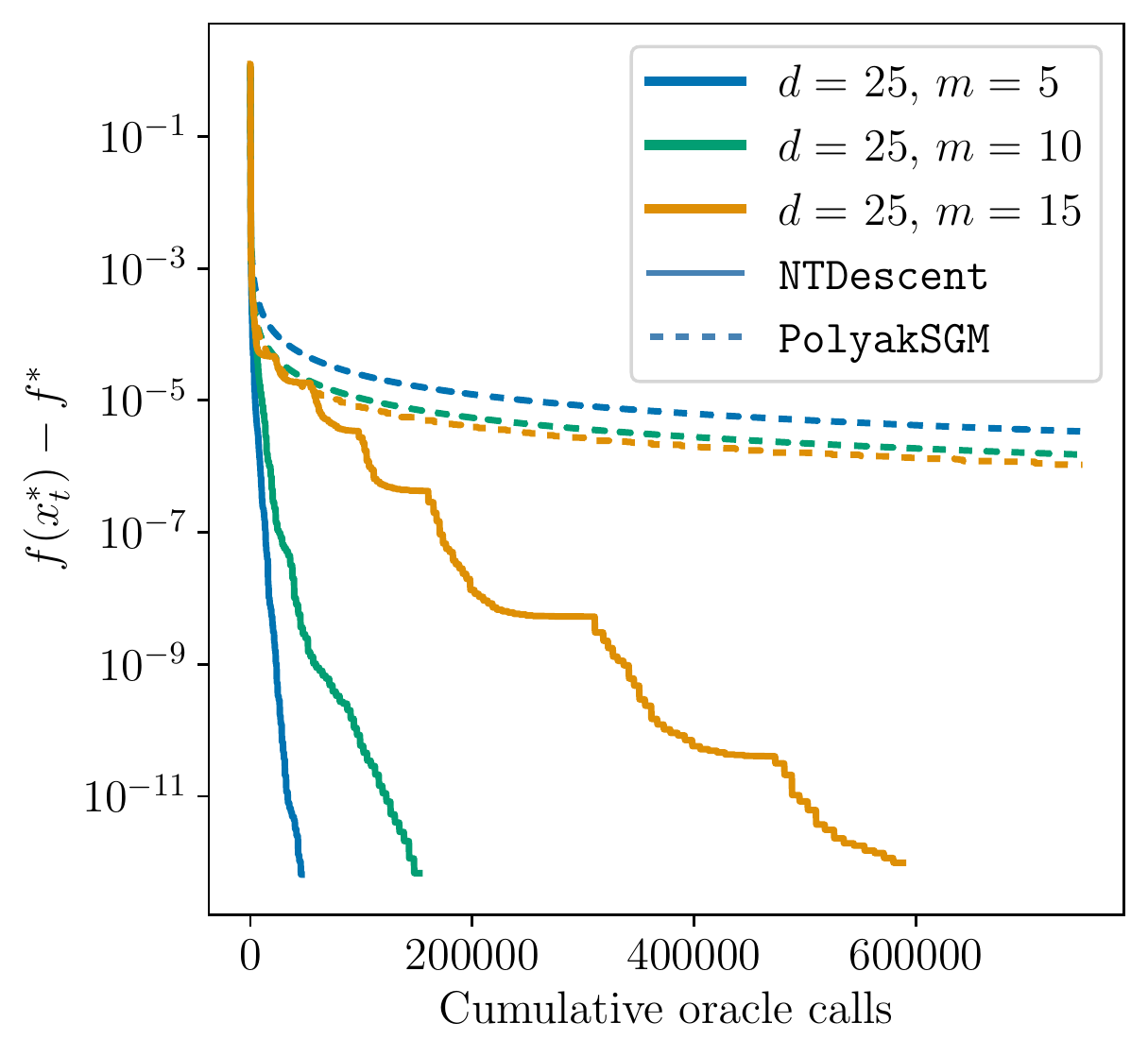}
		\caption{}
		\label{fig:3a}
	\end{subfigure}
	\hfill
	\begin{subfigure}[b]{\widthforbig\textwidth}
		\centering
		\includegraphics[width=1\textwidth]{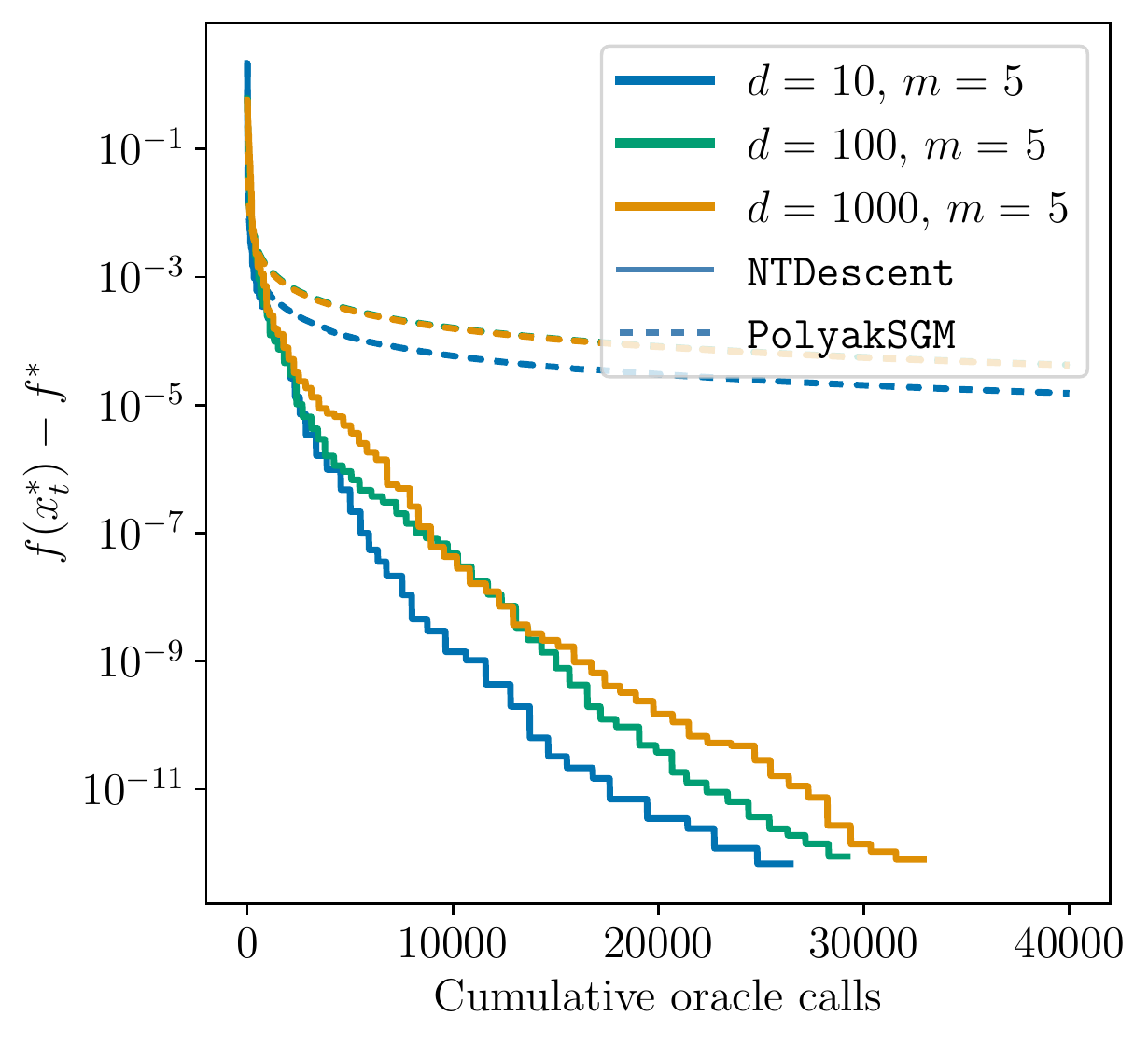}
		\caption{}
		\label{fig:3b}
	\end{subfigure}\\
		\begin{subfigure}[b]{\widthforbig\textwidth}
		\centering
		\includegraphics[width=\textwidth]{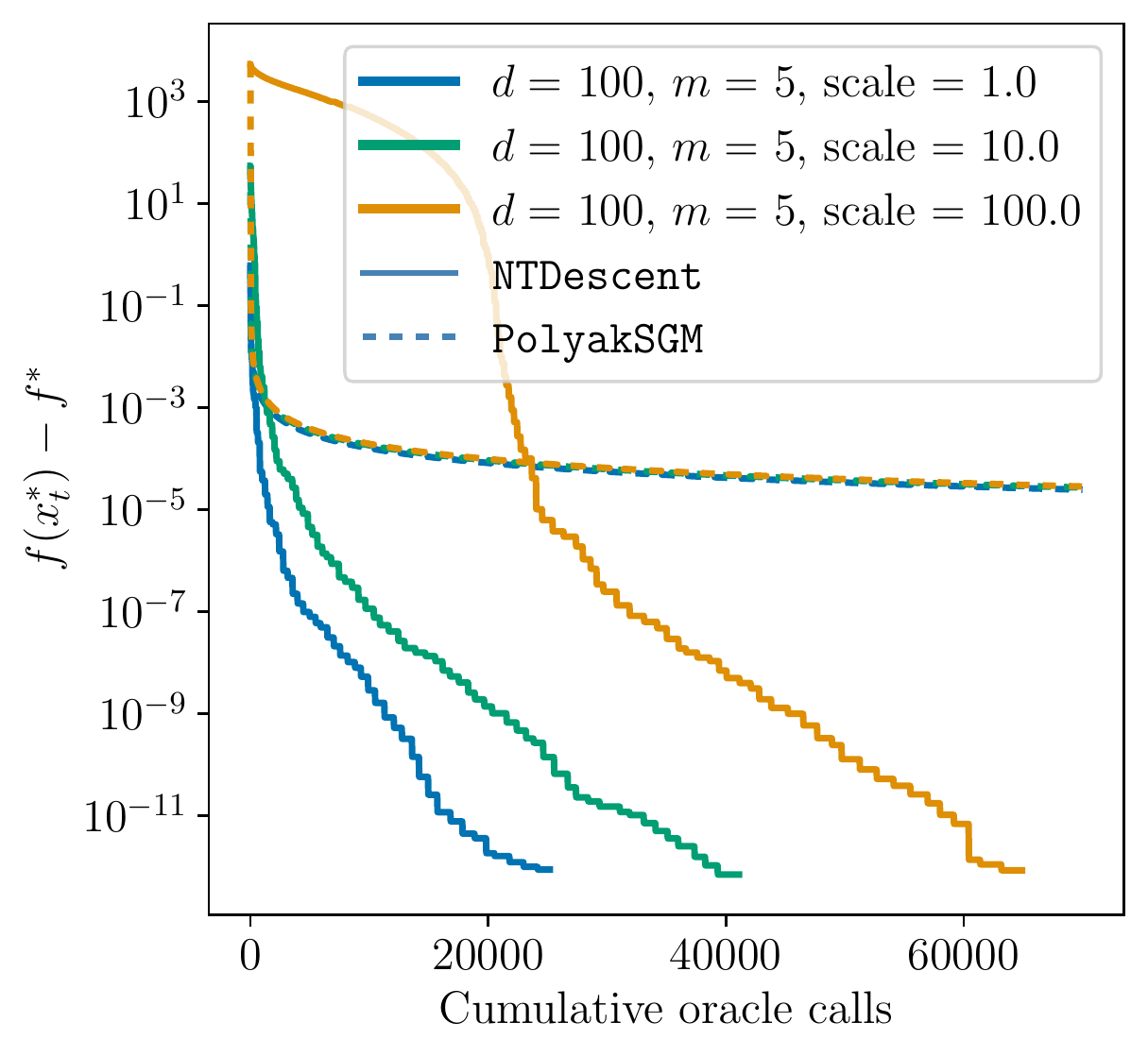}
		\caption{}
		\label{fig:3c}
	\end{subfigure}
	\hfill
	\begin{subfigure}[b]{\widthforbig\textwidth}
		\centering
		\includegraphics[width=1\textwidth]{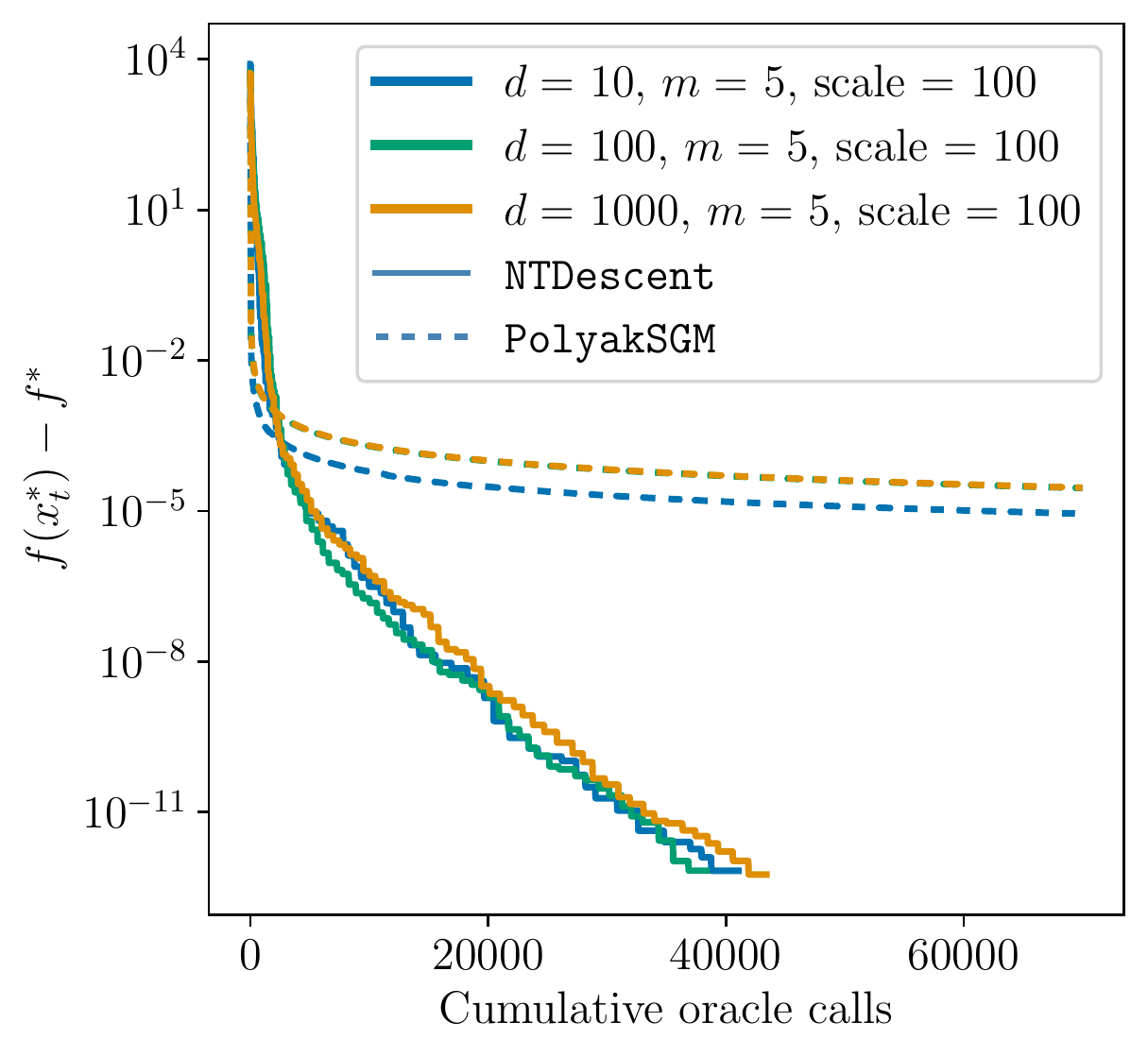}
		\caption{}
		\label{fig:3d}
	\end{subfigure}
		\begin{subfigure}[b]{\widthforbig\textwidth}
		\centering
		\includegraphics[width=\textwidth]{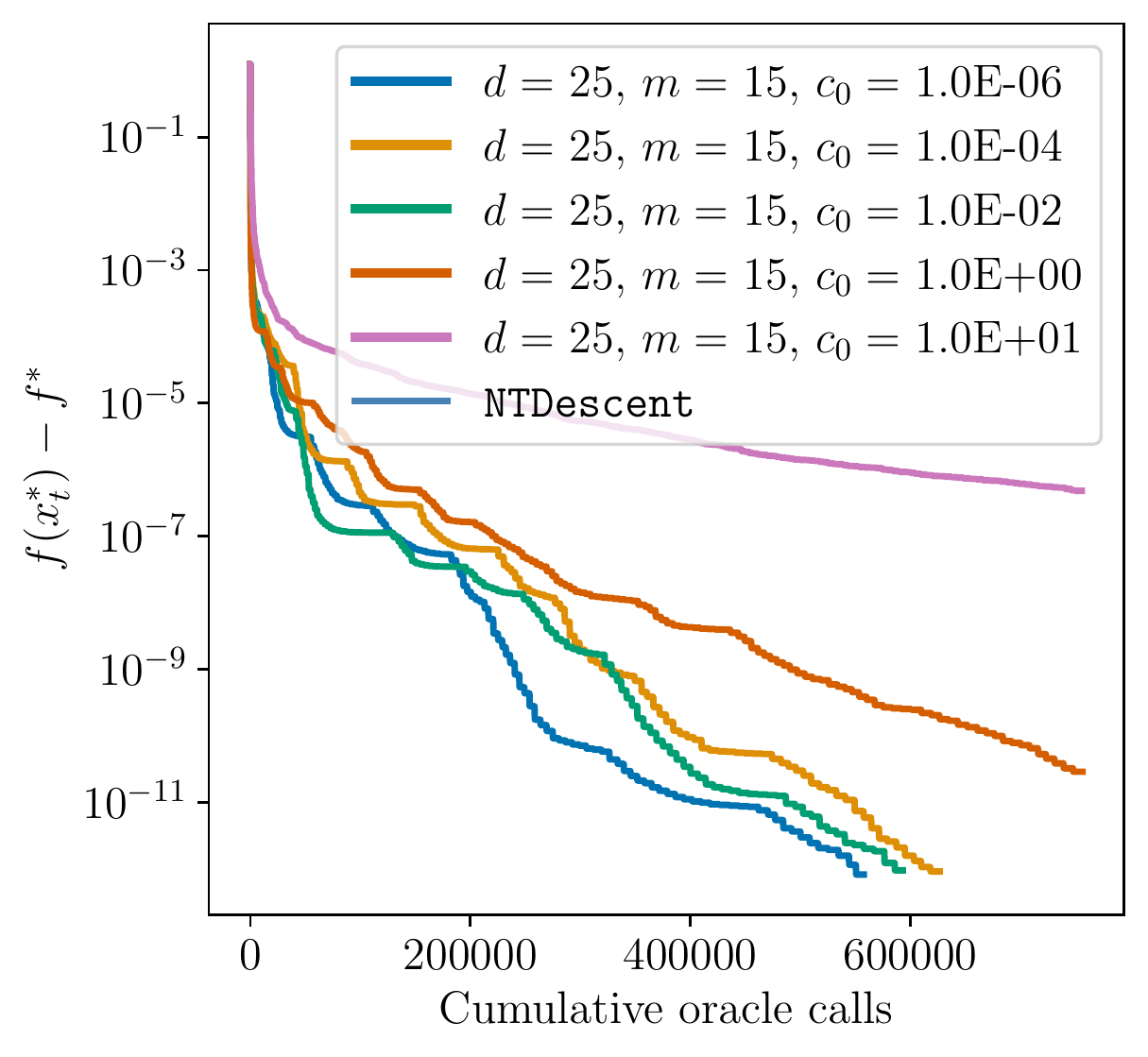}
		\caption{}
		\label{fig:3e}
	\end{subfigure}
	\hfill
	\begin{subfigure}[b]{\widthforbig\textwidth}
		\centering
		\includegraphics[width=1\textwidth]{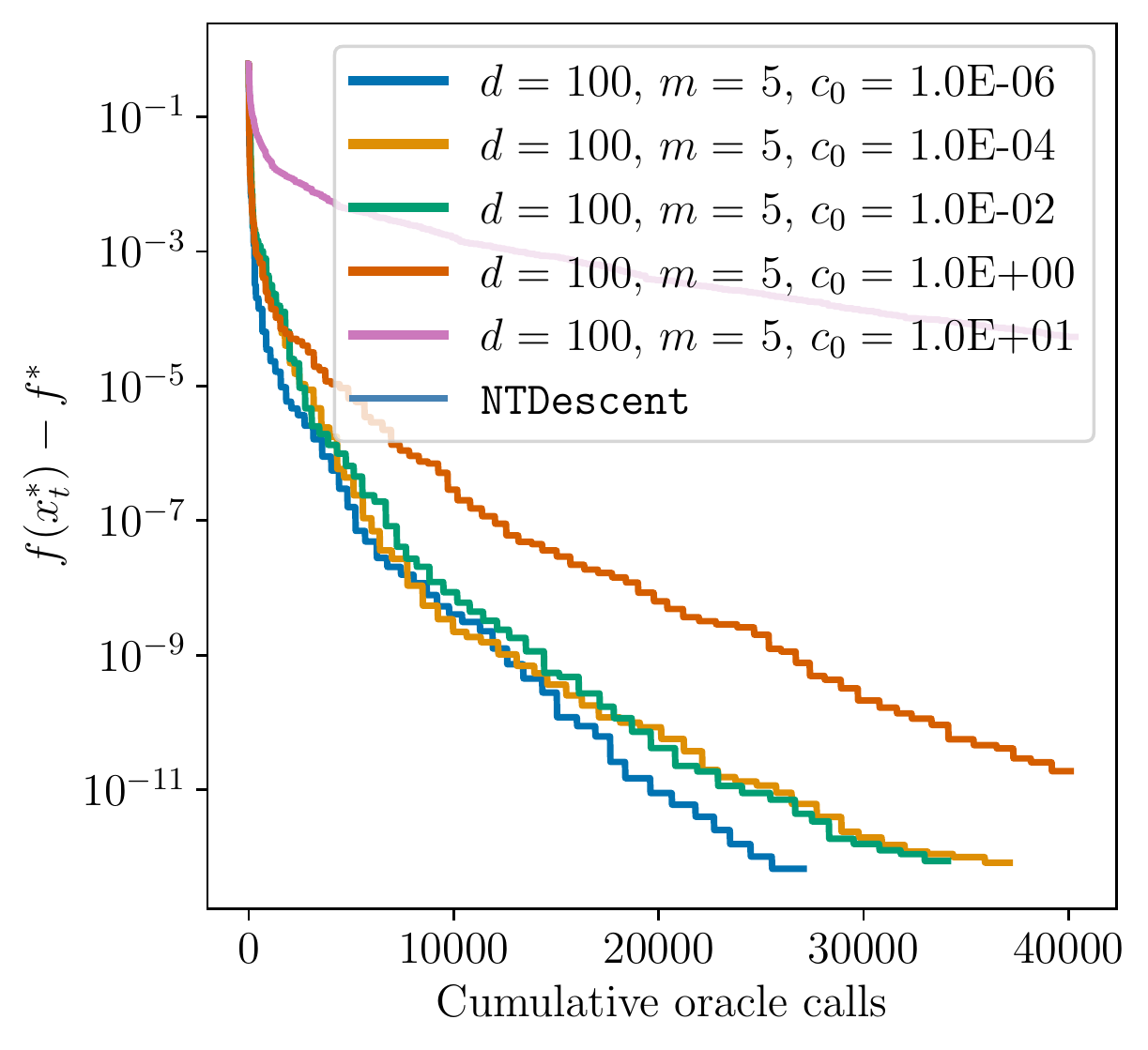}
		\caption{}
		\label{fig:3f}
	\end{subfigure}	
	\caption{Comparison of $\algname$ with $\polyak$ on~\eqref{eq:maxofsmootheqexp}. % In Figure~\ref{fig:3a} we fix $d=25$ and vary $m \in \{5,10, 15\}$; In Figure~\ref{fig:3b} we fix the $m=5$ and vary $d \in \{10, 100, 1000\}$.
%In Figure~\ref{fig:3c} 
%	we fix $d=100$ and $m = 5$ and vary the scale of the random initializer, which is sampled uniformly from spheres of radius $\{1, 10, 100\}$.
%In Figure~\ref{fig:3d} 
%	we fix $d=100$, $m = 5$, and a random initializer on the sphere of radius 100; then we use strategy outlined in the text to adjust the maximal size of $\sigma_i$. 
	For both algorithms, the value $f(x_t^*)$ denotes the best function seen after $t$ oracle evaluations. 
	See text for description.}\label{fig: maxofsmooth}
\end{figure}

In Figure~\ref{fig: maxofsmooth} we plot the performance of $\algname$ and $\polyak$ for multiple pairs of $(d, m)$, {\color{blue}varying initialization scale, a slight modification of $\algname$ that allows longer steps, and varying scales $\sscale$.}
{\color{blue} We begin with Figure~\ref{fig:3a} and Figure~\ref{fig:3b}}.
Figure~\ref{fig:3a} shows that the performance of $\algname$ depends on $m$. 
On the other hand, Figure~\ref{fig:3b} shows $\algname$ performance is independent of $d$, as expected.
Both plots show that $\algname$ {\color{blue} asymptotically} outperforms $\polyak.$ 
{\color{blue}
Turning to initialization, Figure~\ref{fig:3c} shows the result of initializing $\algname$ at a random vector $az$, where $z$ is uniformly drawn from the sphere and $a$ is a scale parameter satisfying $a\in \{1, 10, 100\}$. 
Clearly, $\algname$ is affected by the initialization scale but surpasses $\polyak$ after $30000$ oracle calls. 
While we expect $\algname$ to  converge slowly when far from minimizers, we introduce a simple strategy to mitigate this behavior.
\paragraph{Adaptive grid strategy.}
Briefly, suppose we run $\linesearch$ the full $G$ steps without exiting (via the violation of the trust region constraint). 
Then we simply continue the $\linesearch$ loop trying $\sigma_{-1} = 10\sigma_0$, $\sigma_{-2} = 10\sigma_{-1}$, and so on, until we violate the trust region constraint or $\sigma_{-i}$ exceeds a predefined threshold. \

Figure~\ref{fig:3d} shows the result of this strategy with a predefined threshold $\infty$, showing that it compensates for poor initialization quality.
Finally in Figure~\ref{fig:3e} and~\ref{fig:3f}, we show the effect of changing the $\sscale$ input to $\algname$. It appears $\algname$ is relatively insensitive to $\sscale$ and smaller choices generally result in better performance.
This motivates our default choice $\sscale = 10^{-6}$ in the remainder of the experiments.
}

Before turning to our second experiment, we briefly mention two alternative methods -- Prox-linear~\cite{burke1985descent,fletcher1982model,enright1992discrete,wright1990convergence,yuan1985superlinear} and Survey Descent~\cite{han2021survey} -- which could be applied to this problem. 
In order to explain these algorithms, let us write $f = \max_{i = 1, \ldots, m} \{f_i\}$, where the $f_i$ are the quadratic function from~\eqref{eq:maxofsmootheqexp}. 
\paragraph{Prox-linear method.}
Given a point $x \in \RR^d$, the Prox-linear update $x_+$ solves  
$$
x_{+} = \argmin_{y \in \RR^d} \max_{i=1, \ldots, m}\left\{f_i(x) + \dotp{\nabla f_i(x), y - x} \right\} + \frac{\rho}{2} \|y - x\|^2.
$$
One may show that $x_+$ geometrically improves on $x$; see~\cite{prox_error}.
However, in contrast to $\algname$, the prox-linear method requires that the components $f_i$ are known. 
This is stronger than the first-order oracle model considered in this work.
Thus, we do not compare $\algname$ with prox-linear.

\paragraph{Survey Descent} 
The Survey Descent method is a multi-point generalization of gradient descent, designed for max-of-smooth functions. 
Rather than maintaining a single iterate sequence, the Survey Descent maintains a \emph{survey} $S$ of points, meaning a collection of points $\{s_i\}_{i=1}^m$ at which $f$ is differentiable. 
A single iteration of the Survey Descent method then aims to produce a new survey $S^+ = \{s_i^+\}_{i = 1}^m$ satisfying
\begin{align*}
s_i^+ := &\argmin_{x \in \RR^d} \left\|x - \left(s_i - \frac{1}{L}\nabla f(s_i)\right)\right\|^2\\
&\text{subject to:} \; f(s_j) + \dotp{\nabla f(s_j), x - s_j} + \frac{L}{2}\|x - s_j\|^2 \leq f(s_i) +  \dotp{\nabla f(s_i), x - s_i} \quad \forall j \neq i.
\end{align*}
Here, $L$ is an upper bound on the Lipschitz constant of $\nabla f_i$ for all $i = 1,\ldots, m$.
In~\cite{han2021survey}, Han and Lewis study linear convergence of Survey Descent on max-of-smooth functions under the conditions of Corollary~\ref{cor:maxofsmooth}. 
Given a survey $S$, they show that the updated survey $S^+$ geometrically improves on $S$ (in an  appropriate sense) whenever the following conditions are satisfied: 
(i) all elements of the survey $S$ are near $\bar x$; (i) the survey $S$ is \emph{valid}, meaning there exists a permutation $a$ on $[m]$ such that 
$$
f_{a(i)}(s_i) = f(s_i) \qquad \text{ and } \qquad \partial f(s_i) = \{ \nabla f_{a(i)}(s_i)\} \qquad \text{for all $i = 1, \ldots, m$}.
$$
To estimate the number of components $m$ and find a valid initial survey $S$ sufficiently close to $\bar x$, Han and Lewis suggest an empirical procedure based on running a nonsmooth variant of BFGS~\cite{lewis2013nonsmooth} for several iterations. 
After running BFGS, they suggest to (i) compute an estimate $\hat m$ of $m$ from a singular value decomposition of the computed gradients, and (ii) build the survey from $\hat m$ past iterates in such a way that the computed gradients form an affine independent set. 
From the numerical illustration in~\cite{han2021survey}, Survey Descent performs well on several small problems.
However, since the initialization procedure and implementation of Survey Descent are somewhat sophisticated, we leave a detailed comparison between $\algname$ and Survey Descent and to future work.

{\color{blue} 
\subsection{A matrix sensing problem} \label{sec:matrix_sensing}

In this example, $f$ takes the following form 
\begin{align*}
f(X) = \frac{1}{n}\|\cA(XX^T) - \cA(M_{\star})\|_1, 
\end{align*}
where $M_{\star} \in \RR^{N\times N}$ is an unknown positive semidefinite matrix of rank $r_\star$ that we wish to recover from known linear measurements $\cA(M_\star)$; the linear operator $\cA \colon \RR^{N\times N} \rightarrow \RR^n$ takes the form $Y \mapsto (a_i^T Y a_i - b_i^T Y b_i)_{i = 1}^n$, for $n \in N$, where $a_i, b_i \in \RR^d$ are random vectors sampled from a standard multivariate normal distribution; and the decision variable is a tall and skinny matrix $X \in \RR^{N\times r}$, where in general we allow $r \neq r_\star$. This optimization problem appears in various signal processing applications and is known as \emph{quadratic sensing}~\cite{chen2015exact}.
Note that this objective does not satisfy Assumption~\ref{assumption:mainfinal}, since the solution set is not isolated. 

We consider two settings in this section: the exact setting $r = r_\star$ and the overparameterized setting $r > r_\star$. 
In the exact setting~\cite{charisopoulos2021low} showed that if $n = \Omega(Nr)$, the objective $f$ is sharp, meaning $f(x) = \Omega(\dist(x, \argmin f))$ and that $\polyak$ converges linearly whenever the initial iterate is sufficiently close to the set of minimizers. 
In the overparameterized setting, we are not aware of similar guarantees. 
Note that in practice, $r_\star$ is unknown, so the overparameterized setting is likely to be encountered. 

In Figure~\ref{fig:matrixsensing} we plot the performance of $\algname$ and $\polyak$ in two experiments. 
In Figure~\ref{fig:matrixsensinga} we use base dimensions $N = 100$, optimal rank $r_\star = 5$, and varying overparameterization $r \in \{r_\star, r_\star + 2, r_\star + 5\}$.
In Figure~\ref{fig:matrixsensingab} we use base dimensions $N = 100$, varying optimal rank $r_\star \in \{5, 10, 15\}$, and fixed overparameterization $r = r_\star + 5$. 
In both experiments, we fix $n = 4Nr_\star$.
Note that the dimension of the decision variable $X$ is varying across each run since $d = Nr$ and $r$ are varying.
As can be seen from the plot, $\polyak$ outperforms $\algname$ in the exact setting. This is to be expected, since $f$ is a sharp function on which $\polyak$ is known to perform well. 
On the other hand, when $r > r_\star$, we find that both methods slow down. 
However, $\algname$ continues to converge nearly linearly, while $\polyak$ converges sublinearly.

 \begin{figure}[H]
	\centering
	\begin{subfigure}[b]{0.45\textwidth}
		\centering
		\includegraphics[width=\textwidth]{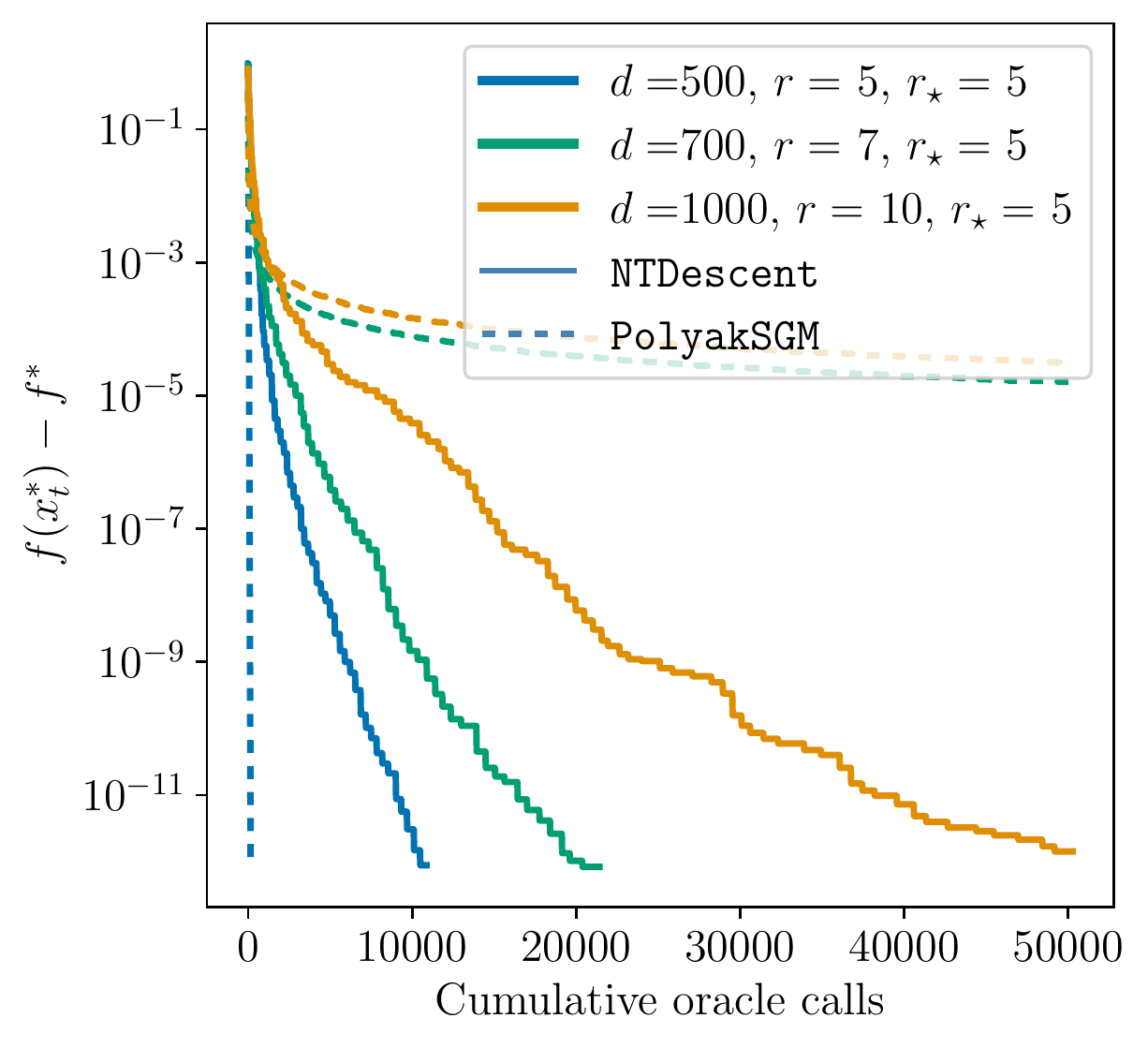}
		\caption{}
		\label{fig:matrixsensinga}
	\end{subfigure}
	\hfill
	\begin{subfigure}[b]{0.45\textwidth}
		\centering
		\includegraphics[width=\textwidth]{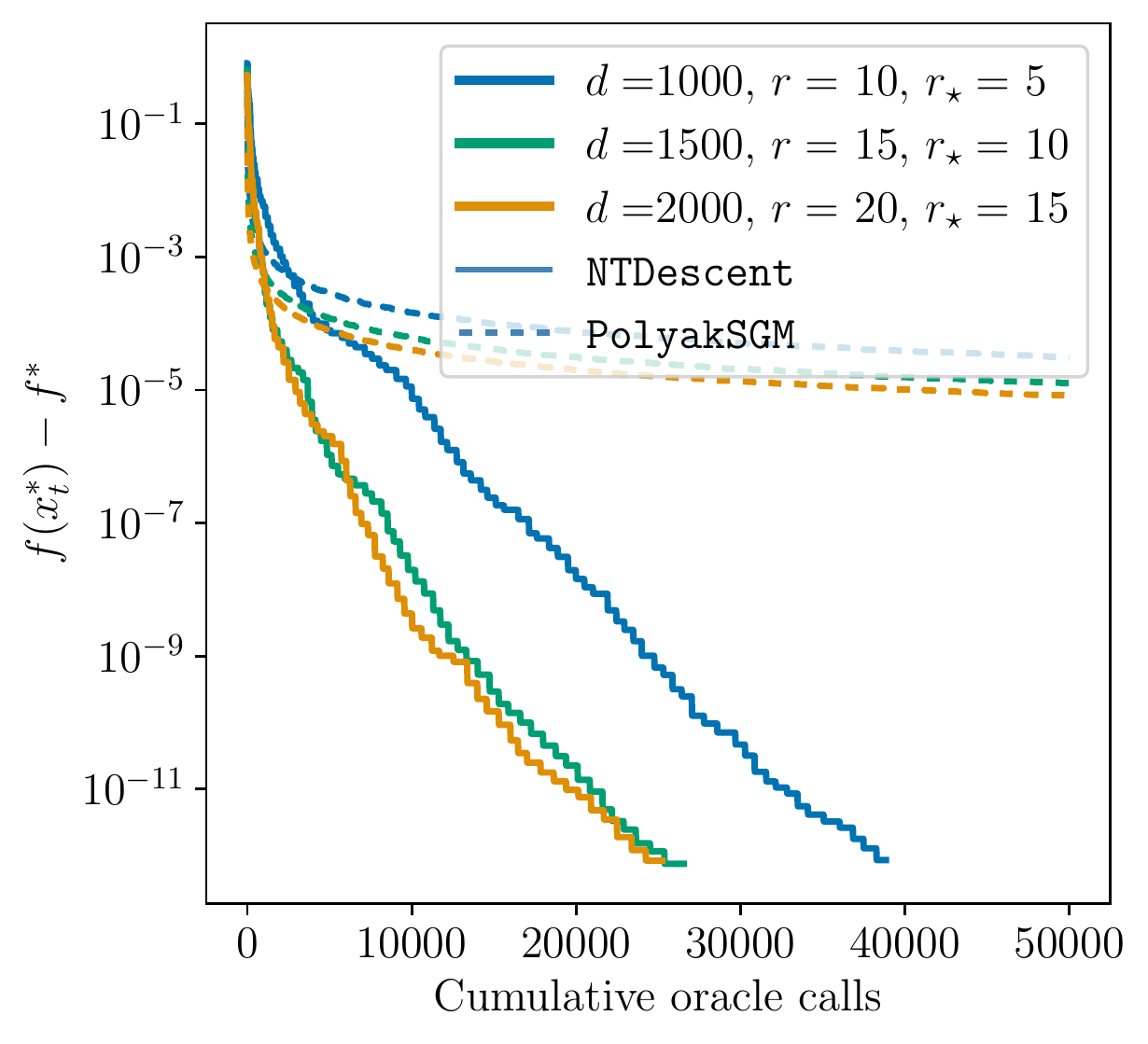}
		\caption{}
		\label{fig:matrixsensingab}
	\end{subfigure}\\
	\caption{{\color{blue}Comparison of $\algname$ with $\polyak$ on~\eqref{eq:maxofsmootheqexp}. In both plots, the base dimension is $N = 100$. Left: fixed optimal rank $r_\star = 5$ and varying overparameterization $r \in \{5, 7, 10\}$;  Right: varying $r_\star \in \{5, 10, 15\}$, fixed overparameterization $r = r_\star + 5$. For both algorithms, the value $f(x_t^*)$ denotes the best function seen after $t$ oracle evaluations.}}\label{fig:matrixsensing}
\end{figure}

}

\subsection{An eigenvalue product function}\label{eq:eigenvalueproduct}
In this example, we aim to optimize a function $\tilde f$ that takes the following form
$$
\tilde f(X) = \log E_K(A\odot X),
$$ 
where $A$ is a fixed positive semi-definite data matrix, $E_K(Y)$ denotes the product of $K$ largest eigenvalues of a symmetric matrix $Y \in \mathbb{S}^{N}$, and $\odot$ denotes the Hadamard (entrywise) matrix product, subject to the constraint that $X$ is positive semi-definite and its diagonal entries are $1$. 
This example is a nonconvex relaxation of an entropy minimization problem arising in an environmental application~\cite{burer2007solving,anstreicher2004masked}. 
In our experiments, we choose $A$ as in~\cite{anstreicher2004masked}: $A$ is the leading $N\times N$ submatrix of a $63\times 63$ covariance matrix, scaled so that the largest entry is $1$.
As suggested by~\cite{burer2007solving}, we reformulate this problem as an unconstrained optimization problem using a Burer-Monteiro type factorization
\begin{align}\label{eq:eigenvalueproblem}
\min_{V \in \RR^{N\times N}}  f(V) = \tilde f(c(V)c(V)^\top),
\end{align}
where $c \colon \RR^{N \times N} \rightarrow \mathbb{S}^N$ satisfies $c(V) = \text{Diag}([\text{diag}(VV^\T)]^{-1/2})V$ for all $V \in  \RR^{N \times N}$. Here, the mapping $\text{diag}(\cdot)$ takes a matrix an $N \times N$ matrix $A$ to the $N$ dimensional vector with $i$th entry $A_{ii}$.
On the other hand, the mapping $\text{Diag}(\cdot)$ takes an $N$ dimensional vector $v$ to the $N\times N$ diagonal matrix with $i$th diagonal entry $v_{i}$. 
A formula for the subgradient of $f$ may be found~\cite{burer2007solving}.
We do not attempt to verify that $f$ satisfies the full Assumption~\ref{assumption:mainfinal}. 
Instead, we point out that under a ``transversality condition," function $f$ admits an active manifold at local minimizers~\cite{lewis2013nonsmooth}.

Turning to the experiment, we consider the case where $N=14$ and $K=7$. 
In this example, the optimal function value $\inf f$ is not known. 
Thus, we run $\algname$ from four random initial starting points.
We terminate each run of $\algname$ when a certain ``optimality gap" $R_k$ satisfies $R_k \le 10^{-12}$. 
We denote the minimal function value achieved across all four runs by $\opt$.
Let us now define and motivate the optimality gap.
For iteration $k$ in Algorithm~\ref{alg:mainalg}, define
$$
R_k := \min\left\{ \max\{\sigma_i^{(k)}, \|v_{i+1}^{(k)}\|^2\} \colon  \sigma_i^{(k)} \le \|v_{i+1}^{(k)}\| \right\},
$$
where $\sigma_i^{(k)}$ and $ v_{i+1}^{(k)}$ are computed in Lines~\ref{line:linesearchbegin} through~\ref{line:linesearchend} of Algorithm~\ref{alg:linesearch} at iteration $k$. 
Provided that $x_k$ is sufficiently close to a point $\bar x$ at which function $f$ satisfies Assumption~\ref{assumption:mainfinal}, 
it is possible to show that $R_k$ satisfies $f(x_k) - f(\bar x) \lesssim R_k$.
This is illustrated in Figure~\ref{fig:4a}: there, the optimality gap closely tracks the estimated function gap, when approximating by $\inf f $ by $\opt$.
In Figure~\ref{fig:4b}, we compare the performance of $\algname$ on the three runs which did not achieve function value $\opt$ before termination. 
In all three cases, we see similar performance. 
Next, for each run of $\algname$, we also run $\polyak$ from the same initial starting point, estimating $\inf f$ by $\opt$.
We see that $\algname$ outperforms $\polyak$.
 \begin{figure}[H]
 	\centering
	   \begin{subfigure}[b]{0.45\textwidth}
         \centering
 	\includegraphics[width=\textwidth]{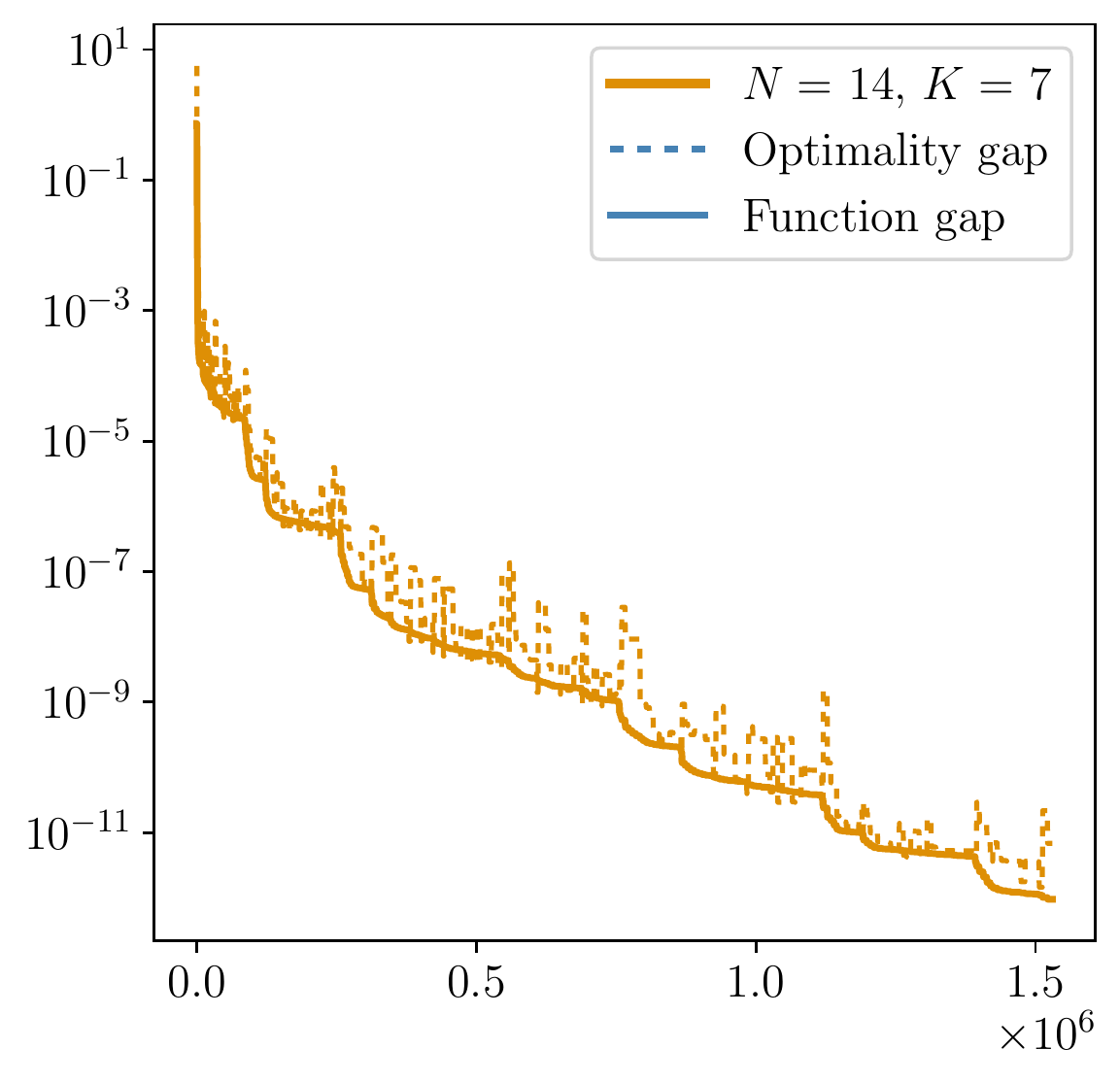}
         \caption{}
         \label{fig:4a}
     \end{subfigure}
     \hfill
     \begin{subfigure}[b]{0.45\textwidth}
         \centering%\vspace{-10pt}
	\includegraphics[width=1.0375\textwidth]{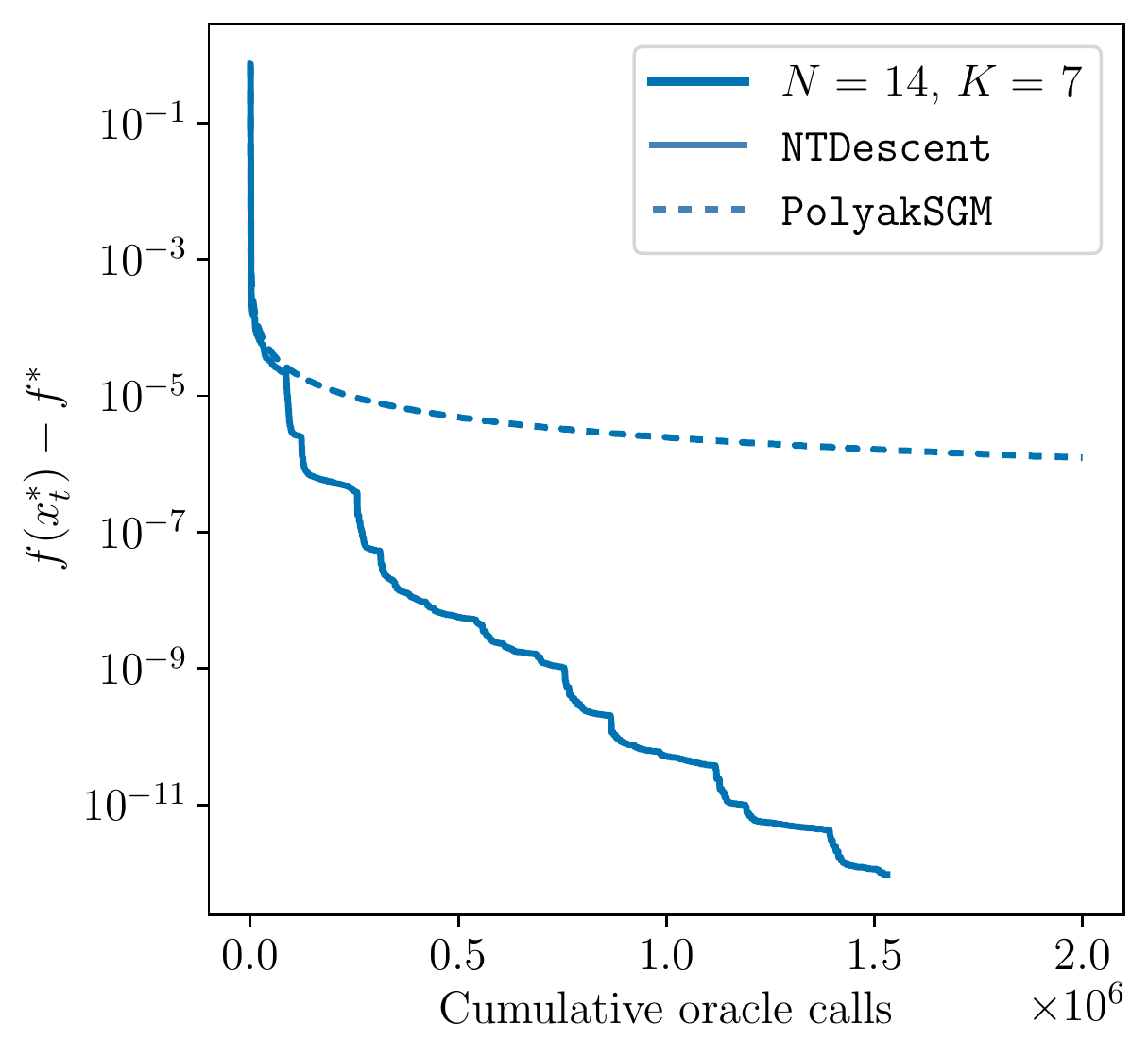}
         \caption{}
         \label{fig:4b}
     \end{subfigure}
 	\caption{Numerical performance on~\eqref{eq:eigenvalueproblem}. Left: the close relationship between ``optimality gap" and function gap; Right: comparison of $\polyak$ and $\algname$ from three initial starting points. For both algorithms, the value $f(x_t^*)$ denotes the best function seen after $t$ oracle evaluations. See text for detail.}\label{fig: optimalitygapvsfunctiongap}
 \end{figure}

%%%%%%%%%%%%%%%%%

			\bibliographystyle{abbrv}
	\bibliography{bibliography}

	%%%%Appendix %%%%%%%%

%% !TEX root = template.tex

\appendix

\section{Proof of Lemma~\ref{lem:goldsteinsubgradientinequality}}\label{sec:prooflem:goldsteinsubgradientinequality}

Let $g$ denote the minimal norm element of $\partial_{\sigma} f(x)$. 
Write $g$ as a convex combination of subgradients: $g = \sum_{i=1}^{n} \lambda_i g_i$ where $\sum_{i=1}^{n} \lambda_i = 1$ and $g_i \in \partial f(x_i)$ for some $x_i \in B_{\sigma}(x)$ and $n > 0$. 
Then 
\begin{align*}
f(x) &\leq f\left( \sum_{i=1}^n \lambda_i x_i + \sum_{i=1}^n \lambda_i(x - x_i)\right)\\
&\leq \sum_{i=1}^n \lambda_i f(x_i) + L \sigma \\
&\leq f(y) + \sum_{i=1}^n \dotp{\lambda_i g_i, x_i - y} + L \sigma\\
&\leq f(y) + \dotp{g, x - y} + \sum_{i=1}^n \lambda_i \dotp{g_i, x_i - x} + L \sigma\\
&\leq f(y) + \dist(0, \partial_{\sigma} f(x))\|x - y\| + 2L\sigma,
\end{align*}
as desired.

\section{Proof of Proposition~\ref{prop:proof:consequencesofA}}\label{sec:prop:proof:consequencesofA}

We begin with preliminary notation and bounds. 
First, since $\cM$ is $C^4$ smooth, the projection $P_{\cM}$ is $C^3$ smooth near $\bar x$. 
Second, since $f$ is $C^3$ smooth along $\cM$ near $\bar x$, the composition $f_{\cM} := f \circ P_{\cM}$ is also $C^3$ smooth near $\bar x$. 
Third, the constant $\mu$ is positive due to the active manifold assumption. 
Fourth, choose $\delta >0$ small enough that the following hold: 
\begin{enumerate}
\item $\nabla P_{\cM}$ is $\shape$-Lipschitz on $B_{\delta}(\bar x)$;
\item $\nabla f_{\cM}$ is $\beta$-Lipschitz on $B_{\delta}(\bar x)$;
\item $\nabla^2 f_{\cM}$ is $\rho$-Lipschitz on $B_{\delta}(\bar x)$ in the operator norm, where $\rho := 2\lip_{\nabla^2 f_{\cM}}^{\text{op}}(\bar x)$;
\item $f$ is $L$-Lipschitz on $B_{\delta}(\bar x)$;
\item the quadratic growth bound~\ref{assum: quad} holds:
$$
f(x) - f(\bar x) \geq \frac{\scc}{2} \|x - \bar x\|^2 \qquad \text{ for all $x \in \overline B_{\delta}(\bar x)$;}
$$
\item the strong $(a)$ bound~\ref{assum: stronga} holds: 
\begin{align}\label{eq:strongaappendix}
\|P_{T_{\cM}(y)} (v - \nabla_\cM f(y))\| \leq \ver \|x - y\|  
\end{align}
for all $x \in \overline B_{\delta}(\bar x)$, $v \in \partial f(x)$, and $y \in \cM\cap \overline B_{\delta}(\bar x)$.
\item the $(b_{\leq})$ regularity bound~\ref{assum: bregularity} holds:
\begin{align}\label{eq:bregappendix}
f(x') \geq f(x) + \dotp{v, x' - x} - \frac{\mu}{2} \|x - \hat x\|
\end{align}
for all $x \in B_{\delta}(\bar x)$, $v \in \partial f(x)$, and $x' \in B_{\delta}(\bar x) \cap \cM$.
\item the sharpness condition holds: 
$$
\dist(0, \partial f(x)) > 2 \mu \qquad \text{ for all $x \in B_{\delta}(\bar x) \backslash \cM$.}
$$
\end{enumerate}

Given these bounds, let us define 
$$
\deltaA := \frac{1}{2}\min\left\{\delta, \frac{9\gamma}{16\rho}, \frac{\sharpc}{2(\ver + 2\beta  + 2\shape L )}\right\}.
$$
For this choice of $\deltaA$, Item~\ref{prop:proof:consequencesofA:item:quadgrowth} holds automatically.
We now prove the remaining items.

\subsection{Item~\ref{prop:proof:consequencesofA:item:smoothproj}: Smoothness of $P_{\cM}$.} 

Fix $x' \in B_{2\deltaA}(\bar x)$ and $x\in B_{\deltaA}(x)$. 
Observe that $P_{\cM}(x) \in B_{2\deltaA}(\bar x)$ and we have the inclusion $x - P_{\cM}(x) \in \normal(P_{\cM}(x))$. 
Consequently, we have 
\begin{enumerate}
\item $P_{\tangent(P_{\cM}(x))}(x) = P_{\tangent(P_{\cM}(x))}(P_{\cM}(x))$;
\item $P_{\cM}(x) = P_{\cM}(P_\cM(x))$;
\item  $\nabla P_{\cM}(P_{\cM}(x)) = P_{\tangent(P_{\cM}(x))}$.
\end{enumerate}
Therefore, we have
\begin{align*}
&\|P_{\cM}(x') - P_{\cM}(x) - P_{\tangent(P_{\cM}(x))}(x' - x)\| \\
&= \|P_{\cM}(x') - P_{\cM}(P_{\cM}(x)) - \nabla P_{\cM}(P_{\cM}(x))(x' - P_{\cM}(x))\| \\
&\leq \frac{\shape}{2}\|x' -  P_{\cM}(x)\|^2\\
&\leq \shape(\|x' -x\|^2 + \dist^2(x, \cM)),
\end{align*}
where the first inequality follows from Lipschitz continuity of $\nabla P_{\cM}$ on $B_{2\deltaA}(\bar x) \subseteq B_{\delta}(\bar x)$.

\subsection{Item~\ref{prop:proof:consequencesofA:item:smoothfunction}: Bounds on $\nabla_{\cM} f$}\label{appendix:smoothnesslowerbound}

Recall that $P_{\cM}(x) \in B_{2\deltaA}(\bar x)$ whenever $x \in B_{\deltaA}(\bar x)$. 
Thus, below we prove that 
$$
\frac{\scc}{2}\|y - \bar x\| \leq \|\nabla f_\cM(y)\| \leq \beta\|y - \bar x\| \qquad \text{ for all $y \in B_{2\deltaA}(\bar x)\cap \cM$.}
$$
This is equivalent to the claimed bound since $\nabla f_{\cM}(y) = \nabla_{\cM} f(y)$ for all $y \in B_{2\deltaA}(\bar x)\cap \cM$.

Let us first prove the claimed upper bound. Due to the inequality, 
$$
f_{\cM}(x) - f_\cM(\bar x) \geq \frac{\scc}{2}\|P_{\cM}(x) - \bar x\|^2 \qquad \text{ for all $x \in B_{\delta}(\bar x)$},
$$
it follows that $\bar x$ is a local minimizer of $f_{\cM}$. Consequently, $\nabla f_{\cM}(\bar x) = 0$.
Thus, since $\beta$ is a local Lipschitz constant of $\nabla f_{\cM}$ on $B_{\delta}(\bar x)$, we have
$$
\|\nabla f_{\cM}(y)\| \leq \beta\|y - \bar x\| \qquad \text{ for all $y \in B_{\delta}(\bar x) \cap \cM$.}
$$
Since $2\deltaA \leq \delta$, this proves the claimed upper bound.

Next, we prove the claimed lower bound. It suffices to establish the following convexity inequality: 
\begin{align}\label{eq:convexitylb}
f_\cM(y) + \dotp{\nabla f_{\cM}(y),  \bar x -y} \leq f_\cM(\bar x) \qquad \text{ for all $y \in B_{2\deltaA}(\bar x)\cap  \cM $.}
\end{align}
Indeed, if this inequality holds, we have 
$$
\dotp{\nabla f_{\cM}(y), y - \bar x} \geq f_\cM(y) - f_\cM(\bar x) \geq \frac{\scc}{2}\|y - \bar x\|^2 \qquad \text{for all $y \in  B_{2\deltaA}(\bar x)\cap  \cM $}, 
$$
and the desired result follows from Cauchy-Schwarz.

To that end, observe that 
since $\nabla f_{\cM}(\bar x) = 0$ and $\nabla^2f_{\cM}$ is $\rho$-Lipschitz in $B_{2\deltaA}(\bar x)$, we have 
$$
f_{\cM}(y) \leq f_{\cM}(\bar x) + \frac{1}{2}\dotp{\nabla^2 f_{\cM}(\bar x)(y - \bar x), y - \bar x} +\frac{\rho}{6}\|y - \bar x\|^3 \qquad \text{for all $y\in B_{2\deltaA}(\bar x) $}.
$$
Consequently, we have the lower bound on the quadratic form: for all $y \in B_{2\deltaA}(\bar x)\cap \cM$, we have
\begin{align}\label{eq:hessianquad}
 \frac{1}{2}\dotp{\nabla^2 f_{\cM}(\bar x)(y - \bar x), (y - \bar x)} &\geq f_{\cM}(y) - f_{\cM}(\bar x) - \frac{\rho}{6}\|y - \bar x\|^3 \notag \\
 &\geq \frac{\scc}{2}\|y - \bar x\|^2  - \frac{\rho}{6}\|y - \bar x\|^3 \notag\\
 &\geq \frac{3\scc}{8}\|y - \bar x\|^2, 
\end{align}
where the second inequality follows from the quadratic growth bound and the third follows from the bound $\|y - \bar x\| \leq 2\deltaA \leq \frac{3\scc}{4\rho}.$
Therefore, for all $y \in \cM \cap B_{2\deltaA}(\bar x)$, we have
\begin{align*}
f_\cM(\bar x) &\geq f_\cM(y) + \dotp{\nabla f_{\cM}(y), \bar x - y} + \frac{1}{2}\dotp{\nabla^2 f_{\cM}(y)(\bar x - y),(\bar x - y)}  - \frac{\rho}{6} \|y - \bar x\|^3\\
&\geq f_\cM(y)+ \dotp{\nabla f_{\cM}(y), \bar x - y} + \frac{1}{2}\dotp{\nabla^2 f_{\cM}(\bar x)(\bar x - y),(\bar x - y)}  - \frac{2\rho}{3} \|y - \bar x\|^3\\
&\geq f_\cM(y)+ \dotp{\nabla f_{\cM}(y), \bar x - y} +  \frac{3\scc}{8}\|y - \bar x\|^2  - \frac{2\rho}{3} \|y - \bar x\|^3 \\
&\geq f_\cM(y)+ \dotp{\nabla f_{\cM}(y), \bar x - y}, 
\end{align*}
where the first and second inequalities follow by Lipschitz continuity of $\nabla^2f_{\cM}$; the third inequality follows from~\eqref{eq:hessianquad}; and the fourth inequality follows from the bound \break $\|y - \bar x\| \leq 2\deltaA \leq \frac{9\scc}{16\rho}.$ This completes the proof.

\subsection{Item~\ref{prop:proof:consequencesofA:item:stronga}: Consequences of strong $(a)$-regularity}

Fix $x \in B_{\deltaA}(\bar x)$ and $\sigma \leq \deltaA$.
Recall that $y: = P_{\cM}(x) \in B_{2\deltaA}(\bar x)$ since $x \in B_{\deltaA}(\bar x)$. 
Fix $g \in \partial_{\sigma} f(x) $. 
By definition of $\partial_{\sigma} f(x)$, there exists a family of coefficients $\lambda_i \in [0, 1]$, points $x_i \in \overline B_{\sigma}(x) \subseteq \overline B_{\delta}(\bar x)$, and subgradients $g_i \in \partial f(x_i)$ indexed by a finite set $i \in I$ such that $\sum_{i\in I} \lambda_i = 1$ and $g =  \sum_{i\in I} \lambda_i g_i$.
Therefore, by averaging the strong $(a)$ bound~\eqref{eq:strongaappendix} over $g_i$, we find that 
\begin{align*}%\label{eq:strongAaveragehelp}
\|P_{T_{\cM}(y)} (g - \nabla_\cM f(y))\|  &\leq\sum_{i\in I} \lambda_i \|P_{T_{\cM}(y)} (g_i - \nabla_\cM f(y))\| \\
&\leq \sum_{i\in I}\lambda_i \ver\|x_i - y\|.\notag\\
&\leq \ver(\dist(x, \cM) + \sigma).
\end{align*}
Since $g$ was arbitrary, it follows that for all $x \in B_{\deltaA}(\bar x)$ and $\sigma \leq \deltaA$, we have 
\begin{align}\label{eq:strongAaveragehelp}
\sup_{g \in \partial_{\sigma} f(x)} \|P_{T_{\cM}(y)} (g - \nabla_\cM f(y))\| \leq \ver(\dist(x, \cM) + \sigma).
\end{align}
Now we apply this bound to establish the two remaining inequalities.

Indeed, first observe that for all $x \in B_{\deltaA}(\bar x)$ and $\sigma \leq \deltaA$, we have
$$
\sup_{g \in \partial_{\sigma} f(x)} \|P_{T_{\cM}(y)} g\| \leq \|\nabla_{\cM} f(y)\| + \ver(\dist(x, \cM) + \sigma) \leq  \beta\|y - \bar x\| + \ver(\dist(x, \cM) + \sigma),
$$
where the first inequality follows from~\eqref{eq:strongAaveragehelp} and the second inequality follows from Item~\ref{prop:proof:consequencesofA:item:smoothfunction}.
This proves the first claimed bound. 
Second, observe that for all $x \in B_{\deltaA}(\bar x)$ and $\sigma \leq \deltaA$, we have
\begin{align*}
\sup_{g, g' \in \partial_{\sigma} f(x)}\|P_{T_{\cM}(y)} (g - g')\| &\leq  \sup_{g \in \partial_{\sigma} f(x)}\|P_{T_{\cM}(y)} (g -  \nabla_\cM f(y))\| +  \sup_{g' \in \partial_{\sigma} f(x)}\|P_{T_{\cM}(y)} (g' - \nabla_\cM f(y))\| \\
&\leq 2 \ver(\dist(x, \cM) + \sigma).
\end{align*}
where the second inequality follows from~\eqref{eq:strongAaveragehelp}. This completes the proof.

\subsection{Item~\ref{prop:proof:consequencesofA:item:sharpaim}: Aiming inequality}

Consider a point $x \in B_{\deltaA}(\bar x)$, let $\kappa = 2\sharpc$, and define
$$
\hat x \in \argmin_{x' \in \closedball_{2\deltaA}(\bar x)} \left\{f(x') + \kappa\|x' - x\|\right\}.
$$
We claim that $\hat x \in \cM \cap B_{2\deltaA}(\bar x)$. 
Indeed, first note that by definition of $\hat x$ and the inclusion $\hat x \in \closedball_{2\deltaA}(\bar x)$, we have
$$
\|\hat x -  x\| \leq \frac{f(\bar x) - f(\hat x)}{\kappa} + \|\bar x - x\| \leq  \|\bar x - x\| < \deltaA, 
$$
where the second inequality follows since $\bar x$ is a minimizer of $f$ on $B_{2\deltaA}(\bar x)$, a consequence of quadratic growth.
Thus, by the triangle inequality, we have $\hat x \in B_{2\deltaA}(\bar x)$. 
By Fermat's rule, we, therefore, have the inclusion: 
$$
0 \in \partial (f + \kappa\|\cdot - x\|)(\hat x) \subseteq \partial f(\hat x) + \kappa \closedball.
$$
If $\hat x \notin \cM$, then $\dist(0, \partial f(\hat x)) > \kappa$, contradicting the above inclusion. 
Therefore, we have $\hat x \in \cM \cap B_{2\deltaA}(\bar x)$.

Turning to the aiming inequality, apply the $(b_{\leq})$-regularity bound~\eqref{eq:bregappendix} to $\hat x$:
\begin{align*}
f(\hat x) \geq f(x) + \dotp{v, \hat x - x} - \varepsilon\|x -\hat x\|
\geq f(\hat x) + \dotp{v,\hat x- x} + (\kappa - \varepsilon)\|x - \hat x\|, 
\end{align*}
where we define $\varepsilon := \sharpc/2$.
Consequently, we have 
\begin{align}\label{eq:aimingprelim}
\dotp{v, x - P_{\cM}(x)} &\geq (\kappa - \varepsilon)\|x - \hat x\| + \dotp{v, \hat x - P_{\cM}(x)} \qquad \text{for all $v \in \partial f(x)$.}
\end{align}
We now bound the term $\dotp{v, \hat x - P_{\cM}(x)}$: By the conclusion of Item~\ref{prop:proof:consequencesofA:item:smoothproj}, we have
\begin{align*}
\|P_{\cM}(\hat x) - P_{\cM}(x) - P_{\tangent(P_{\cM}(x))}(\hat x - x)\| \leq \shape (\|x - \hat x\|^2 + \dist^2(x, \cM)) \leq 2\shape \|x - \hat x\|^2,
\end{align*}
where the second inequality follows since $\hat x \in \cM$.
Thus, we have 
\begin{align*}
|\dotp{v, \hat x - P_{\cM}(x)}|  &\le  |\dotp{v, P_{\tangent(P_{\cM}(x))}(\hat x - x)}| + 2\shape\|v\|  \|x - \hat x\|^2\\
&\leq \|P_{\tangent(P_{\cM}(x))}v\|\|\hat x - x\| + 2\shape L \|x - \hat x\|^2\\
&\leq (\ver\dist(x, \cM) + \beta \|P_{\cM}(x) - \bar x\|)\|\hat x - x\| + 2\shape L \|x - \hat x\|^2 \\
&\leq (\ver\deltaA + 2\beta\deltaA  + 2\shape L \deltaA)\|\hat x - x\|\\
&\leq \varepsilon \|\hat x - x\|.
\end{align*}
where the second inequality follows from Item~\ref{prop:proof:consequencesofA:item:stronga} 
and the third inequality follows from the inclusion $P_{\cM}(x) \in B_{2\deltaA}( \bar x)$.
Therefore, plugging this bound into~\eqref{eq:aimingprelim}, we arrive at 
$$
\dotp{v, x - P_{\cM}(x)} \geq (\kappa - 2\varepsilon)\|x - \hat x\| \geq  \sharpc\dist(x, \cM), 
$$
as desired.

\subsection{Item~\ref{prop:proof:consequencesofA:item:Lipschitzbound}: Bounding subgradients}
Fix $x \in B_{\deltaA}(\bar x)$,  $\sigma \leq \deltaA$, and $g \in \partial_{\sigma} f(x) $. 
By definition of $\partial_{\sigma} f(x)$, there exists a family of coefficients $\lambda_i \in [0, 1]$, points $x_i \in \overline B_{\sigma}(x) \subseteq \overline B_{\delta}(\bar x)$, and subgradients $g_i \in \partial f(x_i)$ indexed by a finite set $i \in I$ such that $\sum_{i\in I} \lambda_i = 1$ and $g =  \sum_{i\in I} \lambda_i g_i$.
Recall that by Lipschitz continuity of $f$ on $B_{\delta}(\bar x)$, we have $\|g_i\| \leq L$ for $i\in I$. 
Therefore, 
$$
\|g\| \leq  \sum_{i\in I} \lambda_i\|g_i\| \leq L,
$$
as desired.

\subsection{Item~\ref{prop:proof:consequencesofA:item:gap}: Bounding the function gap}

Fix a point $x \in B_{\deltaA}(\bar x)$ and recall that $P_{\cM}(x) \in B_{2\deltaA}(\bar x)$. 
Then by Lipschitz continuity of $f$ on $B_{\delta}(\bar x)$, we have 
$$
f(x) - f(P_{\cM}(\bar x)) \leq L\dist(x, \cM). 
$$
Next, arguing as in the proof of Item~\ref{prop:proof:consequencesofA:item:smoothfunction}, we find that $\nabla f_{\cM}(\bar x) = 0$. 
Thus, since $\nabla f_{\cM}$ is $\beta$-Lipschitz on $B_{\delta}(\bar x)$, we have
$$
f(P_{\cM}(x)) - f(\bar x) = f_{\cM}(P_{\cM}(x)) - f(\bar x) \leq \dotp{\nabla f_{\cM}(\bar x), P_{\cM}(x) - \bar x} + \frac{\beta}{2} \|P_{\cM}(x)- \bar x\|^2 = \frac{\beta}{2} \|P_{\cM}(x)- \bar x\|^2.
$$
By putting both bounds together, we have 
\begin{align*}
f(x)- f(\bar x) = f(x) - f(P_{\cM}(x)) + f(P_{\cM}(x)) - f(\bar x) \leq L\dist(x, \cM) + \frac{\beta}{2}\|P_{\cM}(x) - \bar x\|^2, 
\end{align*}
as desired.

\section{Proof of Corollary~\ref{cor:clarkestationarity}} \label{app:cor:clarkestationarity}

{\color{blue}
We begin with the following known Lemma, which immediately follows from~\cite[Proposition 2.8]{goldstein1977optimization}
\begin{lem}\label{lem: clarkecritical}
	Let $f \colon \RR^d \rightarrow \RR$ be a locally Lipschitz function. Suppose that there exists sequences $x_k \rightarrow \bar x$, $\tau_k \rightarrow 0$, and $g_k \in \partial_{\tau_k}f(x_k)$ with $\|g_k\| \rightarrow 0$. Then $\bar x$ is a Clarke critical point.
\end{lem}

Now we turn to the proof of the Corollary.  Since $f$ has bounded initial sublevel set, the following widened sublevel set is bounded:
$$
S: = \{x+ u \colon f(x) \leq f(x_0) \text{ and } u \in \overline B(x)\}.
$$
Thus, there exists $L > 0$ such that $f$ is $L$-Lipschitz on $S$. In addition, $\partial f$ is uniformly bounded by $L$ on $\mathrm{int} \; S$.

We begin with a claim.
	\begin{claim}
		Fix $i > 0$ and define $\tau_i := 2^{-i}$. Let ${\color{blue}s_k :=  \max\{\|g_k\|, \sscale\|g_0\|\}} $ be the trust region parameter used in Algorithm~\ref{alg:linesearch} and define {\color{blue}$\eps_{i, k} :=\sqrt{128}L\tau_i$}. Then  with probability one, the event 
		$$E_k^{(i)} = \left\{\dist(0, \partial_{\tau_i} f(x_k)) > \eps_{i, k} \text{ and }f(x_{k+1}) > f(x_k) - \frac{\tau_i\dist(0,\partial_{\tau_i} f(x_k))}{8}\right\}$$ cannot happen infinitely often, i.e.,
		$$
		P\left(\cap_{T = 1}^{\infty} \cup_{k=T}^{\infty} E_k^{(i)} \right) = 0.
		$$
	\end{claim}
\begin{claimproof}
	We prove that $P(E_k^{(i)})$ is summable in $k$. Indeed, first, note that $P(E_k^{(i)}) = 0$ when $P(\dist(0, \partial_{\tau_i} f(x_k)) > \eps_{i, k}) = 0$. On the other hand, suppose $P(\dist(0, \partial_{\tau_i} f(x_k)) > \eps_{i, k}) > 0$. 
	Now we upper bound $P(E_k^{(i)})$ for all $G_k$ satisfying $G_k \geq i$.
	For such $G: = G_k$,  the radius $\tau_i = \sigma_{G - i}$ is among those considered in Algorithm~\ref{alg:linesearch}. 
	Moreover, since $s_k \le L$ (recall $x_k \in \mathrm{int} (S)$), the radius satisfies the trust region constraint: $\sigma_{G-i} = \tau_i \leq \eps_{i, k}/s_k \leq \dist(0, \partial_{\sigma_{G-i}} f(x))/s_k$.
Therefore, if $\ngoldstein$ terminates with descent at the $(G - i)$-th level in Algorithm~\ref{alg:linesearch}, it follows that 
$$
f(x_{k+1}) > f(x_k) - \frac{\tau_i\dist(0,\partial_{\tau_i} f(x_k))}{8}.
$$
We estimate the probability of this success with Lemma~\ref{lemma:normalgold}: there exist $C > 0$ depending on $\epsilon_{i,k}$ and for all $k \geq i$, we have
	\begin{align*}
		P(E_k^{(i)}) &\le P\left(f(x_{k+1}) > f(x_k) - \frac{\tau_i\dist(0,\partial_{\tau_i} f(x_k))}{8} \middle \vert \dist(0, \partial_{\tau_i} f(x_k)) > \eps_{i, k} \right)\\
		&\le \exp( - Ck).
	\end{align*}
	Therefore, $P(E_k^{(i)})$ is summable in $k$. The result then follows from Borel–Cantelli lemma.
\end{claimproof}

By the claim and a union bound, we know that with probability one, for any fixed $i$, $E_k^{(i)}$ cannot happen infinitely often. Now, suppose that a subsequence $\{x_{k_l}\}$ (where $k_l \geq l$ is strictly increasing in $l$) converges to a point $\bar x$. 
We note that the sequence $\{f(x_k)\}$ is bounded below: Indeed, since $x_{k_l}$ converges and $f$ is continuous, it follows $\{f(x_{k_l})\}$ is bounded below by a constant $c \in \RR$. Consequently, since $\{f(x_{k}))\}$ is nonincreasing and $k_l \geq l$, it follows that $c \leq f(x_{k_l}) \leq f(x_l)$ and for every $l > 0$, as desired. As a result, the following inequalities cannot be valid simultaneously infinitely often:
$$
\dist(0, \partial_{\tau_i} f(x_{k_l})) > \eps_{i, k} \text{ and }f(x_{k_l+1}) \le  f(x_{k_l}) - \frac{\tau_i \dist(0,\partial_{\tau_i} f(x_{k_l}))}{8}.
$$
Therefore, $\dist(0, \partial_{\tau_i} f(x_{k_l})) > \eps_{i, k}$ cannot happen infinitely often. Consequently, we can find a sequence of increasing indices $j_{i}$ such that
$$
\dist(0, \partial_{\tau_i} f(x_{j_i})) \le \eps_{i, k} \qquad \text{and } x_{j_{i}} \rightarrow \bar x.
$$
Since $\eps_{i, k} \rightarrow 0$ as $k \rightarrow \infty$, Lemma~\ref{lem: clarkecritical}, shows that $\bar x$ is Clarke critical. 
}

\section{Proof of Lemma~\ref{lem: approximatereflection2}}\label{appendix:lem: approximatereflection2}

We begin with preliminary notation and bounds. 
We fix $x \in B_{\deltagrid}(\bar x)$ and subgradient $g \in \partial_{\sigma} f(x) \backslash \{0\}$.
We define $y := P_{\cM}(x)$, $T := \tangent(y)$, and $N := \normal(y)$.
We have the following two bounds: First, we have
\begin{align}\label{eq:deltragridbound1}
(\mu+L)\shape (D_1\dist(x, \cM) +\sigma) \leq (\mu+L)\shape \deltagrid(D_1 +1) = \frac{\mu}{8}\shape(\ceight +1)\deltagrid \leq \textcolor{blue}{\frac{\sharpc}{8}}.
\end{align}
Second, we have 
\begin{align}\label{eq:deltragridbound2}
\ver(\dist(x, \cM) + \sigma) + \beta \|y - \bar x\| & \leq  2\ver\deltagrid + 2\beta\deltagrid \leq \frac{\sharpc}{4}.
\end{align}
We now turn to the proof. 

\textcolor{blue}{By Lemma~\ref{claim: approximatereflectionintermediate_lemma} (which is applicable since $x\in B_{\deltaA/2}(\bar x)$ and $\sigma \leq \deltagrid \leq \deltaA/2$), we have} 
\begin{align*}
\dotp{\hat g, \sigma \frac{P_{N}g}{\|g\|}} &\leq -\sigma\mu  \frac{\|P_{N} g\|}{\|g\|} +(\mu + L) \dist(x, \cM) + (\mu+L)\shape (\dist^2(x, \cM) +\sigma^2).
\end{align*}
Rearranging, we find that 
\begin{align*}
\dotp{P_N\hat g, g} &\leq -\sharpc\|P_{N} g\|  + \frac{(\mu + L)\|g\| \dist(x, \cM)}{\sigma}+ \frac{(\mu+L)\|g\|\shape (\dist^2(x, \cM) +\sigma^2)}{\sigma} \\
&\leq -\sharpc\|P_{N} g\|  + \frac{\sharpc}{8}\|g\|+ (\mu+L)\shape (D_1\dist(x, \cM) +\sigma) \cdot \|g\| \\
&\leq -\sharpc\|P_{N} g\|  + \frac{\sharpc}{4}\|g\|,%\\
 \end{align*}
 where the second inequality \textcolor{blue}{follows from the assumption  $\ceight \dist(x,\cM) \le \sigma$ } and the third follows from~\eqref{eq:deltragridbound1}. 
 Now observe that
 $$
 \dotp{P_T \hat g, g} \leq \|P_T\hat g\|\|g\| \leq  (\ver(\dist(x, \cM) + \sigma) + \beta \|y - \bar x\|)\cdot \|g\| \leq \frac{\mu}{4} \|g\|,
 $$
where second inequality follows from~\eqref{eq:consequencestronga} and the third inequality follows from~\eqref{eq:deltragridbound2}.  
Therefore,
 \begin{align*}
\dotp{\hat g, g} = \dotp{P_N \hat g, g} + \dotp{P_T\hat g, g} \leq -\sharpc\|P_{N} g\|  + \frac{\sharpc}{2}\|g\| \leq - \frac{\sharpc}{2}\|g\| + \sharpc\|P_T(g)\|,
 \end{align*}
 as desired.

{\color{blue}\section{Proof that $\mu \leq L$}\label{sec:mulessthanL}
\begin{lem}\label{lem:condition_number_larger}
We have that $\mu \leq L$.
\end{lem}
\begin{proof}
Indeed,
$$
\mu =  \frac{1}{4}\displaystyle\liminf_{\substack{x' \stackrel{\cM^c}{\rightarrow} \bar x }}\dist(0, \partial f(x)) \leq \limsup_{x \rightarrow \bar x} \dist(0, \partial f(x)) \leq L.
$$
by Proposition~\ref{prop:proof:consequencesofA}. 
\end{proof}}

	%%%%%%%%%%%%%%%%%

	\section{Proof of Lemma~\ref{sec:getinneighborhood}}\label{sec:proof:sec:getinneighborhood}

We fix $a > 0$. 
Note that the claimed inclusion is a consequence of the following bound:
\begin{align}\label{eq:convexcaseequivalence}
f(x) - f(\bar x) \geq \frac{\scc}{2}\min\{\delta_A, \|x - \bar x\|\} \|x - \bar x\| \qquad \text{ for all $x \in \RR^d$.}
\end{align}
Here we provide a proof for completeness.

To that end, we remind the reader that Assumption~\ref{assumption:mainfinal} is in force. Consequently, by Item~\ref{prop:proof:consequencesofA:item:quadgrowth} of Proposition~\ref{prop:proof:consequencesofA}, we have:
$$
f(x) - f(\bar x) \geq \frac{\scc}{2}\|x - \bar x\|^2 \qquad \text{for all $x \in \overline B_{\delta_A}(\bar x)$.}
$$
Thus, if $x \in B_{\delta_A}(\bar x)$, bound~\eqref{eq:convexcaseequivalence} is immediate.
On the other hand, suppose that we have $x \in \RR^d \backslash B_{\delta_A}(\bar x)$. Define the curve $x_t \colon t \mapsto (1-t)x + t\bar x$. 
Choose $t_0 \in [0, 1]$ such that $x_{t_0} \in \bdry B_{\delta_A}(\bar x)$.  Then by Jensen's inequality, we have
$$
(1-t_0) f(x)\geq f(x_{t_0}) - t_0f(\bar x) \geq (1-{t_0})f(\bar x) + \frac{\scc}{2}\|x_{t_0} - \bar x\|^2 = (1-{t_0})f(\bar x) + \frac{\scc(1-t_0)}{2}\|x - \bar x\|\|x_{t_0} - \bar x\|.
$$
Consequently, since $\|x_{t_0} - \bar x\| = \delta_A$, we have
$$
f(x) - f(\bar x) \geq \frac{\scc\delta_A}{2}\|x - \bar x\|\geq \frac{\scc}{2}\min\{\delta_A, \|x - \bar x\|\} \|x - \bar x\|, 
$$
as desired. This completes the proof. 

\section{Proof of~\eqref{eq:finalneededbound}}\label{sec:finalneededbound}

Let us expand the left-hand-side of~\eqref{eq:finalneededbound}:
\begin{align*}
\frac{16L'\sqrt{2 \log(2K_1^2/p)}}{K_1^{1/2}} &\leq \underbrace{\frac{16{\color{blue}D}L'\sqrt{2 \log(K_1^2)}}{K_1^{1/2}}}_{=:A} + \underbrace{\frac{16{\color{blue}D}L'\sqrt{2 \log(2/p)}}{K_1^{1/2}}}_{=:B}.
\end{align*}
Note that $B \leq a/4$ by definition of $K_1$. Consequently, the proof will follow if $A\leq a/4.$
To that end, for any $\alpha \in (0, 1)$, we have
\begin{align*}
A = \frac{16{\color{blue}D}L'\sqrt{2 \log(K_1^2)}}{K_1^{1/2}} &= \frac{16{\color{blue}D}L'\sqrt{2 \log(K_1^{2\alpha})/\alpha}}{K_1^{1/2}} \leq\frac{16{\color{blue}D}L'\sqrt{2/\alpha}}{K_1^{(1-\alpha)/2}}, 
\end{align*}
Therefore, we have $A \leq a/4$ whenever
$$
 K_1 \geq \inf_{\alpha \in (0, 1)}\frac{\left(64{\color{blue}D}L'\sqrt{\frac{2}{\alpha}}\right)^{\frac{2}{(1-\alpha)}}}{a^{\frac{2}{(1-\alpha)}}}  = \frac{{\color{blue}D^2}}{a^2}\inf_{\alpha \in (0, 1)}\frac{\left(64L'\sqrt{\frac{2}{\alpha}}\right)^{\frac{2}{(1-\alpha)}}}{(\frac{a}{{\color{blue}D}})^{\frac{2\alpha}{(1-\alpha)}}}   =  \frac{{\color{blue}D^2}}{a^2}b.
$$
This lower bound holds by definition of $K_1$. Consequently $A \leq a/4$. Therefore, the proof is complete.

{\color{blue}
\section{Proof of Lemma~\ref{lem:contractionlowerbound}}\label{sec:contractionlowerbound}
Throughout this section, we use the symbol $a \lesssim b$ to mean that $a \leq \eta b$ for a fixed numerical constant $\eta$ that is independent of $f$.
In addition, we use the bound on the condition number:
$
\kappa \geq 1, 
$
since $\mu \leq L$; see Lemma~\ref{sec:mulessthanL}.

Turning to the bound, we wish to upper bound $q$.
$$
q= \max\left\{\rho, \sqrt{1-\frac{3\sharpc^2}{256\lipf^2}}, \frac{1}{2}\right\}.
$$
First note that 
$$
1- \sqrt{1-\frac{3\sharpc^2}{256\lipf^2}} \gtrsim  \frac{\mu^2}{L^2} \geq  \frac{1}{\kappa^2}.
$$
Next, we upper bound $\rho$. To that end, we must bound the constants $a_1$ and $a_2$, which rely on the somewhat involved constants $\csix$ and $\cseven$.
Thus, we first lower bound $\csix$:
\begin{align*}
\csix &= \min\left\{\frac{\beta }{\ver(1+\deltaA)}, \frac{{\color{blue}\min\left\{\sharpc/\deltaA, \cfour\cfive/\beta\right\}}}{4(1 +  (1+\deltaA)\shape)  (\mu + L))}, {\color{blue}\frac{1}{2}}\right\}\\
&\gtrsim \min \left\{ \frac{\beta}{\ver},\frac{\mu}{L(1+\shape)}, \frac{\gamma^2\mu}{L^2\beta(1+\shape)}\right\} \\
&\geq \frac{1}{\kappa^3(1+\shape)},
\end{align*}
where we use the bounds $\mu \leq L$, $\cfour \gtrsim \gamma^2/L$, and $D_2 \gtrsim \sharpc$. Turning to $C_5$, we have:
\begin{align*}
\cseven &= \min \left\{\frac{\beta}{2\ver}, \frac{\cfour\cfive}{32\ver\beta},\csix, \frac{\ctwo}{4}\right\}\\
&\gtrsim \min \left\{  \frac{\beta}{\ver}, \frac{\gamma^2\mu}{L\ver \beta}, \frac{1}{\kappa^3(1+\shape)}, \frac{\gamma}{\ver}\right\}\\
&\geq \frac{1}{\kappa^3(1+\shape)},
\end{align*}
where we again use $\cfour \gtrsim \gamma^2/L$, and $D_2 \gtrsim \sharpc$.
Therefore, we have the lower bound for $a_2$:
\begin{align*}
a_2 &= \frac{\min\left\{\cone/L, \cseven\right\}}{2} \gtrsim \min \left\{ \frac{\gamma^2}{L^2},\frac{1}{\kappa^3(1+\shape)}\right\} \gtrsim \frac{1}{\kappa^{3}(1+\shape)}.
\end{align*}
In addition, we have the upper bound:
\begin{align*}
a_2 = \frac{\min\left\{\cone/L, \cseven\right\}}{2} \leq \csix/2 \leq 1/4.
\end{align*}
Finally to lower bound $a_1$, we have
\begin{align*}
a_1 &= \min\{\done , \dtwo/L\} \gtrsim \min\left\{\frac{\mu}{L}, \frac{\gamma}{L}\right\} \gtrsim \frac{1}{\kappa}, 
\end{align*}
where we use the bound $\done \gtrsim \mu/L$ and $\dtwo/L \gtrsim \mu/L$.

Now we upper bound $\rho$ by providing a lower bound on $1-\rho$.
\begin{align*}
1-\rho &= \frac{1}{8}\min\left\{\frac{\scc a_2}{8\max\{4La_2^2, \beta\}},  \frac{\sharpc a_1}{4\max\{2L, \beta/a_2^2\}}\right\}\\
&\gtrsim  \min\left\{\frac{\scc}{La_2}, \frac{\gamma a_2}{\beta}, \frac{\mu a_1}{L}, \frac{\mu a_1a_2^2}{\beta} \right\}\\
&\gtrsim  \min\left\{\frac{\scc}{L}, \frac{\gamma}{\kappa^{3}(1+\shape)\beta}, \frac{\mu}{\kappa L}, \frac{\mu}{\beta\kappa^{7}(1+\shape)^2} \right\}\\
&\gtrsim \frac{1}{\kappa^8(1+\shape)^2}
\end{align*}
Putting all these bounds together, we find that: 
$$
1 - q \gtrsim \min \left\{\frac{1}{\kappa^8(1+\shape)^2}, \frac{1}{\kappa^{2}}\right\} \geq \frac{1}{\kappa^8(1+\shape)^2}, 
$$
as desired.}
\end{document}